\documentclass{amsart}
\usepackage{enumitem, subcaption, color, mathtools, mathrsfs, amsmath,amssymb,amsthm,color,mathrsfs, geometry, nicefrac, dutchcal, array, appendix}
\usepackage{physics}
\usepackage{parskip}
\usepackage{tcolorbox}
\usepackage{bbm}
\usepackage[linewidth=1pt]{mdframed}
\usepackage{ulem}

\newgeometry{vmargin={35mm}, hmargin={30mm,30mm}}

\usepackage{hyperref}

\newtheorem{theorem}{Theorem}
\newtheorem{definition}[theorem]{Definition}
\newtheorem{proposition}[theorem]{Proposition}
\newtheorem{lemma}[theorem]{Lemma}

\newtheorem{remark}[theorem]{Remark}
\newtheorem{Notation}[theorem]{Notation}
\newtheorem{Convention}[theorem]{Convention}

\newcommand{\nc}{\newcommand}
\nc{\R}{\mathbb{R}}
\nc{\C}{\mathbb{C}}
\nc{\mrm}{\mathrm}
\nc{\mL}{\mrm{L}}
\nc{\mF}{\mrm{F}}
\nc{\mC}{\mrm{C}}
\nc{\mH}{\mrm{H}}
\nc{\mW}{\mrm{W}}
\nc{\mV}{\mrm{V}}
\nc{\mM}{\mrm{M}}
\nc{\mK}{\mrm{K}}
\nc{\mD}{\mrm{D}}
\nc{\mB}{\mrm{B}}
\nc{\mR}{\mrm{R}}
\nc{\mX}{\mrm{X}}
\nc{\mY}{\mrm{Y}}
\nc{\mS}{\mrm{S}}

\nc{\Ec}{\mrm{E_c}}
\nc{\calL}{\mathcal{L}}
\nc{\loc}{\mrm{loc}}
\nc{\comp}{c}
\nc{\supp}{\mrm{supp}}
\nc{\Hardy}{\mathfrak{H}}
\nc{\calH}{\mathcal{H}}
\nc{\ctru}{\mathfrak{u}}
\nc{\ctrv}{\mathfrak{v}}
\nc{\bc}{\boldsymbol{c}}
\nc{\be}{\boldsymbol{e}}
\nc{\br}{\boldsymbol{r}}
\nc{\bs}{\boldsymbol{s}}
\nc{\bt}{\boldsymbol{t}}
\nc{\bw}{\boldsymbol{w}}
\nc{\bx}{\boldsymbol{x}}
\nc{\by}{\boldsymbol{y}}
\nc{\bz}{\boldsymbol{z}}
\nc{\lbr}{\lbrack}
\nc{\rbr}{\rbrack}
\nc{\dsp}{\displaystyle}
\nc{\vphi}{\varphi}

\begin{document}
	\title[Analysis of the Coupled Cluster Method]{Analysis of the Single Reference Coupled Cluster Method for Electronic Structure Calculations: The Discrete Coupled Cluster Equations}
	
	\author{Muhammad Hassan}\address{(M. Hassan) Technische Universität München, Department of Mathematics, Boltzmannstrasse 3, Garching 85748, Germany}
	\author{Yvon Maday}\address{(Y. Maday) Sorbonne Université, CNRS, Université Paris Cité, Laboratoire Jacques-Louis Lions (LJLL), F-75005 Paris.}

	\begin{abstract} 
		
  Coupled cluster methods are widely regarded as the gold standard of computational quantum chemistry as they are perceived to offer the best compromise between computational cost and a high-accuracy resolution of the ground state eigenvalue of the electronic Hamiltonian-- an unbounded, self-adjoint operator acting on a Hilbert space of antisymmetric functions that describes electronic properties of molecular systems. The present contribution is the second in a series of two articles where we introduce a new numerical analysis of the single-reference coupled cluster method based on the invertibility of coupled cluster Fr\'echet derivative. In this contribution, we study discretisations of the single-reference coupled cluster equations based on a prior mean-field (Hartree-Fock) calculation. We show that under some structural assumptions on the associated discretisation spaces and assuming that the discretisation is fine enough, the discrete coupled cluster equations are locally well-posed, and we derive a priori and residual-based a posteriori error estimates for the discrete coupled cluster solutions. Preliminary numerical experiments indicate that the structural assumptions that we impose for our analysis can be expected to hold for several small molecules and the theoretical constants that appear in our error estimates are an improvement over those obtained from earlier approaches.
	\end{abstract}
	\subjclass{65N25, 65N30, 65Z05,  81V55, 81V70}
	\keywords{Electronic structure theory, coupled cluster method, numerical analysis, non-linear functions, error estimates}
	\maketitle

	\section{Introduction}
	
The electronic structure problem is one of the most important many-body problems in modern computational physics and chemistry. From understanding the emergent superconducting properties of twisted bilayer graphene at certain magic angles \cite{cao2018unconventional} to explaining the complex mechanisms underlying light-harvesting molecules \cite{curutchet2017quantum}, an incredible range of scientific phenomena are governed by many-body electronic interactions. It is therefore hardly surprising that electronic structure calculations have a wide domain of applications from drug discovery and the creation of new compounds for sustainable energy and green catalysis (see, e.g., \cite{deglmann2015application, dieterich2017opinion, hillisch2015computational, lovelock2022road, manly2001impact}), to the design of so-called quantum materials with exotic magnetic, ferroelectric or superconducting properties (see, e.g., \cite{bauer2020quantum, head2020quantum, si2023electronic}).

A central problem in molecular electronic structure is the numerical computation of the ground state eigenvalue of the electronic Hamiltonian-- an unbounded, self-adjoint operator acting on a Hilbert space of antisymmetric functions. Since the spatial dimension of the Hilbert space grows linearly in the number of electrons $N$, the computational cost of a standard Galerkin discretisation of the electronic Hamiltonian typically scales exponentially in the degrees of freedom. State-of-the-art numerical methods in molecular electronic structure theory are therefore based on a \emph{low-rank, non-linear parametrisation} of the sought-after eigenfunction.

 Coupled cluster (CC) methods are one such class of algorithms, which are based on an \emph{exponential ansatz} for the targeted ground state of the electronic Hamiltonian. More precisely, in the so-called single reference CC method, the sought-after ground state is expressed as the action of an exponential cluster operator-- which is the operator exponential of a linear combination of bounded maps (so-called \emph{excitation operators} defined in Section \ref{sec:3a})-- on a judiciously chosen reference function. Using this ansatz, the eigenvalue problem for the ground state energy of the electronic Hamiltonian can be reformulated as a highly \emph{non-linear} system of equations for the unknown coefficients appearing in the linear combination of excitation operators entering the operator exponential. Approximations to the ground state energy are then obtained by restricting the class of excitation operators that appear inside the exponential, which leads to a hierarchy of computationally more tractable non-linear, root-finding problems. Usually these truncations are done on the basis of the excitation orders (see Appendix \ref{sec:spaces} below) and one thus speaks of CCSD (single and double excitation operators), CCSDT (single, double and triple excitation operators) and so on. In particular, the so-called CCSD(T)\footnote{Here, the (T) emphasises the fact that triple excitation orders are not initially included in the CCSD(T) ansatz and are rather treated perturbatively through a post-processing step.} variant, which can be applied to small and medium-sized molecules at a reasonable computational cost, is widely regarded as the `gold standard' of quantum chemistry~\cite{raghavachari1989fifth}.

Despite being the method of choice for computing ground states of dynamically correlated molecules in the quantum chemical community (see, e.g., \cite{lee1995achieving}), the mathematical analysis of coupled cluster methods is a relatively recent phenomenon. Indeed, the numerical analysis of single-reference coupled cluster was initiated about 15 years ago by Schneider and Rohwedder \cite{MR3021693, MR3110488, Schneider_1}, who introduced, for the first time, the correct functional analytic setting for formulating and analysing the coupled cluster equations. These seminal contributions also pioneered the use of the \emph{strong local monotonicity} property, which is essentially a positivity condition on the Fr\'echet derivative of the coupled cluster function in a neighbourhood of the sought-after CC solution $\bt^*$,  as a means to demonstrate the well-posedness of these equations and derive a priori error estimates. This line of reasoning based on strong local monotonicity was subsequently extended to analyse further coupled cluster variants including the so-called extended coupled cluster method \cite{laestadius2018analysis} and the tailored coupled cluster method~\cite{faulstich2019analysis}. 

As we discuss in more detail in Section \ref{sec:3c} below, a rigorous demonstration of the local monotonicity property for coupled cluster functions requires a rather pessimistic smallness assumption on the sought-after coupled cluster solution $\bt^*$. More precisely, in order to establish strong local monotonicity and the local well-posedness analysis and error estimates derived from this approach, we must assume that $\bt^*$ is in a perturbative regime~$\bt^* \approx 0$. On the other hand, there are situations where the CC method is \emph{known numerically} to yield accurate approximations yet the sought-after solution $\bt^*$ is not in the perturbative regime (see, e.g., Table \ref{table2} below). For such problems, the existing a priori analysis yields error estimates with \emph{negative} pre-factors which makes it difficult to consider extensions towards a posteriori error certification and validation. This is particularly troublesome since coupled cluster calculations are often used for benchmark computations and the calibration of Kohn-Sham density functional theory models which are in turn used to create data sets for large-scale machine-learning driven predictive models.

In order to deal with these difficulties, the present authors have introduced a new numerical analysis of the single-reference coupled cluster method based on directly establishing the invertibility of the Fr\'echet derivative of the coupled cluster function using classical inf-sup arguments. Using this approach, we have shown in a recent contribution \cite{Hassan_CC} that, irrespective of the smallness of $\bt^*$, the continuous (infinite-dimensional) CC equations are always locally well-posed provided that the sought-after ground state is non-degenerate and the chosen reference function is non-orthogonal to the targeted ground state wave-function. Preliminary numerical experiments (see, e.g., Table \ref{table2} below) also indicate an improvement in the pre-factors appearing in the error estimates obtained through our new approach. Of course, the main drawback of this approach is that-- in contrast to the strong local monotonicity methodology-- local well-posedness for \emph{discretisations} of the continuous CC equations does not immediately follow from the continuous inf-sup condition. Indeed, this is a well-known phenomenon in the numerical analysis of linear PDEs such as the Helmholtz equation where a discrete inf-sup condition must be established separately for each given discretisation (see, e.g., \cite[Chapter 4]{Schwab}).

The aim of the current contribution is to extend the new analysis developed in \cite{Hassan_CC} to discretisations of the single-reference coupled cluster equations based on a prior mean-field (Hartree-Fock) calculation. We show that under some structural assumptions on the associated discretisation spaces and assuming that the discretisation is fine enough, the discrete coupled cluster equations are locally well-posed, and we derive a priori and residual-based a posteriori error estimates for the discrete coupled cluster solutions. The structural assumptions that we impose for our analysis are one of two types. The first type corresponds to the so-called Full-CC discretisations and the second type to the so-called excitation rank-truncated CC discretisations (see, e.g., \cite[Chapter 13]{helgaker2014molecular} and Appendix \ref{sec:spaces} below) together with a smallness assumption on a specific operator norm involving the difference of the electronic Hamiltonian and the mean-field operator. Preliminary numerical experiments indicate that this smallness assumption can be expected to hold for several small molecules, and the resulting theoretical constants that appear in our error estimates are an improvement over those obtained from earlier approaches.

The remainder of this article is organised as follows. In Section \ref{sec:2}, we introduce the problem formulation, i.e., the electronic Hamiltonian and the Hilbert spaces on which it acts. In Section~\ref{sec:3a} and \ref{sec:3b}, we introduce excitation operators and cluster operators respectively which are fundamental mathematical objects in the coupled cluster methodology. Next, in Section \ref{sec:3c}, we formulate the infinite-dimensional continuous coupled cluster equations and we briefly recall the existing results on the well-posedness of these equations, including our previous contribution \cite{Hassan_CC}. Subsequently, in Section \ref{sec:3d}, we state the discrete coupled cluster equations whose analysis is the main subject of this contribution. We begin our analysis in Section \ref{sec:4} where we first establish, in Subsections \ref{sec:4a} and \ref{sec:4b}, a number of technical lemmas. Finally, in Section \ref{sec:4c}, we state and prove our main result on the local well-posedness and a priori and residual-based a posteriori error estimates for the discrete coupled cluster equations. Supplementary information related to the Hartree-Fock methodology and the practicality of the assumptions needed for our analysis is given in Appendices \ref{sec:HF} and \ref{sec:spaces} respectively.

	\section{Background and Problem Setting}\label{sec:2}

Computational quantum chemistry is the study of the properties of matter through modelling at the molecular scale, i.e., when matter is viewed as a collection of positively charged nuclei and negatively charged electrons. To formalise the problem setting, we assume that we are given a molecule composed of $M \in \mathbb{N}$ nuclei carrying charges $\{Z_{\alpha}\}_{\alpha =1}^M \subset \mathbb{R}_+$ and located at positions $\{\bold{x}_{\alpha}\}_{\alpha =1}^{M} \subset \mathbb{R}^3$ respectively. We further assume the presence of $N\in \mathbb{N}$ electrons whose spatial coordinates are denoted by $\{\bold{x}_i\}_{i=1}^N \subset \mathbb{R}^3$. Throughout this article, we will assume that the Born-Oppenheimer approximation holds, i.e., we will treat the nuclei as fixed, classical particles and we will focus purely on the quantum mechanical description of the electrons.

In order to describe the behaviour of this system of nuclei and electrons under the Born-Oppenheimer approximation, we require the notion of several function spaces. The following construction is largely a repetition of the one found in our previous contribution \cite{Hassan_CC}, which itself was based on \cite{rohwedder2010analysis}.

\subsection{Function Spaces and Norms}\label{sec:2a}~

To begin with, we denote by $\mL^2(\R^3)$ the space of (equivalence classes of) real-valued square integrable functions of three variables, and we denote by $\mH^1(\R^3)$ the subspace of $\mL^2(\R^3)$ consisting of functions that additionally possess square integrable first derivatives. Both spaces are equipped with their usual inner products. Following the convention in the quantum chemical literature, we will frequently refer to $\mL^2(\R^3)$ and $\mH^1(\R^3)$ as infinite-dimensional \emph{single particle} spaces.

Next, we define the tensor space
\begin{align*}
	\mathcal{L}^2:= \bigotimes_{j=1}^N \mL^2(\R^3),
\end{align*}
which is equipped with an inner product that is constructed by defining first for all elementary tensors $ \mathcal{f}, \mathcal{g} \in \mathcal{L}^2$ with  $\mathcal{f}= \otimes_{j=1}^N \mathcal{f}_j$ and $\mathcal{g}= \otimes_{j=1}^N \mathcal{g}_j$
\begin{equation}\label{eq:inner_product}
	\begin{split}
		\left(\mathcal{f}, \mathcal{g}\right)_{\mathcal{L}^2}:= \prod_{j=1}^N \left(\mathcal{f}_j, \mathcal{g}_j\right)_{\mL^2(\R^3)},
	\end{split}
\end{equation}
and then extending bilinearly for general tensorial elements of $\mathcal{L}^2$. 

It is a consequence of Fubini's theorem that the tensor space $\mathcal{L}^2$ is isometrically isomorphic to the space $\mL^2(\R^{3N})$ of real-valued square integrable functions of $3N$ variables with the associated $\mL^2$-inner product. Thanks to this result, we can define the tensor space $\mathcal{H}^1 \subset \mathcal{L}^2 $ as the closure of $\mathscr{C}_0^{\infty}(\mathbb{R}^{3N})$ in $\mL^2(\R^{3N})$ with respect to the usual gradient-gradient inner product on $\mathbb{R}^{3N}$.

In quantum mechanics, a fundamental distinction is made between so-called \emph{bosonic} and \emph{fermionic} particles, the latter obeying the so-called Pauli-exclusion principle and thus being described in terms of antisymmetric functions. We are therefore obligated to also define tensor spaces of antisymmetric functions. To this end, we first introduce the so-called \emph{antisymmetric projection operator} $\mathbb{P}^{\rm as} \colon \mathcal{L}^2\rightarrow\mathcal{L}^2$ that is defined through the action
\begin{align}\label{eq:anti_proj}
	\forall \mathcal{f} \in \mathcal{L}^2\colon \quad (\mathbb{P}^{\rm as} \mathcal{f})(\bold{x}_1, \ldots, \bold{x}_N) :=\frac{1}{{N!}} \sum_{\pi \in {\mS}(N)} {\rm sgn(\pi)} \mathcal{f}(\bold{x}_{\pi(1)}, \ldots,\bold{x}_{\pi(N)}  ),
\end{align}
where ${\mS}(N)$ denotes the permutation group of order $N$, and ${\rm sgn(\pi)} $ denotes the signature of $\pi~\in~{\mS}(N)$. 

It is easy to establish that $\mathbb{P}^{\rm as}$ is an $\mathcal{L}^2$-orthogonal projection with a closed range. We therefore define the antisymmetric tensor spaces $\widehat{\mathcal{L}}^2 \subset \mathcal{L}^2$ and $\widehat{\mathcal{H}}^1 \subset \mathcal{H}^1$ as
\begin{align*}
	\widehat{\mathcal{L}}^2:= \bigwedge_{j=1}^N \mL^2(\R^3) := \text{\rm ran}\; \mathbb{P}^{\rm as} \quad \text{and} \quad \widehat{\mathcal{H}}^1 := \widehat{\mathcal{L}}^2 \cap \mathcal{H}^1,
\end{align*}
equipped with the $(\cdot, \cdot)_{\mathcal{L}^2}$ and $(\cdot, \cdot)_{\mathcal{H}^1}$ inner products respectively. We also remark that normalised elements of $\widehat{\mathcal{L}}^2$ are known as \emph{wave-functions}, and these are antisymmetric in the sense that for any $\mathcal{f} \in	\widehat{\mathcal{L}}^2 $ we have that
\begin{align*}
	\mathcal{f}(\bold{x}_1, \ldots, \bold{x}_i, \ldots, \bold{x}_j, \ldots \bold{x}_N)= - \mathcal{f}(\bold{x}_1, \ldots, \bold{x}_j, \ldots, \bold{x}_i, \ldots \bold{x}_N) \qquad \forall~ i, j \in \{1, \ldots, N\} \text{ with } i\neq j.
\end{align*}

In the sequel, we will also frequently make use of the dual space of $\widehat{\mathcal{H}}^1$. We therefore denote $\widehat{\mathcal{H}}^{-1}:= \big(\widehat{\mathcal{H}}^1\big)^*$, we equip $\widehat{\mathcal{H}}^{-1} $ with the canonical dual norm, and we write $\langle \cdot, \cdot \rangle_{\widehat{\mathcal{H}}^1, \widehat{\mathcal{H}}^{-1}} $ for the associated duality pairing. {It is important to note that $\widehat{\mathcal{H}}^1 \hookrightarrow \widehat{\mathcal{L}}^2 \hookrightarrow \widehat{\mathcal{H}}^{-1}$ form a Gelfland triple (see, e.g, \cite[Section 2.1.2.4]{Schwab}) so that the $\langle \cdot, \cdot \rangle_{\widehat{\mathcal{H}}^1, \widehat{\mathcal{H}}^{-1}} $ duality pairing can be seen as a continuous extension of the $(\cdot, \cdot)_{\widehat{\mathcal{L}}^2}$-inner product. In particular, for all $\Phi \in \widehat{\mathcal{H}}^1$ and all $\Upsilon \in \widehat{\mathcal{L}}^2$ it holds that
\begin{align}\label{eq:Yvon_duality}
    \langle \Phi, \Upsilon \rangle_{\widehat{\mathcal{H}}^1, \widehat{\mathcal{H}}^{-1}} = (\Phi, \Upsilon)_{\widehat{\mathcal{L}}^2}.
\end{align}
This fact will be used frequently in our analysis.}

Finally, let us remark that higher regularity Sobolev spaces  $\widehat{\mathcal{H}}^k, ~k\in \mathbb{N}$ are defined analogously to~$\widehat{\mathcal{H}}^1$. Let us also briefly comment that, as is common for a numerical analysis of the kind that we undertake in this paper, we have ignored spin variables in the construction of our function spaces (see, e.g., Remark 1 in our previous contribution \cite{Hassan_CC} for a discussion of this point.)

\vspace{2mm}
\subsection{Governing Operators and Problem Statement}\label{sec:2b}~

Throughout this article, we assume that the electronic properties of the molecule that we study are described by the action of a many-body electronic Hamiltonian given by
\begin{equation}\label{eq:Hamiltonian}
	H:= -\frac{1}{2} \sum_{j=1}^N \Delta_{\bold{x}_j} + \sum_{j =1}^{N} \sum_{\alpha =1}^{M} \frac{-Z_{\alpha}}{\vert \bold{x}_{\alpha}- \bold{x}_j\vert }  + \sum_{j =1}^{N} \sum_{i =1}^{j-1}  \frac{1}{\vert \bold{x}_i - \bold{x}_j\vert}\qquad \text{acting on } \widehat{\mathcal{L}}^2 \quad \text{with domain } \widehat{\mathcal{H}}^2.
\end{equation}

The electronic properties of the molecule that we study are functions of the spectrum of the electronic Hamiltonian~$H$, and we are therefore interested in its analysis and computation. It is a classical result (see, e.g.,  \cite[Chapter 6]{lewin2018theorie}) that the operator $H$ is self-adjoint on $\widehat{\mathcal{L}}^2$ with domain~$\widehat{\mathcal{H}}^2$ and form domain~$\widehat{\mathcal{H}}^1$. It is also known that the spectrum of $H$ is bounded below and that $H$ has an essential spectrum $\sigma_{\rm ess}$ of the form $\sigma_{\rm ess}:= [\Sigma, \infty)$ where $\Sigma \in (-\infty, 0]$. Additionally, Zhislin and Sigalov \cite{zhislin1960study, zhislin1965spectrum} have shown that for neutral and positively charged systems, i.e., if $Z:= \sum_{\alpha =1}^M Z_{\alpha} > N-1$, the discrete spectrum of $H$ consists of a countably infinite number of eigenvalues, each with finite multiplicity, accumulating at $\Sigma$. In particular, $H$ possesses a lowest eigenvalue $\mathcal{E}^*_{\rm GS} \in \mathbb{R}$, frequently called the ground state energy, such that
	\begin{subequations}
		\begin{align}
			\label{eq:Ground_State}
			\mathcal{E}_{\rm GS}^*= \min_{0\neq \Psi \in \widehat{\mathcal{H}}^1} \frac{\big\langle\Psi,H \Psi\big\rangle_{\widehat{\mathcal{H}}^1 \times \widehat{\mathcal{H}}^{-1}}}{\Vert \Psi\Vert^2_{\widehat{\mathcal{L}}^2}}.
		\end{align}
	Any function $\Psi^*_{\rm GS} \in \widehat{\mathcal{H}}^1$ that achieves the minimum in Equation \eqref{eq:Ground_State} is called a ground state of $H$ and obviously satisfies
	\begin{align}
		H\Psi_{\rm GS}^* = \mathcal{E}^*_{\rm GS} \Psi^*_{\rm GS}.
	\end{align}
	\end{subequations}

	For the purpose of this article, we will assume that $Z=\sum_{\alpha =1}^M Z_{\alpha} > N-1$. Note that if the ground state eigenvalue $\mathcal{E}^*_{\rm GS}$ is simple (which is not always the case), normalised ground states $\Psi_{\rm GS}^*$ (being elements of a real Hilbert space) are unique up to sign. Consequently, when $\mathcal{E}^*_{\rm GS}$ is simple, we will simply refer to \emph{the} ground state $\Psi_{\rm GS}^*$ with the convention that the sign of $\Psi_{\rm GS}^*$ has been fixed once and for all.

From a functional analysis point of view, the electronic Hamiltonian $H$ possesses certain desirable properties, namely continuity and ellipticity on appropriate Sobolev spaces. More precisely (see, for instance, \cite[Chapter 4]{yserentant2010electronic}),
\begin{itemize}
	\item The electronic Hamiltonian defined through Equation \eqref{eq:Hamiltonian} is bounded as a mapping from $\widehat{\mathcal{H}}^1$ to $\widehat{\mathcal{H}}^{-1}$: \vspace{-3mm}
	\begin{align}\label{eq:contin}
		\forall \Phi, \Psi \in \widehat{\mathcal{H}}^1 \colon \qquad \left \vert \left \langle \Phi, H\Psi\right \rangle_{\widehat{\mathcal{H}}^1 \times \widehat{\mathcal{H}}^{-1}} \right \vert \leq \Big(\frac{1}{2}+ 3\sqrt{N}Z\Big)\Vert \Phi\Vert_{\widehat{\mathcal{H}}^1}\Vert \Psi\Vert_{\widehat{\mathcal{H}}^1};
	\end{align}
	\vspace{-3mm}
	\item The electronic Hamiltonian defined through Equation \eqref{eq:Hamiltonian} satisfies the following ellipticity condition on the Gelfand triple $\widehat{\mathcal{H}}^1 \hookrightarrow \widehat{\mathcal{L}}^2 \hookrightarrow \widehat{\mathcal{H}}^{-1}$:
	\begin{align}\label{eq:ellip}
		\forall \Phi \in \widehat{\mathcal{H}}^1 \colon \qquad \left \langle \Phi, H\Phi\right \rangle_{\widehat{\mathcal{H}}^1 \times \widehat{\mathcal{H}}^{-1}}  \geq \frac{1}{4}\Vert \Phi\Vert_{\widehat{\mathcal{H}}^1}^2 - \Big(9NZ^2 +\frac{1}{4}\Big)\Vert \Phi\Vert^2_{\widehat{\mathcal{L}}^2}.
	\end{align}
\end{itemize}

An important consequence of the above ellipticity estimate is that the electronic Hamiltonian, modified by any suitable shift, defines an invertible operator on $\widehat{\mathcal{H}}^1$. This fact will be of great importance in our analysis in Sections \ref{sec:3} and \ref{sec:4}.

\section{Formulation of the Coupled Cluster Equations}\label{sec:3}

\vspace{2mm}

Post-Hartree Fock numerical methods in quantum chemistry such as the single-reference coupled cluster method are based on a particular decomposition of the single-particle space $\mH^1(\R^3)$ and the $N$-particle space $\widehat{\mathcal{H}}^1$. In order to state the single-reference coupled cluster equations, we first describe this decomposition in more detail.  \vspace{2mm}

\begin{Notation}[Occupied and Virtual Spaces]\label{def:occ_vir}~

Let $\mathscr{R}$ denote an $N$-dimensional subspace of the single-particle space $\mH^1(\R^3)$ and let $\mathscr{R}^{\perp} \subset \mL^2(\R^3)$ denote the $\mL^2(\R^3)$-orthogonal complement of $\mathscr{R}$, i.e.,
\begin{equation*}
    \mathscr{R}^{\perp} := \left\{\phi \in \mL^2(\R^3) \colon (\phi, \psi)_{\mL^2(\R^3)}=0 ~\forall \psi \in \mathscr{R}\right\}.
\end{equation*}

We refer to $\mathscr{R}$ as an occupied space of $\mH^1(\R^3)$, and we refer to $\mathscr{R}^{\perp}$ as a virtual space of $\mH^1(\R^3)$ respectively.
\end{Notation}

\begin{definition}[Reference Determinant for an Occupied Space]\label{def:ref_det}~

    Let $\mathbb{P}^{\rm as} \colon \mathcal{L}^2\rightarrow\mathcal{L}^2$ denote the anti-symmetric projection operator defined through Equation \eqref{eq:anti_proj}, let $\mathscr{R}$ denote an occupied space of $\mH^1(\R^3)$ as introduced in Notation \ref{def:occ_vir}, and let $\mathscr{B}_{\rm occ}:= \{\psi_j\}_{j=1}^N \subset \mathscr{R}$ denote any $\mL^2(\R^3)$-orthonormal basis for $\mathscr{R}$. Then we define the function $\Psi_0 \in \widehat{\mathcal{H}}^1$ as
\begin{align*}
 \Psi_0:= \mathbb{P}^{\rm as} \left(\otimes_{i=1}^N \psi_i\right),
\end{align*}
and we say that $\Psi_0$ is a reference determinant corresponding to the occupied space $\mathscr{R}$.
\end{definition}

An important remark is now in order.\vspace{2mm}

\begin{remark}[Uniqueness up to sign of Reference Determinant]\label{rem:ref_det_unique}~
    Consider Definition \ref{def:ref_det} of a reference determinant $\Psi_0 \in \widehat{\mathcal{H}}^1$ constructed from an $\mL^2(\R^3)$-orthonormal basis $\mathscr{B}_{\rm occ}:= \{\psi_j\}_{j=1}^N $ of some occupied space $\mathscr{R}$. Given now a different $\mL^2(\R^3)$-orthonormal basis $\mathscr{B}_{\rm occ}':= \{\psi_j'\}_{j=1}^N $, there obviously exists a unitary matrix $\bold{U}=[\bold{U}]_{ij}\in \mathbb{R}^{N \times N}$ such that
    \begin{align*}
        \forall j \in \{1, \ldots, N\} \colon \qquad \psi_i' := \sum_{j=1}^N [\bold{U}]_{ij}\psi_j.
    \end{align*}
    
It follows from the definition of the anti-symmetric projection operator $\mathbb{P}^{\rm as} $ that
\begin{align*}
    \mathbb{P}^{\rm as} \left(\otimes_{i=1}^N \psi_i'\right)= \text{\rm det}(\bold{U}) \mathbb{P}^{\rm as} \left(\otimes_{i=1}^N \psi_i\right)= \text{\rm det}(\bold{U}) \Psi_0.
\end{align*}
    Consequently, reference determinants corresponding to a given occupied space $\mathscr{R} \subset \mH^1(\R^3)$ are unique up to sign. In the sequel therefore, given an occupied space $\mathscr{R}$ of the single-particle space $\mH^1(\R^3)$,  we will simply refer to \underline{the} reference determinant $\Psi_{0}$ corresponding to $\mathscr{R}$ with the convention that the sign of $\Psi_{0}$ has been fixed once and for all.
\end{remark}

The motivation for introducing reference determinants is that they allow a specific complementary decomposition of the $N$-particle function space $\widehat{\mathcal{H}}^1$.


\begin{definition}[Complementary Decomposition of $N$-particle Function Space]\label{def:decomp}~

    Let $\mathscr{R}$ denote an occupied space of $\mH^1(\R^3)$ as stated in Notation \ref{def:occ_vir}, let $\Psi_0 \in \widehat{\mathcal{H}}^1$ denote the reference determinant corresponding to $\mathscr{R}$ as defined through Definition \ref{def:ref_det}, and let $\{\Psi_0\}^{\perp} \subset \widehat{\mathcal{L}}^2$ denotes the $\widehat{\mathcal{L}}^2$-orthogonal complement of $\Psi_0$. Then we define the infinite-dimensional subspace $\widehat{\mathcal{H}}^{1, \perp}_{\Psi_0} \subset \widehat{\mathcal{H}}^1$ as
\begin{align*}
\widehat{\mathcal{H}}^{1, \perp}_{\Psi_0}:= \{\Psi_0\}^{\perp} \cap \widehat{\mathcal{H}}^1,
\end{align*}
and we introduce the complementary decomposition of $\widehat{\mathcal{H}}^1$ given by
\begin{equation}\label{eq:decomp}
    \widehat{\mathcal{H}}^{1} = \text{\rm span}\{\Psi_0\} \oplus \widehat{\mathcal{H}}^{1, \perp}_{\Psi_0}.
\end{equation}

Additionally we define $\mathbb{P}_0 \colon \widehat{\mathcal{H}}^1 \rightarrow  \widehat{\mathcal{H}}^1$ and $\mathbb{P}_0^{\perp}:={\rm I}-\mathbb{P}_0 \colon \widehat{\mathcal{H}}^1 \rightarrow  \widehat{\mathcal{H}}^1$ as the $\widehat{\mathcal{L}}^2$-orthogonal projection operator onto $ \text{\rm span} \left\{\Psi_0\right\}$ and $\widehat{\mathcal{H}}^{1, \perp}_{\Psi_0}$ respectively. 
\end{definition}

Another important remark is now in order.\vspace{2mm}

\begin{remark}[Properties of Complementary Decomposition of $N$-particle Function Space]~

Consider Definition \ref{def:decomp} of the $\widehat{\mathcal{L}}^2$-orthogonal projection operators $\mathbb{P}_0 \colon \widehat{\mathcal{H}}^1 \rightarrow  \widehat{\mathcal{H}}^1$ and $\mathbb{P}_0^{\perp}\colon \widehat{\mathcal{H}}^1 \rightarrow  \widehat{\mathcal{H}}^1$. 

The fact that $\mathbb{P}_0$ and $\mathbb{P}^{\perp}_0$ are both bounded operators with respect to the $\Vert \cdot \Vert_{\widehat{\mathcal{H}}^1}$ norm is a consequence of the fact that these operators possess a range and a kernel that are both closed in the $\widehat{\mathcal{H}}^1$ topology. Let us also point out that these projection operators allow us to equivalently express the complementary decomposition \eqref{eq:decomp} as
\begin{align*}
    \widehat{\mathcal{H}}^1=  \text{\rm Ran}\mathbb{P}_0 \oplus \widehat{\mathcal{H}}^{1, \perp}_{\Psi_0}= \text{\rm Ran}\mathbb{P}_0\oplus \text{\rm Ran}\mathbb{P}_0^{\perp}.
\end{align*}
\end{remark}

\subsection{Excitation Operators}\label{sec:3a}~

Our next task is to introduce the notion of so-called \emph{excitation} operators which lie at the core of the coupled cluster methodology. As the first step towards doing so, we will introduce suitable orthonormal bases for the $N$-particle function space $\widehat{\mathcal{H}}^1$ and its subspace $\widehat{\mathcal{H}}^{1, \perp}_{\Psi_0}$ introduced above.

\begin{definition}[Slater Determinant Basis for $N$-particle function space $\widehat{\mathcal{H}}^1$]\label{def:Slater_basis}~

    Let $\mathscr{R}$ and $\mathscr{R}^{\perp}$ denote an occupied and virtual space respectively of $\mH^1(\R^3)$ as stated in Notation \ref{def:occ_vir}, let $\mathscr{B}_{\rm occ}:=\{\psi_j\}_{j={1}}^N \subset \mH^1(\R^3)$ and $\mathscr{B}_{\rm vir}:= \{\psi_j\}_{j={N+1}}^\infty \subset \mH^1(\R^3)$ denote $\mL^2(\R^3)$-orthonormal bases for $\mathscr{R}$ and $\mathscr{R}^{\perp}$ respectively, let $\mathcal{J}_{\infty}^N \subset \mathbb{N}^N$ be the index set defined as
\begin{align*}
	\mathcal{J}_{\infty}^N := \Big\{\boldsymbol{\ell}= (\ell_1, \ell_2, \ldots, \ell_N)\in \mathbb{N}^N \colon \ell_1 < \ell_2 < \ldots < \ell_N\Big\},
\end{align*}
 and let the set $\mathcal{B}_{\otimes}^{\rm ord} \subset \mathcal{H}^1$ be defined as
\begin{align*}
	\mathcal{B}_{\otimes}^{\rm ord}:= \Big\{\Psi_{\bold{k}}:= \psi_{k_1} \otimes \psi_{k_2} \otimes \ldots \otimes \psi_{k_N} \colon \bold{k}= (k_1, \ldots, k_N) \in \mathcal{J}_{\infty}^N \Big\}.
\end{align*}   
 
Then we define an $\widehat{\mathcal{L}}^2$-orthonormal basis for $\widehat{\mathcal{H}}^1$ as
\begin{align*}
	\mathcal{B}_{\wedge}:=& \{\sqrt{N!}\; \mathbb{P}^{\rm as} \Psi \colon ~\Psi \in \mathcal{B}_{\otimes}^{\rm ord}\}\\
	=& \left\{ \Psi_{\bold{k}}(\bold{x}_1, \bold{x}_2, \ldots, \bold{x}_N)=\frac{1}{\sqrt{N!}} \sum_{\pi \in {\mS}(N)} {\rm sgn(\pi)} \otimes_{i=1}^N \psi_{k_i}\big(\bold{x}_{\pi(i)}\big) \colon \hspace{1mm} \bold{k}=(k_1, k_2, \ldots, k_N) \in	\mathcal{J}_{\infty}^N\right\},
\end{align*}
and we refer to the elements of this basis set as Slater determinants.
\end{definition}

Two comments are now in order. First, since $\widehat{\mathcal{H}}^1$ is dense in $\widehat{\mathcal{L}}^2$, the set $\mathcal{B}_{\wedge}$ constructed above is also an $\widehat{\mathcal{L}}^2$-orthonormal basis for $\widehat{\mathcal{L}}^2$. Second, to avoid notational complexity, we will frequently use the following notation in the sequel.
\begin{Notation}[Short-hand Notation for Slater Determinants]\label{not:slater}~

    Consider Definition \ref{def:Slater_basis} of the $\widehat{\mathcal{L}}^2$-orthonormal basis set $\mathcal{B}_{\wedge}$. For simplicity, given $\bold{k} \in \mathcal{J}_{\infty}^N$ and a Slater determinant $\Psi_{\bold{k}} \in \mathcal{B}_{\wedge}$ of the~form
\begin{align*}
	\Psi_{\bold{k}}(\bold{x}_1, \bold{x}_2, \ldots, \bold{x}_N)=\frac{1}{\sqrt{N!}} \sum_{\pi \in {\mS}(N)} {\rm sgn(\pi)} \otimes_{i=1}^N \psi_{k_i}\big(\bold{x}_{\pi(i)}\big),
\end{align*}
we will write the Slater determinant $\Psi_{\bold{k}}$ in the succinct form 
\begin{align*}
	\Psi_{\bold{k}}(\bold{x}_1, \bold{x}_2, \ldots, \bold{x}_N)=\frac{1}{\sqrt{N!}}\text{\rm det} \big(\psi_{k_i}(\bold{x}_j)\big)_{i, j=1}^N.
\end{align*}
\end{Notation}

\begin{definition}[Slater Determinant Basis for $\widehat{\mathcal{H}}^{1, \perp}_{\Psi_0}$]\label{rem:Slater}~

    Let $\mathscr{R}$ denote an occupied space of $\mH^1(\R^3)$ as stated in Notation \ref{def:occ_vir}, let $\mathcal{B}_{\wedge}$ denote an $\widehat{\mathcal{L}}^2$-orthonormal basis set for the $N$-particle function space $\widehat{\mathcal{H}}^1$ as defined through Definition~\ref{def:Slater_basis}, let $\Psi_0 \in \widehat{\mathcal{H}}^1$ denote the reference determinant corresponding to $\mathscr{R}$ as defined through Definition \ref{def:ref_det}, and let the infinite-dimensional subspace $\widehat{\mathcal{H}}^{1, \perp}_{\Psi_0} \subset \widehat{\mathcal{H}}^1$ be defined as in Definition \ref{def:decomp}.
    
    Then we define an $\widehat{\mathcal{L}}^2$-orthonormal basis for $\widehat{\mathcal{H}}_{\Psi_0}^{1, \perp}$ as
\begin{align*}
	\widetilde{\mathcal{B}_{\wedge}}:= \mathcal{B}_{\wedge} \setminus \{\Psi_0\}.
\end{align*}
\end{definition}

Equipped with Definitions \ref{def:Slater_basis} and \ref{rem:Slater} for our $N$-particle basis sets, we are now ready to define the so-called excitation operators. To do so, we will first introduce the notion of excitation index sets.

\begin{definition}[Excitation Index Sets]\label{def:Excitation_Index}~
	
	For each $j \in \{1, \ldots, N\}$ we define the index set $\mathcal{I}_j$ as
	\begin{align*}
		\mathcal{I}_j := \left\{ {{a_1, \ldots, a_j}\choose{\ell_1, \ldots, \ell_j}} \colon \ell_1 < \ldots< \ell_j \in \{1, \ldots, N\} \text{ and } a_1< \ldots < a_j  \in \{N+1, N+2, \ldots\} \right\},
	\end{align*}
	and we say that $\mathcal{I}_j$ is the excitation index set of order $j$. Additionally, we define
	\begin{align*}
		\mathcal{I}:= \bigcup_{j=1}^N \mathcal{I}_j,
	\end{align*}
	and we say that $\mathcal{I}$ is the global excitation index set.
\end{definition}

\begin{definition}[Excitation Operators]\label{def:Excitation_Operator}~
	
	Let $j \in \{1, \ldots, N\}$, let $\mu \in \mathcal{I}_j$ be of the form
	\begin{align*}
		\mu={{a_1, \ldots, a_j}\choose{\ell_1, \ldots, \ell_j}} \colon \ell_1 < \ldots< \ell_j \in \{1, \ldots, N\} \text{ and } a_1< \ldots < a_j  \in \{N+1, N+2, \ldots\},
	\end{align*}
and let $\mathcal{B}_{\wedge}$ denote an $\widehat{\mathcal{L}}^2$-orthonormal basis set for the $N$-particle function space $\widehat{\mathcal{H}}^1$ (and hence also for $\widehat{\mathcal{L}}^2$) as defined through Definition \ref{def:Slater_basis}.

We define the excitation operator $\mathcal{X}_{\mu} \colon \widehat{\mathcal{L}}^2 \rightarrow \widehat{\mathcal{L}}^2$ through its action on the $N$-particle basis set $\mathcal{B}_{\wedge}$. For $\Psi_{\nu}(\bold{x}_1, \ldots, \bold{x}_N) =  \frac{1}{\sqrt{N!}} \text{\rm det}\; \big(\psi_{\nu_j}(\bold{x}_i)\big)_{ i, j=1}^N$, we set 
	\begin{align*}
		\mathcal{X}_{\mu} \Psi_{\nu} = \begin{cases}
			0  &\quad \text{ if } \{\ell_1, \ldots, \ell_j\} \not \subset \{\nu_1, \ldots, \nu_N\},\\
			0   & \quad \text{ if } \exists  a_{m} \in  \{a_1, \ldots, a_j\}  \text{ such that } a_m \in \{\nu_1, \ldots, \nu_N\},\\
			\Psi_{\nu, a} \in \mathcal{B}_{\wedge} & \quad \text{ otherwise},
		\end{cases}
	\end{align*}
	where the determinant $\Psi_{\nu, a} $ is constructed from $\Psi_{\nu}$ by replacing all functions $\psi_{\ell_1}, \ldots, \psi_{\ell_j}$ used to construct $\Psi_{\nu}$ with functions $\psi_{a_1},\ldots \psi_{a_j} $  respectively.
\end{definition}

\begin{definition}[De-excitation Operators]\label{def:De-excitation_Operator}~
	
	Let $j \in \{1, \ldots, N\}$, let $\mu \in \mathcal{I}_j$ be of the form
	\begin{align*}
		\mu={{a_1, \ldots, a_j}\choose{\ell_1, \ldots, \ell_j}} \colon \ell_1 < \ldots< \ell_j \in \{1, \ldots, N\} \text{ and } a_1< \ldots < a_j  \in \{N+1, N+2, \ldots\},
	\end{align*}
	and let $\mathcal{B}_{\wedge}$ denote an $\widehat{\mathcal{L}}^2$-orthonormal basis set for the $N$-particle function space $\widehat{\mathcal{H}}^1$ (and hence also for $\widehat{\mathcal{L}}^2$) as defined through Definition \ref{def:Slater_basis}.
 
	We define the de-excitation operator $\mathcal{X}^{\dagger}_{\mu} \colon \widehat{\mathcal{L}}^2 \rightarrow \widehat{\mathcal{L}}^2$ through its action on the $N$-particle basis set $\mathcal{B}_{\wedge}$. For $\Psi_{\nu}(\bold{x}_1, \ldots, \bold{x}_N) =  \frac{1}{\sqrt{N!}} \text{\rm det}\; \big(\psi_{\nu_j}(\bold{x}_i)\big)_{ i, j=1}^N$, we set 
	\begin{align*}
		\mathcal{X}_{\mu}^{\dagger} \Psi_{\nu} = \begin{cases}
			0  &\quad \text{ if } \{a_1, \ldots, a_j\} \not \subset \{\nu_1, \ldots, \nu_N\},\\
			0   & \quad \text{ if } \exists~\ell_{m} \in  \{\ell_1, \ldots, \ell_j\}  \text{ such that } \ell_m \in \{\nu_1, \ldots, \nu_N\},\\
			\Psi_{\nu}^{a} \in \mathcal{B}_{\wedge} & \quad \text{ otherwise},
		\end{cases}
	\end{align*}
	where the determinant $\Psi_{\nu}^{a}$ is constructed from $\Psi_{\nu}$ by replacing all functions $\psi_{a_1},\ldots \psi_{a_j} $ used to construct $\Psi_{\nu}$ with functions $\psi_{\ell_1}, \ldots, \psi_{\ell_j}$ respectively.
\end{definition}

Consider Definitions \ref{def:Excitation_Operator} and \ref{def:De-excitation_Operator} of excitation and de-excitation operators. Two important comments are now in order.

\begin{itemize}
    \item First, de-excitation operators reverse the action of excitation operators in a specific sense: for any excitation index $\mu$ in the global excitation index $\mathcal{I}$ and any determinant $\Psi_\nu \in \mathcal{B}_{\wedge}$, if $\mathcal{X}_{\mu}\Psi_{\nu} \neq 0$, then it holds that $\mathcal{X}_{\mu}^{\dagger}\mathcal{X}_{\mu} \Psi_\nu = \Psi_\nu$.

    \item Second, we emphasise that the definition of both excitation and de-excitation operators depends a priori on the chosen basis of Slater determinants $\mathcal{B}_{\wedge}$ for the $N$-particle space $\widehat{\mathcal{H}}^1$. This Slater determinant basis $\mathcal{B}_{\wedge}$ itself depends on the basis sets $\mathscr{B}_{\rm occ}$ and $\mathscr{B}_{\rm vir}$ for the occupied space $\mathscr{R}$ and the virtual space $\mathscr{R}^{\perp}$ respectively. This dependency will be the subject of further discussion in the sequel.
\end{itemize}

The following theorem summarises some basic properties of excitation and de-excitation operators.

\begin{theorem}[Properties of Excitation and De-excitation Operators]\label{thm:excitation_operators}~

	Let the excitation index set $\mathcal{I}$ be defined through Definition \ref{def:Excitation_Index}, let $\mathcal{B}_{\wedge}$ denote an $\widehat{\mathcal{L}}^2$-orthonormal basis of Slater determinants for $\widehat{\mathcal{H}}^1$ as defined through Definition \ref{def:Slater_basis}, and let the excitation operators $\{\mathcal{X}_{\mu}\}_{\mu \in \mathcal{I}}$ and de-excitation operators $\{\mathcal{X}^{\dagger}_{\mu}\}_{\mu \in \mathcal{I}}$ be defined through Definitions \ref{def:Excitation_Operator} and \ref{def:De-excitation_Operator} respectively. Then
	\begin{enumerate}

            
         \item For all $\mu \in \mathcal{I}$ the excitation operator $\mathcal{X}_{\mu}\colon \widehat{\mathcal{L}}^2 \rightarrow \widehat{\mathcal{L}}^2$ and de-excitation operator $\mathcal{X}^{\dagger}_{\mu}\colon \widehat{\mathcal{L}}^2 \rightarrow \widehat{\mathcal{L}}^2$ are bounded linear maps\footnote{Note that excitation operators are defined through their action on the basis $\mathcal{B}_{\wedge}$ which, as remarked before, is a complete orthonormal basis for $\widehat{\mathcal{L}}^2$. Thus, we can extend the domain of definition of excitation and de-excitation operators to $\widehat{\mathcal{L}}^2$.}.
         
        \item For all $\mu \in \mathcal{I}$ the de-excitation operator $\mathcal{X}^{\dagger}_{\mu}\colon \widehat{\mathcal{L}}^2 \rightarrow \widehat{\mathcal{L}}^2$ is the $\widehat{\mathcal{L}}^2$-adjoint of the excitation operator $\mathcal{X}_{\mu}\colon \widehat{\mathcal{L}}^2 \rightarrow \widehat{\mathcal{L}}^2$.

        \item For all $\mu, \nu \in \mathcal{I}$, it holds that $\mathcal{X}_{\mu} \mathcal{X}_{\nu}=\mathcal{X}_{\nu} \mathcal{X}_{\mu}$ and $\mathcal{X}^{\dagger}_{\mu} \mathcal{X}^{\dagger}_{\nu}=\mathcal{X}^{\dagger}_{\nu} \mathcal{X}^{\dagger}_{\mu}$ .
		
	\end{enumerate}
\end{theorem}
\begin{proof}
See \cite[Section 2]{MR3021693}.
\end{proof}

The primary motivation for introducing the complicated notion of excitation operators is that they provide a systematic means of indexing an $N$-particle Slater determinant basis $\mathcal{B}_{\wedge}$ in terms of their application on the reference determinant $\Psi_0$. This observation is summarised more precisely in the following remark. 

\begin{remark}[Interpretation of $N$-particle Bases in terms of Excited Determinants]\label{rem:excited_determinants}~
	
	Let $\mathscr{R}$ and $\mathscr{R}^{\perp}$ denote an occupied and virtual space respectively of $\mH^1(\R^3)$ as introduced in Notation~\ref{def:occ_vir}, let $\Psi_0$ denote the reference determinant corresponding to $\mathscr{R}$, let the infinite-dimensional subspace $\widehat{\mathcal{H}}^{1, \perp}_{\Psi_0}\subset \widehat{\mathcal{H}}^1$ be defined as in Definition \ref{def:decomp}, let $\mathscr{B}_{\rm occ}$ and $\mathscr{B}_{\rm vir}$ denote $\mL^2(\R^3)$-orthonormal basis sets for $\mathscr{R}$ and $\mathscr{R}^{\perp}$ respectively, and let ${\mathcal{B}_{\wedge}}$ and $\widetilde{\mathcal{B}_{\wedge}}$ denote $\widehat{\mathcal{L}}^2$-orthonormal basis sets for $\widehat{\mathcal{H}}^1$ and $\widehat{\mathcal{H}}^{1, \perp}_{\Psi_0}$ respectively, constructed from $\mathscr{B}_{\rm occ}$ and $\mathscr{B}_{\rm vir}$, as outlined in Definition \ref{def:Slater_basis}. Then it holds that
	 \begin{equation}\label{eq:basis_excited}
			{\mathcal{B}_{\wedge}}=\{\Psi_0\}\cup\left\{\mathcal{X}_\mu \Psi_0 \colon \mu \in \mathcal{I}\right\} \qquad \text{and} \qquad  \widetilde{\mathcal{B}_{\wedge}}=\left\{\mathcal{X}_\mu \Psi_0 \colon \mu \in \mathcal{I}\right\}.
	\end{equation}

	
\end{remark}

\vspace{3mm}

\subsection{Cluster Operators}\label{sec:3b}~

 In view of Remark \ref{rem:excited_determinants}, we see that for any arbitrary function ${\Phi} \in \widehat{\mathcal{H}}^{1}$, there exists a constant $\bs_0 \in \mathbb{R}$ and a sequence $\{\bs_{\mu}\}_{\mu \in \mathcal{I}} \subset \mathbb{R}$ such that 
 \begin{align}\label{eq:cluster_1}
     \Phi = \bs_0 \Psi_0+ \sum_{\mu \in \mathcal{I}} \bs_{\mu} \mathcal{X}_{\mu} \Psi_0,
 \end{align}
where $\{\mathcal{X}_{\mu}\}_{\mu \in \mathcal{I}}$ denote excitation operators defined according to Definition \ref{def:Excitation_Operator}.

Equation \eqref{eq:cluster_1} suggests that any element of the $N$-particle space $\widehat{\mathcal{H}}^{1}$ can be expressed in terms of an infinite weighted summation of excitation operators acting on the reference determinant~$\Psi_0$. Such infinite summations of excitation operators are known as \emph{cluster} operators in the quantum chemistry literature. Of course, while each excitation operator itself is a bounded map on $\widehat{\mathcal{L}}^2$, it does not immediately follow that an infinite summation of excitation operators is also a well-defined mapping on $\widehat{\mathcal{L}}^2$. Fortunately, the following result was proven in \cite[Theorem 2.7]{MR3021693}.

\begin{proposition}[Cluster Operators as Bounded Maps on $\widehat{\mathcal{L}}^2$]\label{prop:Yvon}~
	
	Let the excitation index set $\mathcal{I}$ be defined through Definition \ref{def:Excitation_Index}, let $\mathcal{B}_{\wedge}$ denote an $\widehat{\mathcal{L}}^2$-orthonormal basis of Slater determinants for $\widehat{\mathcal{H}}^1$ as defined in Definition \ref{def:Slater_basis}, let $\Psi_0 \in \mathcal{B}_{\wedge}$ denote the reference determinant as defined through Definition \ref{def:ref_det}, and let the excitation operators $\{\mathcal{X}_{\mu}\}_{\mu \in \mathcal{I}}$ and de-excitation operators $\{\mathcal{X}^{\dagger}_{\mu}\}_{\mu \in \mathcal{I}}$ be defined through Definitions \ref{def:Excitation_Operator} and \ref{def:De-excitation_Operator} respectively. Then
 \begin{enumerate}
     \item  For any square-summable sequence $\bt= \{\bt_{\mu}\}_{\mu \in \mathcal{I}} \in \ell^2(\mathcal{I})$, there exists a unique bounded linear operator $\mathcal{T} \colon\widehat{\mathcal{L}}^2 \rightarrow \widehat{\mathcal{L}}^2$, the so-called \emph{cluster operator} generated by $\bt$, such that 
     \begin{align*}
         \mathcal{T}:= \sum_{\mu \in \mathcal{I}} \bt_{\mu}\mathcal{X}_{\mu}
     \end{align*}
     where the series convergence holds with respect to the $\Vert \cdot \Vert_{\widehat{\mathcal{L}}^2 \to\widehat{\mathcal{L}}^2} $ operator norm. Moreover, there exists a constant $\beta >0$, depending only on $N$, such that
 \begin{align*}
 \Vert \bt\Vert_{\ell^2(\mathcal{I})} \leq \Vert \mathcal{T}  \Vert_{\widehat{\mathcal{L}}^2 \to \widehat{\mathcal{L}}^2} \leq \beta \Vert \bt\Vert_{\ell^2(\mathcal{I})}.
 \end{align*}

\item In particular, for any $\Phi :=\sum_{\mu \in \mathcal{I}} \bt_{\mu} \mathcal{X}_{\mu} \Psi_0\in \{\Psi_0\}^{\perp} \subset \widehat{\mathcal{L}}^2$ there exists a unique   cluster operator $\mathcal{T}\left(\Phi\right) \colon\widehat{\mathcal{L}}^2 \rightarrow \widehat{\mathcal{L}}^2$ such that
 \begin{align*}
         \mathcal{T}\left(\Phi\right)=\sum_{\mu \in \mathcal{I}} \bt_{\mu}\mathcal{X}_{\mu}
     \end{align*}
     where the series convergence again holds with respect to the $\Vert \cdot \Vert_{\widehat{\mathcal{L}}^2 \to\widehat{\mathcal{L}}^2} $ operator norm. Moreover, the mapping $\{\Psi_0\}^{\perp} \ni \Phi \mapsto \mathcal{T}\left(\Phi\right)$ is linear and bounded, i.e., there exists a constant $\beta >0$, depending only on $N$, such that
 \begin{align*}
 \Vert \Phi\Vert_{\widehat{\mathcal{L}}^2} \leq \Vert \mathcal{T}\left(\Phi\right) \Vert_{\widehat{\mathcal{L}}^2 \to \widehat{\mathcal{L}}^2} \leq \beta \Vert \Phi\Vert_{\widehat{\mathcal{L}}^2}.
 \end{align*}

\item  For any square-summable sequence $\bt= \{\bt_{\mu}\}_{\mu \in \mathcal{I}} \in \ell^2(\mathcal{I})$ with corresponding cluster operator $\mathcal{T} \colon\widehat{\mathcal{L}}^2 \rightarrow \widehat{\mathcal{L}}^2$, the $\widehat{\mathcal{L}}^2$-adjoint $\mathcal{T}^{\dagger} \colon\widehat{\mathcal{L}}^2 \rightarrow \widehat{\mathcal{L}}^2$ satisfies
     \begin{align*}
         \mathcal{T}^{\dagger}=\sum_{\mu \in \mathcal{I}} \bt_{\mu}\mathcal{X}_{\mu}^{\dagger}
     \end{align*}
     where the series convergence once again holds with respect to the $\Vert \cdot \Vert_{\widehat{\mathcal{L}}^2 \to\widehat{\mathcal{L}}^2} $ operator norm. 
 \end{enumerate}
 \end{proposition}

Consider Proposition \ref{prop:Yvon} that describes the mapping properties of cluster operators and Definition \ref{def:Excitation_Operator} of the excitation operators $\{\mathcal{X}_{\mu}\}_{\mu \in \mathcal{I}}$. We recall that excitation operators have a priori been defined relative to an $N$-particle Slater determinant basis $\mathcal{B}_{\wedge}$ constructed from a specific choice of occupied space $\mathscr{R}$ as well as $\mL^2(\R^3)$-orthonormal basis sets $\mathscr{B}_{\rm occ}$ and $\mathscr{B}_{\rm vir}$ for the occupied and virtual space $\mathscr{R}$ and $\mathscr{R}^{\perp}$ respectively. At first glance therefore, it might seem that the definition of a cluster operator $\mathcal{T}\left(\Phi\right)$ generated by some $\Phi \in \{\Psi_0\}^{\perp}$ depends not only on the function $\Phi$ but also on the basis sets $\mathscr{B}_{\rm occ}$ and $\mathscr{B}_{\rm vir}$. The following result shows that the definition of the cluster operator $\mathcal{T}\left(\Phi\right)$ is in fact \emph{independent} of the basis sets $\mathscr{B}_{\rm occ}$ and $\mathscr{B}_{\rm vir}$, and thus depends only on the choice of occupied space $\mathscr{R}$.

\begin{proposition}[Basis-Independence of Cluster Operators]\label{prop:excit_unique}~

    Let $\mathscr{R}$ and $\mathscr{R}^{\perp}$ denote an occupied and virtual space of $\mH^1(\R^3)$ as stated in Notation \ref{def:occ_vir}, let $\Psi_0\in \widehat{\mathcal{H}}^1$ denote the reference determinant corresponding to $\mathscr{R}$ as defined through Definition \ref{def:ref_det}, let $\mathscr{B}^1_{\rm occ}:= \{\psi^1_j\}_{j=1}^N$ and $\mathscr{B}^2_{\rm occ}:= \{\psi^2_j\}_{j=1}^N$ denote two different $\mL^2(\R^3)$-orthonormal bases for the occupied space $\mathscr{R}$, let $\mathscr{B}^1_{\rm vir}:= \{\psi^1_j\}_{j=N+1}^\infty$ and $\mathscr{B}^2_{\rm vir}:= \{\psi^2_j\}_{j=N+1}^\infty$ denote two different $\mL^2(\R^3)$-orthonormal bases for the virtual space $\mathscr{R}^{\perp}$, let $\mathcal{B}_{\wedge}^1$ and $\mathcal{B}_{\wedge}^2$ denote the $N$-particle Slater bases for $\widehat{\mathcal{H}}^1$ constructed using the sets $\mathscr{B}^1_{\rm occ} \cup \mathscr{B}^1_{\rm vir}$ and $\mathscr{B}^2_{\rm pcc} \cup \mathscr{B}^2_{\rm vir}$ respectively as defined through Definition \ref{def:Slater_basis}, let the excitation index set $\mathcal{I}$ be defined through Definition \ref{def:Excitation_Index}, let the excitation operators $\{\mathcal{X}^1_{\mu}\}_{\mu \in \mathcal{I}}$ and $\{\mathcal{X}^2_{\mu}\}_{\mu \in \mathcal{I}}$ corresponding to the $N$-particle Slater bases $\mathcal{B}_{\wedge}^1$ and $\mathcal{B}_{\wedge}^2$ respectively be defined through Definition \ref{def:Excitation_Operator}, and let $\Phi \in \{\Psi_0\}^{\perp} \subset \widehat{\mathcal{L}}^2$ be given by
    \begin{align*}
        \Phi = \sum_{\mu \in \mathcal{I}} \bt_{\mu}^1\mathcal{X}^1_{\mu}\Psi_0 =\sum_{\mu \in \mathcal{I}} \bt_{\mu}^2\mathcal{X}^2_{\mu}\Psi_0.
    \end{align*}

Then it holds that
\begin{align*}
    \mathcal{T}\left(\Phi\right)=\sum_{\mu \in \mathcal{I}} \bt_{\mu}^1\mathcal{X}^1_{\mu} =\sum_{\mu \in \mathcal{I}} \bt_{\mu}^2\mathcal{X}^2_{\mu}.
\end{align*}
\end{proposition}
\begin{proof}
    See \cite[Theorem 2.7 (ii)]{MR3021693}
\end{proof}

Proposition \ref{prop:excit_unique} implies that once the reference determinant $\Psi_0$ constructed from an occupied space $\mathscr{R}$ of $\mH^1_{\rm as}(\mathbb{R}^3)$ is fixed, each element $\Phi \in \{\Psi_0\}^{\perp} \subset\widehat{\mathcal{L}}^2$ uniquely defines a bounded cluster operator $\mathcal{T}\left(\Phi\right) \colon \widehat{\mathcal{L}}^2\rightarrow \widehat{\mathcal{L}}^2$. Of course, in order to have an explicit representation of the cluster operator $\mathcal{T}\left(\Phi\right)$, we require a Slater determinant basis $\mathcal{B}_{\wedge}$ on $\widehat{\mathcal{H}}^1$, which in turn requires a specific choice of bases $\mathscr{B}_{\rm occ}, \mathscr{B}_{\rm vir}$ for the occupied and virtual spaces $\mathscr{R}$ and $\mathscr{R}^{\perp}$ respectively but the above proposition shows that any basis ultimately leads to the same operator. This implies, in particular, that we can define cluster operators by referencing only the occupied space $\mathscr{R}$ and the reference determinant $\Psi_0$ generated by this occupied space. This observation will be crucial for the analysis we present in Section \ref{sec:4}.

Returning to the mapping properties of cluster operators, we notice that our results thus far pertain to~$\widehat{\mathcal{L}}^2$, i.e., the space of anti-symmetric, square-integrable functions of $3N$ variables. On the other hand, the $N$-particle function space for our problem is the Sobolev space $\widehat{\mathcal{H}}^1$. For elements $\Phi \in \widehat{\mathcal{H}}^{1, \perp}_{\Psi_0} =\{\Psi_0\}^{\perp} \cap \widehat{\mathcal{H}}^1$, stronger results on the mapping properties of the cluster operator $\mathcal{T}\left(\Phi\right)$ can be deduced. The following result is the main technical achievement of the seminal contribution \cite{MR3021693}.

\begin{theorem}[Cluster Operators as Bounded Maps on $\widehat{\mathcal{H}}^1$ and $\widehat{\mathcal{H}}^{-1}$]\label{thm:1}~

     Let $\mathscr{R}$ denote an occupied space of $\mH^1(\R^3)$ as stated in Notation \ref{def:occ_vir}, let $\Psi_0\in \widehat{\mathcal{H}}^1$ denote the reference determinant corresponding to $\mathscr{R}$ as defined through Definition \ref{def:ref_det}, let the $N$-particle function space $\widehat{\mathcal{H}}^1$ be decomposed as $\widehat{\mathcal{H}}^1 = \text{\rm span}\{\Psi_0\} \oplus \widehat{\mathcal{H}}^{1, \perp}_{\Psi_0}$ as described in Definition \ref{def:decomp}, let $\Phi \in \widehat{\mathcal{H}}^{1, \perp}$, and let $\mathcal{T}\left(\Phi\right) \colon \widehat{\mathcal{L}}^2\rightarrow \widehat{\mathcal{L}}^2$ denote the cluster operator generated by $\Phi$ as stated in Proposition \ref{prop:Yvon}. Then
     \begin{enumerate}
         \item The cluster operator $\mathcal{T}\left(\Phi\right) \colon \widehat{\mathcal{H}}^1\rightarrow \widehat{\mathcal{H}}^1$ is a bounded linear map, and there exists a constant $\beta_{\mathcal{H}}>0$, depending only on $N$ and $\Vert \Psi_0 \Vert_{\widehat{\mathcal{H}}^1}$, such that
         \begin{align*}
              \Vert \Phi\Vert_{\widehat{\mathcal{H}}^1} \leq \Vert \mathcal{T}\left(\Phi\right) \Vert_{\widehat{\mathcal{H}}^1 \to \widehat{\mathcal{H}}^1} \leq \beta_{\mathcal{H}} \Vert \Phi\Vert_{\widehat
{\mathcal{H}}^1}.
         \end{align*}

          \item The $\widehat{\mathcal{L}}^2$-adjoint cluster operator $\mathcal{T}\left(\Phi\right)^{\dagger} \colon \widehat{\mathcal{H}}^1\rightarrow \widehat{\mathcal{H}}^1$ is also a bounded linear map and there exists a constant $\beta^{\dagger}_{\mathcal{H}}>0$, depending only on $N$ and $\Vert \Psi_0 \Vert_{\widehat{\mathcal{H}}^1}$, such that
         \begin{align*}
              \Vert \mathcal{T}\left(\Phi\right) ^{\dagger}\Vert_{\widehat{\mathcal{H}}^1 \to \widehat{\mathcal{H}}^1} \leq \beta^{\dagger}_{\mathcal{H}} \Vert \Phi\Vert_{\widehat{\mathcal{H}}^1}.
         \end{align*}

          \item  The cluster operator $\mathcal{T}\left(\Phi\right) \colon \widehat{\mathcal{H}}^1\rightarrow \widehat{\mathcal{H}}^1$ can be extended to a bounded map from $\widehat{\mathcal{H}}^{-1}$ to $\widehat{\mathcal{H}}^{-1}$ and it holds that
         \begin{align*}
              \Vert \mathcal{T}\left(\Phi\right) \Vert_{\widehat{\mathcal{H}}^{-1} \to \widehat{\mathcal{H}}^{-1}} =\Vert \mathcal{T}\left(\Phi\right) ^{\dagger}\Vert_{\widehat{\mathcal{H}}^1 \to \widehat{\mathcal{H}}^1}.
         \end{align*}
     \end{enumerate}
\end{theorem}
\begin{proof}
    See \cite[Theorem 4.1, Lemma 5.1]{MR3021693}.
\end{proof}

\begin{theorem}[Cluster Operators Form an Algebra]\label{thm:2}~

     Let $\mathscr{R}$ denote an occupied space of $\mH^1(\R^3)$ as stated in Notation \ref{def:occ_vir}, let $\Psi_0\in \widehat{\mathcal{H}}^1$ denote the reference determinant corresponding to $\mathscr{R}$ as defined through Definition \ref{def:ref_det}, let the $N$-particle function space $\widehat{\mathcal{H}}^1$ be decomposed as $\widehat{\mathcal{H}}^1 = \text{\rm span}\{\Psi_0\} \oplus \widehat{\mathcal{H}}^{1, \perp}_{\Psi_0}$ as described in Definition \ref{def:decomp}, and define the set of operators 
     \begin{align*}
			\mathfrak{L}:= \left\{t_0 {\rm I} +  \mathcal{T}\left(\Phi\right) \colon \quad t_0 \in \mathbb{R}, ~ {\Phi} \in  \widehat{\mathcal{H}}^{1, \perp}_{\Psi_0} ~ \text{ and }  ~\mathcal{T}\left(\Phi\right)~\text{ is the cluster operator generated by } \Phi\right\}.
		\end{align*}
		Then the following hold:
		
		\begin{itemize}
			\item The set $\mathfrak{L}$ forms a closed commutative subalgebra in the algebra of bounded linear operators acting from $\widehat{\mathcal{H}}^{1}$ to $\widehat{\mathcal{H}}^{1}$ (and also from $\widehat{\mathcal{H}}^{-1}$ to $\widehat{\mathcal{H}}^{-1}$).
			
			\item The spectrum of any $ \mathcal{L}(\Phi)= t_0{\rm I} +  \mathcal{T}(\Phi) \in \mathfrak{L}$ is exactly $\sigma(\mathcal{L})=\{t_0\}$. In particular any element $ \mathcal{L}(\Phi)= t_0{\rm I} +  \mathcal{T}(\Phi)\in \mathfrak{L}$ with $t_0 \neq 0$ is invertible and the inverse is also an element of $\mathfrak{L}$.
			
			\item Any element in $\mathfrak{L}$ of the form $\mathcal{T}(\Phi) =\sum_{\mu \in \mathcal{I}}  \bt_{\mu} \mathcal{X}_{\mu}$ is nilpotent: it holds that $\mathcal{T}^{N+1}\equiv 0$.
			
			\item The exponential function is a locally $\mathscr{C}^{\infty}$ map on $\mathfrak{L}$, and is also a bijection from the sub-algebra
   \begin{align*}
			\left\{\mathcal{T}\left(\Phi\right) \colon \hspace{2mm}\qquad  {\Phi} \in  \widehat{\mathcal{H}}^{1, \perp}_{\Psi_0} ~ \text{ and }  ~\mathcal{T}\left(\Phi\right)~\text{ is the cluster operator generated by } \Phi\right\}.
		\end{align*}
  to the sub-algebra
  \begin{align*}
			\left\{{\rm I} +  \mathcal{T}\left(\Phi\right) \colon \quad {\Phi} \in  \widehat{\mathcal{H}}^{1, \perp}_{\Psi_0} ~ \text{ and }  ~\mathcal{T}\left(\Phi\right)~\text{ is the cluster operator generated by } \Phi\right\}.
		\end{align*}
   
		\end{itemize}   
\end{theorem}
\begin{proof}
    See \cite[Lemma 5.2]{MR3021693}.
\end{proof}

\vspace{3mm}

\subsection{The Continuous Coupled Cluster Equations and their Analysis}\label{sec:3c}~

Theorems \ref{thm:1} and \ref{thm:2} are the foundation stones of the continuous (infinite-dimensional) coupled cluster equations. Indeed, as a consequence of Theorems \ref{thm:1} and \ref{thm:2}, it becomes possible to prove that for a given reference determinant $\Psi_0$, any intermediately normalised element of the $N$-particle function space $\widehat{\mathcal{H}}^1$, i.e., any $\Phi \in \widehat{\mathcal{H}}^1$ such that $(\Psi_0, \Phi)_{\widehat{\mathcal{L}}^2}=1$, can be parameterised through the action of an exponential cluster operator acting on~$\Psi_0$. 

More precisely, let $\mathscr{R}$ and $\mathscr{R}^{\perp}$ denote an occupied and virtual space of $\mH^1(\R^3)$ as stated in Notation~\ref{def:occ_vir}, let $\Psi_0\in \widehat{\mathcal{H}}^1$ denote the reference determinant corresponding to $\mathscr{R}$ as defined through Definition \ref{def:ref_det}, and let the $N$-particle function space $\widehat{\mathcal{H}}^1$ be decomposed as $\widehat{\mathcal{H}}^1 = \text{\rm span}\{\Psi_0\} \oplus \widehat{\mathcal{H}}^{1, \perp}_{\Psi_0}$ as described in Definition \ref{def:decomp}. Then for any $\Phi \in \widehat{\mathcal{H}}^1$ such that $\left(\Phi, \Psi_0\right)_{\widehat{\mathcal{L}}^2}=1$, there exists a cluster operator $\mathcal{T}\left(\Theta\right)$ generated by some $\Theta\in \widehat{\mathcal{H}}^{1, \perp}_{\Psi_0}$ such that
\begin{align}\label{thm:3}
	\Phi = e^{\mathcal{T}(\Theta)} \Psi_0.
\end{align}

Equation \eqref{thm:3} implies in particular that if the sought-after ground state wave-function $\Psi^*_{\rm GS} \in \widehat{\mathcal{H}}^1$ that solves the minimisation problem \eqref{eq:Ground_State} is intermediately normalised, then it can be written in the form
\begin{align*}
	\Psi^*_{\rm GS} = e^{\mathcal{T}(\Theta^*)}\Psi_0,
\end{align*}
for some cluster operator $\mathcal{T}(\Theta^*)$ generated by some element $\Theta^* \in \widehat{\mathcal{H}}^{1, \perp}_{\Psi_0}$ (see, e.g., \cite{MR3021693}). In other words, the minimisation problem \eqref{eq:Ground_State} can be replaced by an equivalent problem which consists of finding the appropriate cluster operator $\mathcal{T}(\Theta^*)$ that appears in the exponential parametrisation of $\Psi_{\rm GS}^*$. This observation leads directly to the so-called continuous coupled cluster equations.

\vspace{3mm}

\textbf{Continuous Coupled Cluster Equations:}~

 Let $\mathscr{R}$ denote an occupied space of $\mH^1(\R^3)$ as stated in Notation \ref{def:occ_vir}, let $\Psi_0\in \widehat{\mathcal{H}}^1$ denote the reference determinant corresponding to $\mathscr{R}$ as defined through Definition \ref{def:ref_det}, and let the $N$-particle function space $\widehat{\mathcal{H}}^1$ be decomposed as $\widehat{\mathcal{H}}^1 = \text{\rm span}\{\Psi_0\} \oplus \widehat{\mathcal{H}}^{1, \perp}_{\Psi_0}$ as described in Definition~\ref{def:decomp}. We seek a cluster operator $\mathcal{T}(\Theta^*)$ generated by some element $\Theta^*\in \widehat{\mathcal{H}}^{1, \perp}_{\Psi_0}$ according to Proposition~\ref{prop:Yvon} such that for all $\Phi \in \widehat{\mathcal{H}}^{1, \perp}_{\Psi_0}$ it holds that
\begin{equation}\label{eq:CC}
	\left\langle \Phi, e^{-\mathcal{T}(\Theta^*)}H e^{\mathcal{T}(\Theta^*)} \Psi_0\right\rangle_{\widehat{\mathcal{H}}^{1} \times \widehat{\mathcal{H}}^{-1}}	=0.
\end{equation}

Once Equation \eqref{eq:CC} has been solved, the associated coupled cluster energy $\mathcal{E}_{\rm CC}^*$ is given by
\begin{equation}\label{eq:CC_Energy}
	\mathcal{E}_{\rm CC}^* := \left\langle\Psi_0, e^{-\mathcal{T}(\Theta^*)}H e^{\mathcal{T}(\Theta^*)} \Psi_0\right\rangle_{\widehat{\mathcal{H}}^{1} \times \widehat{\mathcal{H}}^{-1}}.
\end{equation}

The continuous coupled cluster equations \eqref{eq:CC} are thus an infinite-dimensional system of non-linear equations in which the unknown is the cluster operator that appears in the exponential parametrisation of the sought-after ground-state eigenfunction $\Psi_{\rm GS}^*$ of the electronic Hamiltonian. In practice, these equations are approximated by considering suitable Galerkin discretisations, which shall be the subject of an extensive discussion in Section \ref{sec:4}. For the remainder of the current section however, we restrict ourselves to a discussion of the continuous coupled cluster equations themselves. 

From the perspective of numerical analysis, the natural question that arises is whether or not the continuous coupled cluster equations are well-posed. To discuss the answer to this question, it is useful to first introduce the infinite-dimensional coupled cluster function.

\vspace{0mm}

\begin{definition}[Continuous Coupled Cluster function]~\label{def:CC}

 Let $\mathscr{R}$ denote an occupied space of $\mH^1(\R^3)$ as stated in Notation \ref{def:occ_vir}, let $\Psi_0\in \widehat{\mathcal{H}}^1$ denote the reference determinant corresponding to $\mathscr{R}$ as defined through Definition \ref{def:ref_det}, let the $N$-particle function space $\widehat{\mathcal{H}}^1$ be decomposed as $\widehat{\mathcal{H}}^1 = \text{\rm span}\{\Psi_0\} \oplus \widehat{\mathcal{H}}^{1, \perp}_{\Psi_0}$ as described in Definition~\ref{def:decomp}, and for any $\Theta \in \widehat{\mathcal{H}}^{1, \perp}_{\Psi_0}$, let $\mathcal{T}(\Theta)$ denote the cluster operator generated by $\Theta$ as described in Proposition \ref{prop:Yvon}. We define the continuous coupled cluster function $\mathcal{f} \colon \widehat{\mathcal{H}}^{1, \perp}_{\Psi_0} \rightarrow \left(\widehat{\mathcal{H}}^{1, \perp}_{\Psi_0}\right)^*$ as the mapping with the property that for all $\Theta, \Phi \in \widehat{\mathcal{H}}^{1, \perp}_{\Psi_0}$ it holds that
\begin{align}\label{eq:12bis}
    \big\langle \Phi, \mathcal{f}(\Theta) \big \rangle_{\widehat{\mathcal{H}}^{1, \perp}_{\Psi_0} \times \left(\widehat{\mathcal{H}}^{1, \perp}_{\Psi_0}\right)^*}:= \left\langle \Phi, e^{-\mathcal{T}(\Theta)}H e^{\mathcal{T}(\Theta)}\Psi_0\right\rangle_{\widehat{\mathcal{H}}^1 \times \widehat{\mathcal{H}}^{-1}}. 
\end{align} 
\end{definition}

Consider the continuous coupled cluster equations \eqref{eq:CC} and Definition \ref{def:CC} of the continuous coupled cluster function. It can readily be checked that if a cluster operator $\mathcal{T}(\Theta^*)$ generated by some element $\Theta^*\in \widehat{\mathcal{H}}^{1, \perp}_{\Psi_0}$ is a solution to the CC equations \eqref{eq:CC}, then $\Theta^*$ is a zero of the coupled cluster function $\mathcal{f} \colon \widehat{\mathcal{H}}^{1, \perp}_{\Psi_0} \rightarrow \left(\widehat{\mathcal{H}}^{1, \perp}_{\Psi_0}\right)^*$. The converse is also true of course, and this justifies the decision to directly study the zeros of the coupled cluster function $\mathcal{f}$ using the usual tools of non-linear numerical analysis.

It turns out that there is a clear relationship between the zeros of the coupled cluster function $\mathcal{f}$ and eigenfunctions of the electronic Hamiltonian $H$. Indeed, as demonstrated in \cite[Theorem 5.3]{MR3021693}, 

\begin{enumerate}
    \item If $\Phi^* \in \widehat{\mathcal{H}}^{1, \perp}_{\Psi_0}$ is a zero of $\mathcal{f}$, then $\Psi^* = e^{\mathcal{T}(\Phi^*)}\Psi_0$ is an intermediately normalised eigenfunction of $H$.

    \item Conversely, if $\Psi^* \in \widehat{\mathcal{H}}^1$ is an intermediately normalised eigenfunction of $H$, then there exists a $\Phi^* \in \widehat{\mathcal{H}}^{1, \perp}_{\Psi_0}$ such that $\Psi^* = e^{\mathcal{T}(\Phi^*)}\Psi_0$ and $\Phi^*$ is a zero of the coupled cluster function $\mathcal{f}$.
\end{enumerate}

In other words, we have a complete classification of the zeros of the coupled cluster function $\mathcal{f}$, namely, that they correspond precisely to intermediately normalised eigenfunctions of the electronic Hamiltonian, and the remaining task is to determine conditions under which these zeros are non-degenerate.

A sufficient condition for non-degeneracy was first given in the seminal articles \cite{MR3021693, MR3110488, Schneider_1} wherein the authors proposed the use of the following \emph{local, strong monotonicity criterion}: let $\Theta^*_{\rm GS} \in \widehat{\mathcal{H}}^{1, \perp}_{\Psi_0}$ denote the zero of the coupled cluster function corresponding to the intermediately normalised ground-state eigenfunction~$\Psi^*_{\rm GS}$ of the electronic Hamiltonian $H$, i.e., let $\Psi_{\rm GS}^* = e^{\mathcal{T}(\Theta^*_{\rm GS})}\Psi_0$, and for any $r>0$, let $\mathbb{B}_{r}(\Theta^*_{\rm GS})$ denote the open ball in $ \widehat{\mathcal{H}}^1$ of radius $r$ centred at $\Theta^*_{\rm GS}$. If there exist two constants $\delta>0$ and $\gamma>0$ such that
\begin{equation}
    \forall ~\Phi, \Upsilon \in \mathbb{B}_{\delta}(\Theta^*_{\rm GS}) \cap \widehat{\mathcal{H}}^{1, \perp}_{\Psi_0}\colon \quad \big\langle \Upsilon - \Phi, \mathcal{f}(\Upsilon)- \mathcal{f}(\Phi)\big\rangle_{\widehat{\mathcal{H}}^{1, \perp}_{\Psi_0} \times \left(\widehat{\mathcal{H}}^{1, \perp}_{\Psi_0}\right)^*} \geq \gamma \Vert \Upsilon -\Phi\Vert^2_{\widehat{\mathcal{H}}^1},
\end{equation}
then $\Theta^*_{\rm GS}$ is a non-degenerate zero of the coupled cluster function (see \cite[Theorem 4.1]{MR3110488}).

An important advantage of this approach is that the strong local monotonicity property is directly inherited by \emph{conforming Galerkin discretisations} of the continuous coupled cluster equations, i.e., by the discrete coupled cluster equations that are solved in practice. Consequently, there is a straightforward path to a priori error estimates for practical discretisations of the continuous coupled cluster equations (see \cite[Theorem 4.1]{MR3110488}). 

On the other hand, an unfortunate feature of the local monotonicity approach is that the theoretical local monotonicity constant that one can actually obtain (see \cite[Theorem 3.4]{MR3110488} or \cite[Theorem 5.7]{Schneider_1}) is of the form 
\begin{align}\label{eq:loc_mon}
    \gamma \approx \Lambda_0 -& \Vert H - \mathcal{E}_{\rm GS}^*\Vert_{\widehat{\mathcal{H}}^1 \to \widehat{\mathcal{H}}^{-1}} \Vert \mathcal{T}(\Theta^*_{\rm GS})-\mathcal{T}(\Theta^*_{\rm GS})^{\dagger}\Vert_{\widehat{\mathcal{H}}^1 \to \widehat{\mathcal{H}}^{1}} \\ \nonumber
    -& \mathcal{O}\left(\Vert \mathcal{T}(\Theta^*_{\rm GS})\Vert^2_{\widehat{\mathcal{H}}^1 \to \widehat{\mathcal{H}}^{1}}+\Vert \mathcal{T}(\Theta^*_{\rm GS})^{\dagger}\Vert^2_{\widehat{\mathcal{H}}^1 \to \widehat{\mathcal{H}}^{1}}\right), 
\end{align}
where $\Lambda_0$ is the coercivity constant of the shifted electronic Hamiltonian $H- \mathcal{E}_{\rm GS}^*$ on $\left\{\Psi_{\rm GS}^*\right\}^{\perp} \cap \widehat{\mathcal{H}}^1$. Even though this coercivity constant is guaranteed to be positive if $H$ possesses a spectral gap, the theoretical constant $\gamma$ as estimated from \eqref{eq:loc_mon} may be negative, even for well-behaved molecular systems (see, for instance, Table \ref{table2} below). 

In order to deal with this difficulty, the present authors proposed a new analysis based on establishing the invertibility of the Fr\'echet derivative of the coupled cluster function using inf-sup arguments. This analysis was the subject of our previous contribution \cite{Hassan_CC}, where we were able to establish, in particular, the following result.

\begin{theorem}[Invertibility of the Coupled Cluster Fr\'echet Derivative]\label{thm:CC_der_inv}~

Let the coupled cluster function $\mathcal{f} \colon \widehat{\mathcal{H}}^{1, \perp}_{\Psi_0} \rightarrow \left(\widehat{\mathcal{H}}^{1, \perp}_{\Psi_0}\right)^* $ be defined through Equation \eqref{eq:12bis} for some choice of occupied space $\mathscr{R}$ and reference determinant $\Psi_0$ corresponding to $\mathscr{R}$, let $\mD\mathcal{f}(\Theta) \colon \widehat{\mathcal{H}}^{1, \perp}_{\Psi_0} \rightarrow \left(\widehat{\mathcal{H}}^{1, \perp}_{\Psi_0}\right)^*$ denote the Fr\'echet derivative of $\mathcal{f}$ at any $\Theta \in \widehat{\mathcal{H}}^{1, \perp}_{\Psi_0}$, and let $\Theta^* \in \widehat{\mathcal{H}}^{1, \perp}_{\Psi_0}$ denote a zero of the coupled cluster function corresponding to an intermediately normalised eigenfunction $\Psi^* \in \widehat{\mathcal{H}}^1$ of the electronic Hamiltonian $H\colon\widehat{\mathcal{H}}^1 \rightarrow \widehat{\mathcal{H}}^{-1}$ with a simple, isolated eigenvalue $\mathcal{E}^*$. Then, the Fr\'echet derivative $\mD\mathcal{f}(\Theta^*)$ is an isomorphism and it holds that
\begin{align*}
	\Vert \mD \mathcal{f}(\Theta^*)^{-1}\Vert_{\left(\widehat{\mathcal{H}}^{1, \perp}_{\Psi_0}\right)^{*} \to \widehat{\mathcal{H}}^{1, \perp}_{\Psi_0}}^{-1} \geq \frac{\Lambda^*}{\beta},
\end{align*}
where $\Lambda^*>0$ is the inf-sup constant of the shifted electronic Hamiltonian $H- \mathcal{E}^*$ on $\left\{\Psi^*\right\}^{\perp} \cap \widehat{\mathcal{H}}^1$, i.e.,
\begin{align*}
    \Lambda^* :=& \inf_{0\neq \Upsilon \in \left\{\Psi^*\right\}^{\perp} \cap \widehat{\mathcal{H}}^1} \; \sup_{0\neq \Phi \in \left\{\Psi^*\right\}^{\perp} \cap \widehat{\mathcal{H}}^1}  \frac{\left\langle \Phi, \left(H-\mathcal{E}^*\right)\Upsilon\right\rangle_{\widehat{\mathcal{H}}^1 \times  \widehat{\mathcal{H}}^{-1}}}{\Vert \Phi\Vert_{\widehat{\mathcal{H}}^1} \Vert \Upsilon\Vert_{\widehat{\mathcal{H}}^1}},\\[0.5em]
    \text{and }~ \beta:=& \Vert \mathbb{P}_0^{\perp}e^{-\mathcal{T}(\Theta^*)}\Vert_{\widehat{\mathcal{H}}^1 \to \widehat{\mathcal{H}}^1} \Vert e^{\mathcal{T}(\Theta^*)^{\dagger}}\Vert_{\widehat{\mathcal{H}}^1 \to \widehat{\mathcal{H}}^1}.
\end{align*}

In particular, $\Theta^*$ is a non-degenerate zero of the coupled cluster function $\mathcal{f}$.
\end{theorem}
\begin{proof}
    See \cite[Theorems 31 and 33, Corollary 32, and Remark 34]{Hassan_CC}.
\end{proof}

\begin{table}[ht]
	\centering
	\begin{tabular}{||c| c| c| c| c||} 
		\hline \hline
		Molecule  &  \shortstack{Monotonicity constant \\$\Gamma$ from Eq. \eqref{eq:loc_mon}} &  $\Vert \mD \mathcal{f}(\Theta^*)^{-1}\Vert^{-1}_{\left(\widehat{\mathcal{H}}^{1, \perp}_{\Psi_0}\right)^{*} \to \widehat{\mathcal{H}}^{1, \perp}_{\Psi_0}}$  & \shortstack{ Inf-Sup constant\\[0.25em] $\nicefrac{\Lambda^*}{\beta}$}\\ [0.5ex] 
		\hline\hline
		{\rm BeH$_2$}&  \hphantom{-}0.0363 & 0.3379 &0.2568\\ 
		{\rm BH$_3$}   & {-0.0950} &  0.3060&0.2081\\
		\rm{HF} & {-0.0083} &0.2995&0.2529\\
		\rm{H$_2$O}    &  \hphantom{-}0.0249 &0.4113&0.2784\\ 
		\rm{LiH}    & {-0.0065} &0.2628&0.2164\\
		\rm{NH$_3$}   & {-0.0325} &0.3576&0.2789\\[1ex] 
		\hline\hline
	\end{tabular}
	\caption{Examples of numerically computed constants for a collection of small molecules at equilibrium geometries. The calculations were performed in STO-6G basis sets with the exception of the HF and LiH molecules for which 6-31G basis sets were used. To simplify calculations, the canonical $\widehat{\mathcal{H}}^1$ norm was replaced with an equivalent norm induced by the mean-field Hartree-Fock operator (see, e.g., \cite{Schneider_1}). The numerical results shown here are taken from the previous contribution \cite{Hassan_CC}.}\label{table2}
\end{table}

Compared to the local monotonicity approach, Theorem \ref{thm:CC_der_inv} provides a sharper estimate of the operator norm of the inverse coupled cluster Fr\'echet derivative (see also Table \ref{table2}). Unfortunately, as is well known from the classical numerical analysis theory for linear PDEs, continuous inf-sup conditions are not, in general, inherited by Galerkin discretisations of the infinite-dimensional problem.  The natural question to ask therefore, is if an analysis in the spirit of Theorem \ref{thm:CC_der_inv} can be carried out for discretisations of the continuous coupled cluster equations \eqref{eq:CC}. This is the main subject of the remainder of this article.

\vspace{1mm}

\subsection{The Discrete Coupled Cluster Equations}\label{sec:3d}~

We begin with a brief, informal description of generic discretisations of the continuous coupled cluster equations \eqref{eq:CC} that are typically considered in the numerical practice. 

The essential idea of these discretisations is to start with some choice $\mathscr{R}$ of an occupied space of $\mH^1(\R^3)$ as stated in Notation \ref{def:occ_vir}. Corresponding to this choice of $\mathscr{R}$, we obviously have a reference determinant $\Psi_0\in \widehat{\mathcal{H}}^1$ (see Definition \ref{def:ref_det}) and a decomposition of the $N$-particle function space $\widehat{\mathcal{H}}^1 = \text{\rm span}\{\Psi_0\} \oplus \widehat{\mathcal{H}}^{1, \perp}_{\Psi_0}$ as described in Definition \ref{def:decomp}. The key discretisation step now consists of introducing a finite-dimensional subspace $\mathcal{V}\subset \widehat{\mathcal{H}}^{1, \perp}_{\Psi_0}$ that is spanned by a finite Slater determinant basis. We then seek a cluster operator $\mathcal{T}(\Theta_{\mathcal{V}})$ generated by some element $\Theta_{\mathcal{V}}\in \mathcal{V} $ such that for all $\Phi_{\mathcal{V}} \in \mathcal{V}$ it holds that
\begin{equation}\label{eq:CC_galerkin}
	\left\langle \Phi_{\mathcal{V}}, e^{-\mathcal{T}(\Theta_{\mathcal{V}})}H e^{\mathcal{T}(\Theta_{\mathcal{V}})} \Psi_0\right\rangle_{\widehat{\mathcal{H}}^{1} \times \widehat{\mathcal{H}}^{-1}}	=0.
\end{equation}

Note that Equation \eqref{eq:CC_galerkin} is simply a Galerkin approximation of the continuous coupled cluster equation \eqref{eq:CC}. The associated discrete coupled cluster energy $\mathcal{E}^*_{\mathcal{V}}$ is therefore defined in analogy with the continuous coupled cluster energy \eqref{eq:CC_Energy} as
\begin{equation}\label{eq:CC_Energy_galerkin}
	\mathcal{E}_{\mathcal{V}}^* := \left\langle\Psi_0, e^{-\mathcal{T}(\Theta_{\mathcal{V}})}H e^{\mathcal{T}(\Theta_{\mathcal{V}})} \Psi_0\right\rangle_{\widehat{\mathcal{H}}^{1} \times \widehat{\mathcal{H}}^{-1}}.
\end{equation}

\begin{remark}[Solving the Discrete Coupled Cluster equations]\label{rem:solve_disc}~

Consider the discrete coupled cluster equations \eqref{eq:CC_galerkin}. A natural question that now arises is how the approximation space $\mathcal{V}$ is typically chosen and the resulting discrete equations formulated in a manner amenable to numerical resolution.

To construct the approximation space $\mathcal{V}$, we typically proceed in two steps:

\begin{enumerate}

\item We first fix an $N$-dimensional occupied space $\mathscr{R} \subset \mH^1(R^3)$. We then introduce a \emph{finite-dimensional subspace} $\mathscr{R}^{\perp}_{\rm approx}$ of the virtual space $\mathscr{R}^{\perp}$ together with a corresponding finite $\mL^2$-orthonormal basis $\mathscr{B}_{\rm vir}^{\rm approx}$. 

\item Next, we introduce $\widetilde{\mathcal{B}}_{\wedge, \rm approx}$ as \underline{some subset} of all possible Slater determinants that can be constructed from $\mathscr{B}_{\rm occ} \cup \mathscr{B}_{\rm vir}^{\rm approx}$. The approximation space $\mathcal{V}$ is now defined as 
\begin{align*}
\mathcal{V}:= \text{\rm span}\widetilde{\mathcal{B}}_{\wedge, \rm approx}.
\end{align*}
\end{enumerate}

Turning now to the practical numerical resolution of the discrete coupled cluster equations \eqref{eq:CC_galerkin}, we observe that, using the notion of excitation index sets and excitation operators introduced in Section~\ref{sec:3a}, we can parameterise the set of Slater determinants  $\widetilde{\mathcal{B}}_{\wedge, \rm approx}$ as
\begin{align*}
    \exists \mathcal{I}_{\mathcal{V}} \subset \mathcal{I} \colon \quad \widetilde{\mathcal{B}}_{\wedge, \rm approx}= \left\{\mathcal{X}_{\mu} \Psi_0\colon ~ \mu \in \mathcal{I}_{\mathcal{V}}\right\}.
\end{align*}
It follows that the sought-after cluster operator $\mathcal{T}(\Theta_{\mathcal{V}})$ generated by $\Theta_{\mathcal{V}} \in {\mathcal{V}}$ can be written as
\begin{align*}
    \mathcal{T}(\Theta_{\mathcal{V}})=\sum_{\mu \in \mathcal{I}_{\mathcal{V}}} \bt^{\mathcal{V}}_{\mu} \mathcal{X}_{\mu},
\end{align*}
for the unknown expansion coefficients $\{\bt^{\mathcal{V}}_{\mu}\}_{\mu \in \mathcal{I}_{\mathcal{V}}}$ of $\Theta_{\mathcal{V}} \in {\mathcal{V}}$. 

Solving the discrete coupled cluster equations \eqref{eq:CC_galerkin} then consists of determining the unknown coefficients $\{\bt^\mathcal{V}_{\mu}\}_{\mu \in \mathcal{I}_{\mathcal{V}}}$ such that for all $\mu \in \mathcal{I}_{\mathcal{V}}$ it holds that
\begin{equation}\label{eq:CC_discrete_2a}
	\left\langle \mathcal{X}_{\mu} \Psi_0, e^{-\mathcal{T}(\Theta_{\mathcal{V}})}H e^{\mathcal{T}(\Theta_{\mathcal{V}})} \Psi_0\right\rangle_{\widehat{\mathcal{H}}^{1} \times \widehat{\mathcal{H}}^{-1}}	=0, \quad \text{ where } ~\mathcal{T}(\Theta_{\mathcal{V}})= \sum_{\mu \in \mathcal{I}_{\mathcal{V}}} \bt^{\mathcal{V}}_{\mu} \mathcal{X}_{\mu}.
\end{equation}
Equation \eqref{eq:CC_discrete_2a} is simply a finite-dimensional system of coupled, non-linear equations that can readily be solved using iterative methods such as a quasi-Newton scheme.

\end{remark}


As Remark \ref{rem:solve_disc} suggests, there are essentially two discretisation parameters that can be tuned to produce different variations of the discrete coupled cluster equations.

\begin{itemize}
    \item We can modify the choice of occupied space $\mathscr{R}$ and the size of the corresponding virtual space approximation $\mathscr{R}^{\perp}_{\rm approx}$. Typically, these spaces are defined as the span of some non-orthogonal functions (known as atomic basis functions in the quantum chemistry literature). In a subsequent step, we then produce $\mL^2$-orthonormal bases $\mathscr{B}_{\rm occ}$ and $\mathscr{B}_{\rm vir}^{\rm approx}$ for $\mathscr{R}$ and $\mathscr{R}^{\perp}_{\rm approx}$ respectively.

    \item We can modify the choice of Slater determinants constructed from $\mathscr{B}_{\rm occ}\cup\mathscr{B}_{\rm vir}^{\rm approx}$ that are used to define $\mathcal{V} \subset \widehat{\mathcal{H}}^1$. In the quantum chemistry literature, Slater determinants are typically characterised according to their excitation rank. The excitation rank of a Slater determinant $\Psi_\mu \in \widehat{\mathcal{H}}^{1, \perp}_{\Psi_0}$ is the natural number $j \in \{1, \ldots, N\}$ such that the excitation index $\mu$ is an element of the $j^{\rm th}$ excitation index set $\mathcal{I}_j$. With this classification, we can, for instance, define $\mathcal{V}$ as the span of only singly and doubly excited Slater determinants constructed from $\mathscr{B}_{\rm occ}\cup\mathscr{B}_{\rm vir}^{\rm approx}$, the resulting discretisation being known as coupled cluster singles and doubles (CCSD). At the other extreme, we can define $\mathcal{V}$ as the span of all possible Slater determinants constructed from $\mathscr{B}_{\rm occ}\cup\mathscr{B}_{\rm vir}^{\rm approx}$, this being called full coupled cluster (Full-CC). 
\end{itemize}

A graphical description of these discretisation parameters and the resulting discrete coupled cluster equations is given in Figure \ref{fig:1}. Note that the continuous coupled cluster equations correspond to the case of taking a complete $\mL^2$-orthonormal basis of $\mH^1(\R^3)$ as the atomic basis and including Slater determinants of all possible excitation ranks in the approximation space. Let us also remark that, despite the simple appearance of Figure \ref{fig:1}, one can consider coupled cluster discretisations in which Slater determinants of different excitation ranks are only partially included, i.e., instead of using, e.g., all doubly excited Slater determinants to construct the approximation space, we can use \emph{some} doubly excited, \emph{some} triply excited, and \emph{some} quadruply excited determinants etc (see, e.g., \cite{lyakh2010adaptive}).

\begin{figure}[ht]
\begin{center}
\includegraphics[width=1\textwidth, trim={0cm, 1cm, 0cm, 1cm},clip=true]{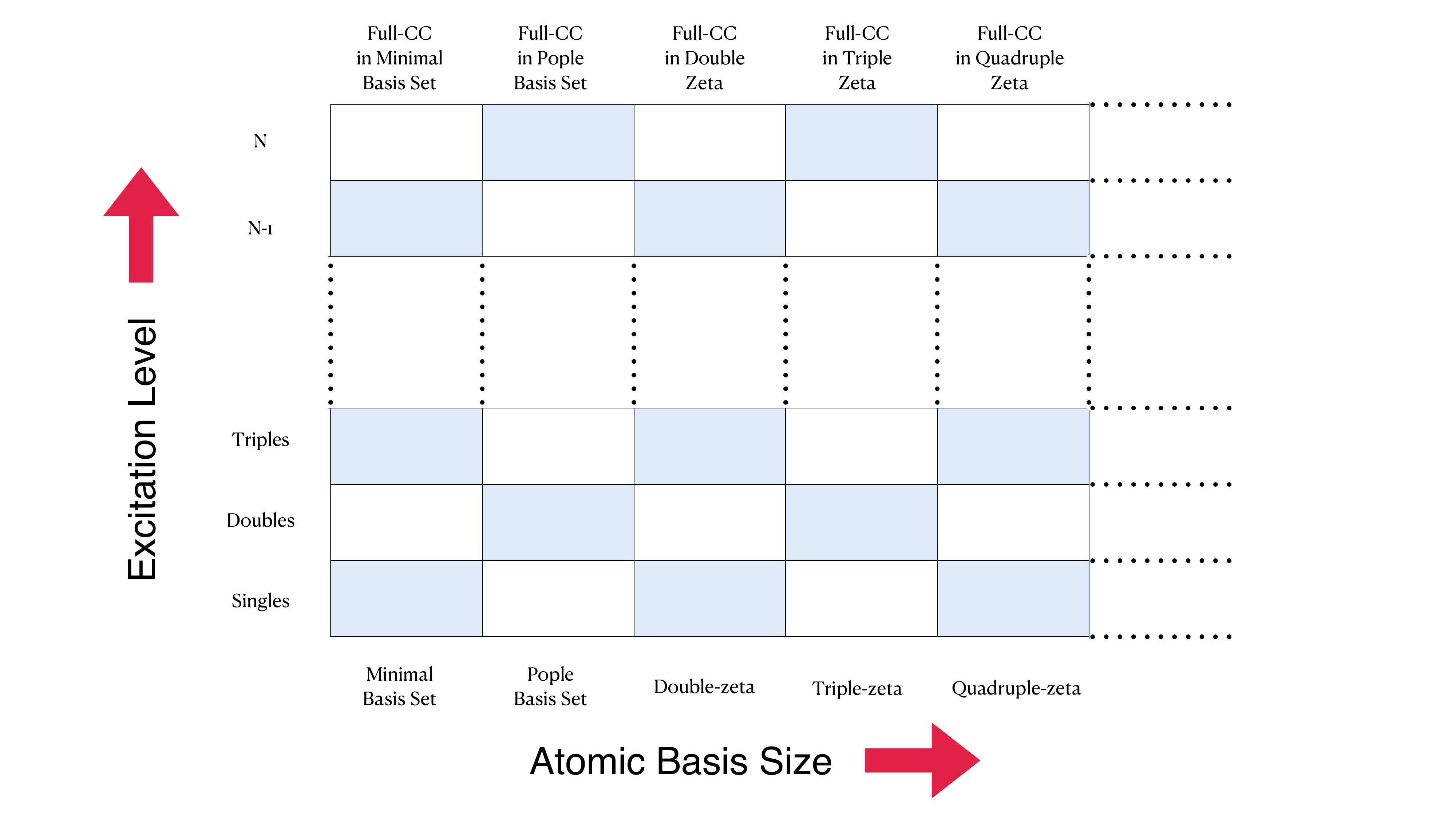}
\end{center}
\caption{Graphical depiction of the coupled cluster discretisation parameters. Informally, the continuous coupled cluster equations should appear at the upper right corner of this graphic.}
\label{fig:1}
\end{figure}

The task that we now confront is to develop a local well-posedness theory and a priori error estimates for the discrete coupled cluster equations \eqref{eq:CC_galerkin}. As is common in the non-linear numerical analysis literature, we will attempt to show that the discrete coupled cluster equations \eqref{eq:CC_galerkin} are locally well-posed provided that the associated approximation space $\mathcal{V}$ is rich enough, which, in the present context means that both the size of the atomic basis and the maximal excitation rank of admissible Slater determinants is sufficiently large. Note that since the coupled cluster methodology is targeted at obtaining the \emph{ground state} energy of the electronic Hamiltonian, we will focus only on solutions of the discrete coupled cluster equations corresponding to a ground state energy approximation.

An asymptotic well-posedness analysis of this nature presents two main difficulties: 

\begin{description}
    \item[Difficulty One] Due to the non-linearity of the coupled cluster function, the primary tool that we have at our disposal for the analysis of the discrete problem \eqref{eq:CC_galerkin} is the inverse function theorem for Banach spaces. As we show in the subsequent Section \ref{sec:4}, in order to apply the inverse function theorem in the present setting, we must establish a \textbf{discrete} inf-sup condition for the Fr\'echet derivative of the coupled cluster function on the approximation space $\mathcal{V}$. At first glance, we can simply attempt to replicate the proof for the continuous inf-sup condition (see \cite{Hassan_CC}). Unfortunately, for an arbitrary trial function $\Phi \in \mathcal{V}$ and a general cluster operator~$\mathcal{T}(\Theta)$ constructed from some $\Theta \in \mathcal{V}$, we have that 
    \begin{align*}
    e^{\mathcal{T}(\Theta)} \Phi \notin \mathcal{V},
    \end{align*}
    which causes the continuous inf-sup argument to break down. Moreover, the `extent' of this~non inclusion is decreased only if $\mathcal{V}$ includes Slater determinants of all possible excitation ranks.
    
    \item[Difficulty Two] In the numerical practice, occupied space $\mathscr{R}$ and the corresponding reference determinant $\Psi_0 \in \widehat{\mathcal{H}}^1$ are obtained from the first $N$ eigenfunctions of some \text{discrete} mean-field operator (so-called canonical orbitals). In other words, given an atomic basis, one first performs a mean-field Hartree-Fock (or Kohn-Sham) calculation \emph{in this atomic basis}, and then uses the resulting eigenfunctions to construct the occupied space and reference determinant. This means however, that as the size of the atomic basis changes, the results of the mean-field calculation and therefore the underlying reference determinant also change.  In particular, since cluster operators are defined relative to a fixed reference determinant, it is not possible to compare two different cluster operators defined with respect to two different reference determinants.    
\end{description}

\vspace{2mm}

Partial solutions to \textbf{Difficulty One} shall be the subject of extensive discussion in the forthcoming Section \ref{sec:4}. In the remainder of the current section, we introduce a formalism which, under some assumptions, will allow us to address \textbf{Difficulty Two}. 

\vspace{3mm}

\begin{Notation}[Sequence of Occupied Spaces and Reference Determinants]\label{def:seq_ref}~

    We denote by $\left\{\mathscr{R}_K\right\}_{K\geq N}$ a sequence of occupied spaces of $\mH^1(\R^3)$ as stated in Notation \ref{def:occ_vir}, and for each $K\geq N$, we denote by $\Psi_{0, K} \in \widehat{\mathcal{H}}^1$ the reference determinant corresponding to $\mathscr{R}_K$ as defined through Definition \ref{def:ref_det}.
\end{Notation}


\begin{Notation}[Sequence of Finite-Dimensional Approximation Spaces]\label{def:V_K}~

   Let the sequence of reference determinants $\{\Psi_{0, K}\}_{K\geq N}$ be constructed as in Notation~\ref{def:seq_ref} and for each $K \geq N$, let $\{\Psi_{0, K}\}^{\perp}$ be the $\widehat{\mathcal{L}}^2$-orthogonal complement of $\Psi_{0, K}$ in $\widehat{\mathcal{L}}^2$. We denote by $\{\widetilde{\mathcal{V}}_K\}_{K\geq N}$ a sequence of finite-dimensional subspaces of $\widehat{\mathcal{H}}^1$ with the following properties:
\begin{align*}
 \forall K\geq N \colon& \quad \widetilde{\mathcal{V}}_K \subset \{\Psi_{0,K}\}^{\perp}\qquad \text{and}\\[0.5em] \forall \Phi \in \widehat{\mathcal{H}}^1, ~ \forall \epsilon > 0, ~~ \exists K_0 \in \mathbb{N}| ~~  \forall K\ge K_0, \exists \Phi_K \in \widetilde{\mathcal{V}}_K, ~c_{0, K} \in \mathbb{R} \colon& \quad \left\Vert \Phi -\left(\Phi_K +c_{0, K}\Psi_{0, K}\right)\right\Vert_{\widehat{\mathcal{H}}^1}< \epsilon.
\end{align*}

Additionally, for notational convenience, we define for all $K\geq N$, the subspace $\mathcal{V}_K \subset \widehat{\mathcal{H}}^1$ as follows:
\begin{equation}
    \label{eqV}
\mathcal{V}_K 
   := \widetilde{\mathcal{V}}_K \oplus \text{\rm span}\{\Psi_{0, K}\}  
\end{equation}

\end{Notation}

Consider Notations \ref{def:seq_ref} and \ref{def:V_K}. A valid question that arises is why we define the sequence of reference determinants and approximation spaces only for indices $K\geq N$. This choice is, in fact, motivated by the observation that in the numerical practice, the approximation spaces $\{\mathcal{V}_K\}_{K \geq N}$ are typically defined as the span of some Slater determinants constructed from $K\geq N$ single-particle functions in $\mH^1(\R^3)$. Indeed, such constructions are the subject of detailed discussions in Appendix \ref{sec:spaces}. To ensure a uniform notation therefore, we define the sequence of reference determinants and approximation spaces only for indices $K\geq N$. Of course, this choice has no bearing on the analysis.




Equipped with Notations \eqref{def:seq_ref} and \eqref{def:V_K}, we can now state a \emph{sequence} of discrete coupled cluster equations whose asymptotic well-posedness will be the main object of study in the sequel. 
\vspace{3mm}

\textbf{Sequence of Discrete Coupled Cluster Equations:}~

 Let the sequence of reference determinants $\{\Psi_{0, K}\}_{K\geq N}$ be constructed as in Notation \ref{def:seq_ref} and let the sequence of finite-dimensional subspaces $\{\widetilde{\mathcal{V}}_K\}_{K\geq N}$ be constructed as in Notation \ref{def:V_K}. For each $K \geq N$, we seek a cluster operator $\mathcal{T}(\Theta^*_K)$ generated by some element $\Theta^*_K\in \widetilde{\mathcal{V}}_K$ such that 
\begin{equation}\label{eq:CC_discrete_2}
\forall \Phi_K \in \widetilde{\mathcal{V}}_K\colon \qquad	\left\langle \Phi_K, e^{-\mathcal{T}(\Theta^*_K)}H e^{\mathcal{T}(\Theta^*_K)} \Psi_{0, K}\right\rangle_{\widehat{\mathcal{H}}^{1} \times \widehat{\mathcal{H}}^{-1}}	=0.
\end{equation}

\begin{Notation}
    Recall Definition \ref{def:CC} of the coupled cluster function $\mathcal{f} \colon \widehat{\mathcal{H}}^{1, \perp}_{\Psi_{0}} \rightarrow \left(\widehat{\mathcal{H}}^{1, \perp}_{\Psi_{0}}\right)^*$. Notice that, according to Equation \eqref{eq:12bis}, $\mathcal{f}$ is defined relative to the chosen reference determinant $\Psi_0 \in \widehat{\mathcal{H}}^1$.

    In order to account for the sequence of reference determinant $\{\Psi_{0, K}\}_{K \geq N} \subset \widehat{\mathcal{H}}^1$ introduced via Notation~\ref{def:seq_ref}), we will, in the sequel, write $\mathcal{f}_K \colon \widehat{\mathcal{H}}^{1, \perp}_{\Psi_{0, K}} \rightarrow \left(\widehat{\mathcal{H}}^{1, \perp}_{\Psi_{0, K}}\right)^*$ to denote the coupled cluster function defined relative to the Slater determinant $\Psi_{0, K}$.
\end{Notation}

\section{Analysis of the Discrete Coupled Cluster Equations}\label{sec:4}

Throughout this section, we assume the settings of Sections \ref{sec:2}-\ref{sec:3}. As mentioned at the end of Section~\ref{sec:3}, we are now concerned with the well-posedness and error analysis of the sequence of discrete coupled cluster equations \eqref{eq:CC_discrete_2}. To do so, we will make use of the classical machinery of non-linear numerical analysis, and, in particular, the abstract framework for analysing such problems introduced by Caloz and Rappaz \cite{MR1470227}. Of particular importance is the following fundamental result which is essentially the inverse function theorem in Banach spaces, adapted for the purpose of non-linear numerical analysis.

\begin{theorem}[Inverse Function Theorem in Banach Spaces]\label{Caloz_Rappaz_2}~
	
	Let $\left(\mathcal{X}, \Vert \cdot \Vert_{\mathcal{X}}\right)$ and $\left(\mathcal{Z}, \Vert \cdot \Vert_{\mathcal{Z}}\right)$ be two Banach spaces, let $\mathscr{G}\colon \mathcal{X}\rightarrow \mathcal{Z}$ be a $\mathscr{C}^1$ mapping, let $\bold{u} \in \mathcal{X}$ be such that the Fr\'echet derivative $\mD \mathcal{G}(\bold{u}) \colon \mathcal{X}\rightarrow \mathcal{Z}$ is an isomorphism and define the following quantities:
	\begin{align*}
		\epsilon :=& \Vert \mathcal{G}(\bold{u})\Vert_{\mathcal{Z}},\\[0.25em]
  \gamma_{\mathcal{G}}:=& \Vert \mD \mathcal{G}(\bold{u})^{-1}\Vert_{\mathcal{Z}\to \mathcal{X}},\\[0.25em]
  \forall\alpha \in \mathbb{R}_+\colon \quad ~\mL(\alpha):=& \sup_{\bold{v} \in \overline{\mathbb{B}_{\alpha}(\bold{u})}} \Vert \mD \mathcal{G}(\bold{u})-\mD \mathcal{G}(\bold{v})\Vert_{\mathcal{X} \to \mathcal{Z}}.
	\end{align*}
	
	Under the assumption that $2\gamma_{\mathcal{G}}\mL(2\gamma_{\mathcal{G}}\epsilon)<1$, the closed ball $\overline{\mathbb{B}_{2\gamma_{\mathcal{G}}\epsilon}(\bold{u})} \subset \mathcal{X}$ contains a unique solution~$\bold{u}^*$ to the equation
	\begin{align}\label{eq:Caloz}
		\mathcal{G}(\bold{v})=0,
	\end{align}
	the Fr\'echet derivative $\mD \mathcal{G}(\bold{u}^*) \colon \mathcal{X}\rightarrow \mathcal{Z}$ is an isomorphism with
	\begin{align*}
		\Vert \mD \mathcal{G}(\bold{u}^*)^{-1}\Vert_{\mathcal{Z}\to \mathcal{X}} \leq 2 \gamma_{\mathcal{G}},
	\end{align*}
	and for all $\bold{v} \in \overline{\mathbb{B}_{2\gamma_{\mathcal{G}}\epsilon}(\bold{u})}$ we have the error estimate
	\begin{align*}
		\Vert \bold{u}^*-\bold{v}\Vert_{\mathcal{X}} \leq 2 \gamma_{\mathcal{G}} \Vert \mathcal{G}(\bold{v})\Vert_{\mathcal{Z}}.
	\end{align*}
\end{theorem}
\begin{proof}
	See \cite[Theorem 2.1]{MR1470227}.
\end{proof}

\begin{remark}
    Consider Theorem \ref{Caloz_Rappaz_2}. As pointed out in \cite[Remark 2.1]{MR1470227}, the uniqueness statement of this result can be improved: For any $\alpha \geq 2 \gamma_{\mathcal{G}}\epsilon$ such that $\gamma_{\mathcal{G}} \mL(\alpha) < 1$, the closed ball $\overline{\mathbb{B}_{\alpha}}$ contains a unique solution to Equation \eqref{eq:Caloz}.
\end{remark}

The main focus of our subsequent analysis will be to demonstrate how, under suitable assumptions, the sequence of discrete coupled cluster equations \eqref{eq:CC_discrete_2} satisfies the hypotheses of Theorem \ref{Caloz_Rappaz_2}. For clarity of exposition, we will proceed in the following steps:
\vspace{2mm}

\begin{enumerate}
	\item First, in Section \ref{sec:4a}, we will prove a technical lemma related to the \emph{consistency} of the discrete coupled cluster equations \eqref{eq:CC_discrete_2}. This step will require imposing certain assumptions on the sequence of reference determinants $\{\Psi_{0, K}\}_{K \in \mathbb{N}}$ constructed as in Notation~\ref{def:seq_ref}. 
	
	\item Second, in Section \ref{sec:4b}, we will identify certain classes of approximation spaces $\{\widetilde{\mV}_K\}_{K\geq N}$ for which the stability of the discrete coupled cluster equations \eqref{eq:CC_discrete_2} can be expected to hold, at least asymptotically.
	
	\item In Section \ref{sec:4c}, we will combine the results of the previous two subsections to complete the sought-after demonstration.
\end{enumerate}


    


\subsection{Technical Lemma Pertaining to Consistency}\label{sec:4a}~

For the purpose of the analysis in this section, we require the following assumptions. 

\vspace{0.5cm}

\begin{mdframed}
	\textbf{Assumption A.I: Uniform Overlap.} We assume that the sequence of reference determinants $\{\Psi_{0, K}\}_{K \geq N}$ constructed as in Notation~\ref{def:seq_ref} have uniformly bounded below overlaps with the exact ground state eigenfunction $\Psi^*_{\rm GS} \in \widehat{\mathcal{H}}^1$ of the electronic Hamiltonian $H$, i.e.,
	\begin{align*}
		\exists C >0  \quad \text{s.t.} \quad \forall K \geq N \colon \qquad \vert \left(\Psi_{0, K}, \Psi_{\rm GS}^*\right)\vert > C.
	\end{align*} 
\end{mdframed}

\vspace{0.5cm} 

\begin{mdframed}
	\textbf{Assumption A.II: Uniform Upper Bound.} We assume that the sequence of reference determinants $\{\Psi_{0, K}\}_{K \geq N}$ constructed as in Notation~\ref{def:seq_ref} is uniformly bounded in $\widehat{\mathcal{H}}^1$.
\end{mdframed}

\vspace{0.5cm}

In addition to the above \textbf{Assumptions A.I} and \textbf{A.II}, we of course also require that the continuous coupled cluster equations \eqref{eq:CC} be locally well-posed for any choice of reference determinant $\Psi_{0, K}$. In view of Theorem \ref{thm:CC_der_inv}, a sufficient condition to achieve this local uniqueness is the following.

\vspace{0.2cm}

\begin{mdframed}
	\textbf{Assumption A.III: Existence of a Spectral Gap.} We assume that the ground state eigenvalue~$\mathcal{E}^*_{\rm GS}$ of the  electronic Hamiltonian is simple.
\end{mdframed}

\vspace{0.2cm} 


Equipped with \textbf{Assumptions A.I, A.II, and A.III}, we now prove the following lemma.

 \vspace{-1mm}
\begin{lemma}[Uniform Boundedness and Approximability of Exact Coupled Cluster Zeros]\label{lem:approx}~
	
	Let the sequence of reference determinants $\{\Psi_{0, K}\}_{K\geq N}$ and the sequence of approximation spaces  $\{\widetilde{\mathcal{V}}_K\}_{K\geq N}$ be constructed as in Notation \ref{def:seq_ref} and Notation \ref{def:V_K} respectively, and assume that {\textbf{Assumptions A.I, A.II, and A.III}} hold.
	
	Additionally, for each $K\geq N$, let $\{\Psi_{0, K}\}^{\perp}$ be the $\widehat{\mathcal{L}}^2$ orthogonal complement of $\Psi_{0, K}$ in $\widehat{\mathcal{L}}^2$, let the infinite-dimensional subspace $\widehat{\mathcal{H}}^{1, \perp}_{\Psi_{0, K}}$ be defined as $\widehat{\mathcal{H}}^{1, \perp}_{\Psi_{0, K}}:= \{\Psi_{0, K}\}^{\perp} \cap \widehat{\mathcal{H}}^1$, let the coupled cluster function $\mathcal{f}_K \colon \widehat{\mathcal{H}}^{1, \perp}_{\Psi_{0, K}} \rightarrow \left(\widehat{\mathcal{H}}^{1, \perp}_{\Psi_{0, K}}\right)^*$ be defined through Equation \eqref{eq:12bis}, where we adopt the convention of using the subscript $K$ to signify the dependency of this coupled cluster function on $\Psi_{0, K}$, and let $\Theta_{{K}, \rm GS}^* \in \widehat{\mathcal{H}}^{1, \perp}_{\Psi_{0, K}}$ denote the zero of the coupled cluster function $\mathcal{f}_K$ corresponding to the $\widehat{\mathcal{L}}^2$-normalised ground state eigenfunction $\Psi^*_{\rm GS}$ of the electronic Hamiltonian, i.e.,  
 \begin{align*}
     e^{\mathcal{T}(\Theta_{K, \rm GS}^*)}\Psi_{0, K}= \frac{1}{\left(\Psi^*_{\rm GS}, \Psi_{0, K}\right)_{\widehat{\mathcal{L}}^2}} \Psi^*_{\rm GS}.
 \end{align*}
	Then the following hold:
 \begin{enumerate}
     \item The sequence of coupled cluster zeros $\{\Theta_{{K}, \rm GS}^*\}_{K \geq N}$ is uniformly bounded in $\widehat{\mathcal{H}}^1$.

     \item There exists a sequence of functions $\{\Upsilon_K\}_{K \in \mathbb{N}}$ with each $\Upsilon_K \in \widetilde{\mathcal{V}}_K$ such that
     \begin{align*}
         \lim_{K \to \infty}  \Vert \Upsilon_K - \Theta_{{K}, \rm GS}^*\Vert_{\widehat{\mathcal{H}}^1}=0.
     \end{align*}
 \end{enumerate}

\end{lemma}
\begin{proof}
 We begin by defining, for each $K \geq N$, the constant 
	\begin{align*}
		c_{0, K}:=\frac{1}{\left(\Psi_{0, K}, \Psi^*_{\rm GS}\right)_{\widehat{\mathcal{L}}^2}} \leq C' < \infty,
	\end{align*}
	where the upper bound holds thanks to \textbf{Assumption A.I}. We can therefore define, for every $K \geq N$, 

	\begin{align*}
		\widetilde{\Psi}^*_{{K}, \rm GS} := c_{0, K}\Psi^*_{\rm GS} - \Psi_{0, K} \qquad \text{and it follows that} \qquad \widetilde{\Psi}^*_{{K}, \rm GS}\in  \widehat{\mathcal{H}}^{1, \perp}_{\Psi_{0, K}}.
	\end{align*}
	
	Next, for every $K \geq N$, we denote by $\mathcal{C}_K(\widetilde{\Psi}^*_{{K}, \rm GS}) \colon \widehat{\mathcal{H}}^1 \rightarrow \widehat{\mathcal{H}}^1$ the cluster operator generated by $\widetilde{\Psi}^*_{{K}, \rm GS} \in \widehat{\mathcal{H}}^{1, \perp}_{\Psi_{0, K}}$. It follows that
	\begin{align}\label{eq:revise_1}
		\left({\rm I}+ \mathcal{C}_K(\widetilde{\Psi}^*_{{K}, \rm GS})\right)\Psi_{0, K}= c_{0, K}\Psi^*_{\rm GS}= e^{\mathcal{T}\left(\Theta_{{K}, \rm GS}^*\right)}\Psi_{0, K}.
	\end{align}

	Using now Theorem \ref{thm:1} on the mapping properties of cluster operators together with the fact that the cluster operators $\mathcal{C}_K(\widetilde{\Psi}^*_{{K}, \rm GS}) $ and $\mathcal{T}(\Theta_{{K}, \rm GS}^*)$ are both constructed from functions in $\widehat{\mathcal{H}}^{1, \perp}_{\Psi_{0, K}}$ with respect to the same reference determinant, we conclude that for every $K \geq N$ it holds that
	\begin{align}\nonumber
		\Vert {\rm I}+ \mathcal{C}_K(\widetilde{\Psi}^*_{{K}, \rm GS}) -e^{\mathcal{T}(\Theta_{{K}, \rm GS}^*)} \Vert_{\widehat{\mathcal{H}}^1 \to \widehat{\mathcal{H}}^1} &\leq \beta_{\widehat{\mathcal{H}}} \left\Vert\left({\rm I}+ \mathcal{C}_K(\widetilde{\Psi}^*_{{K}, \rm GS}) -e^{\mathcal{T}(\Theta_{{K}, \rm GS}^*)}\right)\Psi_{0, K} \right\Vert_{\widehat{\mathcal{H}}^1}= 0, 
		\intertext{and therefore, we have the following operator equality:}
		{\rm I} + \mathcal{C}_K(\widetilde{\Psi}^*_{{K}, \rm GS})&= e^{\mathcal{T}(\Theta_{{K}, \rm GS}^*)}. \label{eq:revise_0}
	\end{align}

	Recalling from Theorem \ref{thm:2} the bijection property of the exponential mapping between certain sub-algebras of cluster operators, a direct calculation involving the Taylor series expansion of the logarithm function together with the nil-potency property of cluster operators allows us to conclude that for all~$K \geq N$, it holds that
	\begin{align}\label{eq:lemma_has_1a}
		\mathcal{T}(\Theta_{{K}, \rm GS}^*)&=\sum_{j=1}^N  (-1)^{j+1}\frac{\mathcal{C}_K(\widetilde{\Psi}^*_{{K}, \rm GS})^j}{j}.
	\end{align}

      The proof of the first assertion of the present lemma now follows by making use of Theorem \ref{thm:1}. Indeed, from Equation \eqref{eq:revise_0} we deduce that
    \begin{align*}
        \Vert \mathcal{C}_K(\widetilde{\Psi}^*_{{K}, \rm GS})\Vert_{\widehat{\mathcal{H}}^1 \to \widehat{\mathcal{H}}^1} = \Vert e^{\mathcal{T}(\Theta_{{K}, \rm GS}^*)} -{\rm I} \Vert_{\widehat{\mathcal{H}}^1 \to \widehat{\mathcal{H}}^1} &\leq \beta_{\widehat{\mathcal{H}}} \Vert e^{\mathcal{T}(\Theta_{{K}, \rm GS}^*)}\Psi_{0, K} -\Psi_{0, K}\Vert_{\widehat{\mathcal{H}}^1}\\
        &=\beta_{\widehat{\mathcal{H}}} \Vert c_{0, K}\Psi^*_{\rm GS}-\Psi_{0, K}\Vert_{\widehat{\mathcal{H}}^1}.
    \end{align*}
Applying \textbf{Assumption A.II}, we immediately obtain that the right-hand side of the above inequality is upper bounded uniformly in $K$. 

In view of Equation \eqref{eq:lemma_has_1a} and taking advantage once again of Theorem \ref{thm:1} we finally obtain the existence of a constant $C>0$ that is independent of $K$ such that
\begin{align*}
\Vert \Theta_{{K}, \rm GS}^*\Vert_{\widehat{\mathcal{H}}^1}=  \Vert \mathcal{T}(\Theta_{{K}, \rm GS}^*)\Psi_{0, K}\Vert_{\widehat{\mathcal{H}}^1} \leq      \Vert \mathcal{T}(\Theta_{{K}, \rm GS}^*)\Vert_{\widehat{\mathcal{H}}^1 \to \widehat{\mathcal{H}}^1} =\left\Vert \sum_{j=1}^N  (-1)^{j+1}\frac{\mathcal{C}_K(\widetilde{\Psi}^*_{{K}, \rm GS})^j}{j}\right\Vert_{\widehat{\mathcal{H}}^1 \to \widehat{\mathcal{H}}^1}\leq M.
    \end{align*}

  It now remains to prove the second assertion of the present lemma. To this end, recall from Notation~\ref{def:V_K} that we have imposed an approximability assumption on the sequence of approximation spaces $\big\{\text{span }\{\Psi_{0, K}\}\oplus \widetilde{\mathcal{V}}_K\big\}_{K \geq N}$. This assumption implies, in particular, the existence a sequence of functions $\{\Phi_K\}_{K \geq N}$ with each $\Phi_K \in \text{span }\{\Psi_{0, K}\}\oplus \widetilde{\mathcal{V}}_K$ such that
    \begin{align}\label{eq:revise_conv_1}
        \lim_{K \to \infty} \Vert \Phi_K - \Psi^*_{\rm GS} \Vert_{\widehat{\mathcal{H}}^1}=0.
    \end{align}

In particular, we may define, for each $K$ sufficiently large,
\begin{align*}
    	d_{0, K}:=\frac{1}{\left(\Psi_{0, K}, \Phi_K\right)_{\widehat{\mathcal{L}}^2}} \leq C''< \infty,
\end{align*}
and \textbf{Assumption A.II} together with the approximability result \eqref{eq:revise_conv_1} implies that 
\begin{align}\label{eq:revise_conv_2}
\lim_{K \to \infty} \vert d_{0, K}-c_{0, K}\vert =0.    
\end{align}
In particular, for every $K$ sufficiently large, it holds that
	\begin{align*}
		\widetilde{\Phi}_{{K}}:= d_{0, K}\Phi_K- \Psi_{0, K}  ~\in  \widehat{\mathcal{H}}^{1, \perp}_{\Psi_{0, K}}.
	\end{align*}
	
	Consequently, for any $K$ large enough, we may denote by $\mathcal{C}_K(\widetilde{\Phi}_{{K}}) \colon \widehat{\mathcal{H}}^1 \rightarrow \widehat{\mathcal{H}}^1$ the cluster operator generated by $\widetilde{\Phi}_{{K}}\in \widetilde{\mathcal{V}}_K\subset \widehat{\mathcal{H}}^{1, \perp}_{\Psi_{0, K}}$. It follows that there exists some $\Upsilon_K \in \widetilde{\mathcal{V}}_K \subset \widehat{\mathcal{H}}^{1, \perp}_{\Psi_{0, K}}$ such that for all $K$ sufficiently large
	\begin{align}\label{eq:revise_2}
		\left({\rm I}+ \mathcal{C}_K(\widetilde{\Phi}_{{K}})\right)\Psi_{0, K}= d_{0, K}\Phi_{K}=e^{\mathcal{T}\left(\Upsilon_K\right)}\Psi_{0, K}.
	\end{align}
	
	Using again Theorem \ref{thm:1} on the mapping properties of cluster operators together with the fact that the cluster operators $\mathcal{C}_K(\widetilde{\Phi}_K) $ and $\mathcal{T}(\Upsilon_K)$ are both constructed from functions in $ \widetilde{\mathcal{V}}_K \subset \widehat{\mathcal{H}}^{1, \perp}_{\Psi_{0, K}}$ with respect to the same reference determinant $\Psi_{0, K}$, we conclude that for every $K$ large enough 
	\begin{align*}
		\Vert {\rm I}+ \mathcal{C}_K(\widetilde{\Phi}_K) -e^{\mathcal{T}(\Upsilon_{K})} \Vert_{\widehat{\mathcal{H}}^1 \to \widehat{\mathcal{H}}^1} &\leq \beta_{\widehat{\mathcal{H}}} \left\Vert\left({\rm I}+ \mathcal{C}_K(\widetilde{\Phi}_{K}) -e^{\mathcal{T}(\Upsilon_{K})}\right)\Psi_{0, K} \right\Vert_{\widehat{\mathcal{H}}^1}= 0, 
		\intertext{and therefore,we have the following operator equality:}
		{\rm I} + \mathcal{C}_K(\widetilde{\Phi}_{{K}})&= e^{\mathcal{T}(\Upsilon_{{K}})}.
	\end{align*}
	As argued above, the bijectivity of the exponential mapping given in Theorem \ref{thm:2} together with Taylor series expansion of the logarithm function and the nil-potency property of cluster operators allows us to conclude that for all~$K$ sufficiently large, it holds that
	\begin{align}\label{eq:lemma_has_1b}
		\mathcal{T}(\Upsilon_{K})&=\sum_{j=1}^N  (-1)^{j+1}\frac{\mathcal{C}_K(\widetilde{\Phi}_{{K}})^j}{j}.
	\end{align}

    Notice now that since $\widetilde{\Psi}^*_{{K}, \rm GS}$ and $\Upsilon_{K}$ are both elements of $\widehat{\mathcal{H}}^{1, \perp}_{\Psi_{0, K}}$, Theorem \ref{thm:1} implies that
    \begin{align*}
       \left \Vert \Theta_{{K}, \rm GS}^*- \Upsilon_{K}\right\Vert_{\widehat{\mathcal{H}}^1}=\left \Vert \mathcal{T}(\Theta_{{K}, \rm GS}^*)\Psi_{0, K}- \mathcal{T}(\Upsilon_{K}) \Psi_{0, K}\right\Vert_{\widehat{\mathcal{H}}^1}\leq   \left \Vert \mathcal{T}(\Theta_{{K}, \rm GS}^*)- \mathcal{T}(\Upsilon_{K})\right\Vert_{\widehat{\mathcal{H}}^1 \to \widehat{\mathcal{H}}^1}.
    \end{align*}
    
    We therefore deduce from Equations \eqref{eq:lemma_has_1a} and \eqref{eq:lemma_has_1b} that it suffices to show that
    \begin{align*}
        \lim_{K \to \infty}\Vert \mathcal{C}_K(\widetilde{\Psi}^*_{{K}, \rm GS})-\mathcal{C}_K(\widetilde{\Phi}_{{K}})\Vert_{\widehat{\mathcal{H}}^1 \to \widehat{\mathcal{H}}^1}= 0.
    \end{align*}

To this end, we appeal once again to Theorem \ref{thm:1} and make use of Equations \eqref{eq:revise_1} and \eqref{eq:revise_2} to deduce that for all $K$ sufficiently large
\begin{align*}
    \Vert \mathcal{C}_K(\widetilde{\Psi}^*_{{K}, \rm GS})-\mathcal{C}_K(\widetilde{\Phi}_{{K}})\Vert_{\widehat{\mathcal{H}}^1 \to \widehat{\mathcal{H}}^1} &\leq \beta_{\widehat{\mathcal{H}}} \Vert  \mathcal{C}_K(\widetilde{\Psi}^*_{{K}, \rm GS})\Psi_{0, K}-\mathcal{C}_K(\widetilde{\Phi}_{{K}})\Psi_{0, K}\Vert_{\widehat{\mathcal{H}}^1}\\[0.5em]
    &=\beta_{\widehat{\mathcal{H}}} \left\Vert  \big({\rm I}+\mathcal{C}_K(\widetilde{\Psi}^*_{{K}, \rm GS})\big)\Psi_{0, K}-\big( {\rm I}+\mathcal{C}_K(\widetilde{\Phi}_{{K}}) \big)\Psi_{0, K} \right\Vert_{\widehat{\mathcal{H}}^1}\\[0.5em]
    &=\beta_{\widehat{\mathcal{H}}} \left\Vert  c_{0, K}\Psi^*_{\rm GS}-d_{0, K}\Phi_{K} \right\Vert_{\widehat{\mathcal{H}}^1}.
\end{align*}
Making use of the convergence results \eqref{eq:revise_conv_1} and \eqref{eq:revise_conv_2} therefore yields the required result and completes the proof.
\end{proof}

\subsection{Technical Lemmas Pertaining to Stability}\label{sec:4b}~

Throughout this subsection, we assume the setting of Section \ref{sec:4a}. Our goal now is to identify certain classes of finite-dimensional approximation spaces $\{\widetilde{\mathcal{V}}_{K}\}_{K\geq N}$, defined as in Notation \ref{def:V_K}, for which we can reasonably establish a specific stability property of the discrete coupled cluster equations \eqref{eq:CC_discrete_2}. Essentially, we wish to show that the Fr\'echet derivative $\mD \mathcal{f}_K(\Theta_{K, \rm GS}^*)$ at the ground state zero $\Theta^*_{K, \rm GS}$ of the coupled cluster function $\mathcal{f}_{K} \colon \widehat{\mathcal{H}}^{1, \perp}_{\Psi_{0, K}} \rightarrow (\widehat{\mathcal{H}}^{1, \perp}_{\Psi_{0, K}})^*$ satisfies a discrete inf-sup condition on $\widetilde{\mathcal{V}}_K$, i.e.,
\begin{align}\label{eq:ideal_inf-sup}
	\exists \gamma >0,  \forall K\geq N  \colon ~\inf_{\Upsilon_K \in \widetilde{\mathcal{V}}_K} \sup_{\Phi_K \in \widetilde{\mathcal{V}}_K} \frac{\left\langle  \Upsilon_K, \mD\mathcal{f}_K(\Theta^*_{K, \rm GS})\Phi_K\right\rangle_{\widehat{\mathcal{H}}^1\times \mathcal{H}^{-1}} }{\Vert \Upsilon_K\Vert_{\widehat{\mathcal{H}}^1}\Vert \Phi_K\Vert_{\widehat{\mathcal{H}}^1}} \geq \gamma.
\end{align}

\vspace{3mm}

\begin{remark}\label{rem:discrete_inf_sup_adjoint}
Let us briefly note here that the discrete inf-sup condition \eqref{eq:ideal_inf-sup}, establishes, in fact, that the discrete \underline{adjoint} mapping $\left(\mD\mathcal{f}_K(\Theta^*_{K, \rm GS})\right)^{*}\colon \widetilde{\mathcal{V}}_K\rightarrow \widetilde{\mathcal{V}}_K^*$ is a bounded below mapping, i.e.,
\begin{align*}
    \inf_{\Upsilon_K \in \widetilde{\mathcal{V}}_K} \frac{\Vert \left(\mD\mathcal{f}_K(\Theta^*_{K, \rm GS})\right)^{*} \Upsilon_K\Vert_{\widetilde{\mathcal{V}}_K^*}}{\Vert \Upsilon_K\Vert_{\widehat{\mathcal{H}}^1}} \ge \gamma.
\end{align*}
Since the approximation space $\widetilde{\mathcal{V}}_K$ is finite-dimensional, we deduce that $\left(\mD\mathcal{f}_K(\Theta^*_{K, \rm GS})\right)^{*}\colon \widetilde{\mathcal{V}}_K\rightarrow \widetilde{\mathcal{V}}_K^*$ is a bijection, and therefore so is $\mD\mathcal{f}_K(\Theta^*_{K, \rm GS}) \colon \widetilde{\mathcal{V}}_K\rightarrow \widetilde{\mathcal{V}}_K^*$.
\end{remark}

Unfortunately, establishing a discrete inf-sup condition of the form \eqref{eq:ideal_inf-sup} for arbitrary choices of approximation spaces $\{\widetilde{\mathcal{V}}_{K}\}_{K\geq N}$ and all $K\geq N$ seems out of reach; indeed, we strongly suspect that the condition \eqref{eq:ideal_inf-sup} does not hold in complete generality without further assumptions. On the other hand, if we impose more structure on the choice of the approximation spaces $\{\widetilde{V}_{K}\}_{K\geq N}$, and we restrict ourselves to the asymptotic regime, i.e., when $K$ is sufficiently large (meaning that $\widetilde{\mathcal{V}}_K$ is sufficiently rich), then it becomes possible to establish certain results.

In the current article, we will focus on two classes of finite-dimensional approximation spaces, each with a different additional structure (beyond the definition given in Notation \ref{def:V_K} and \textbf{Assumptions A.I-A.III}). Roughly speaking, the two classes that we consider correspond to the so-called Full-CC approximation spaces and more general discretised CC equations arising from an initial mean-field calculation, and the additional structure that we impose are natural abstractions of the core properties of these two types of coupled cluster approximation spaces (see Appendix \ref{sec:spaces} for a detailed discussion of this point). 

The remainder of this subsection is organised as follows. We will first specify the two types of additional structure that we impose on the finite-dimensional approximation spaces $\{\widetilde{V}_{K}\}_{K\geq N}$ and all $K\geq N$; we will refer to these conditions as \textbf{Structure Assumption B.I} and \textbf{Structure Assumption B.II}. We will then state precisely the discrete inf-sup condition that we require for our analysis and we will show how \textbf{Structure Assumption B.I} or \textbf{Structure Assumption B.II} allows us to prove this result. A detailed discussion pertaining to our earlier claim that \textbf{Structure Assumption B.I} and \textbf{B.II} are satisfied by Full-CC approximation spaces and more general discretised CC equations arising from an initial mean-field calculation respectively, is postponed to Appendix \ref{sec:spaces} although we briefly address this point following our proofs.



\vspace{0.5cm}

\begin{mdframed}
	\textbf{Structure Assumption B.I:} We assume that the sequence of approximation spaces $\{\widetilde{\mathcal{V}}_K\}_{K \in \mathbb{N}}$ constructed as in Notation~\ref{def:V_K} are both \textbf{excitation and de-excitation complete}, i.e., for any $K \geq N$ and any function $\Upsilon_K \in \widetilde{\mathcal{V}}_K \subset \widehat{\mathcal{H}}^{1, \perp}_{\Psi_{0, K}}$, the cluster operator $\mathcal{T}(\Upsilon_K)$ has the property that
	\begin{align}\label{eq:excitation_complete}
		\forall \Phi_K \in {\mathcal{V}}_K=\widetilde{\mathcal{V}}_K\oplus \text{span}\{\Psi_{0, K}\}\colon& \qquad  \mathcal{T}(\Upsilon_K) \Phi_K  \in \widetilde{\mathcal{V}}_K \quad \text{and} \quad \mathcal{T}(\Upsilon_K)^{\dagger} \Phi_K  \in \mathcal{V}_K.
	\end{align} 
\end{mdframed}

\vspace{0.5cm} 

\begin{mdframed}
	\textbf{Structure Assumption B.II:} We assume that the sequence of approximation spaces $\{\mathcal{V}_K\}_{K \geq N}$ constructed as in Notation~\ref{def:V_K} satisfy the following properties: For each $K\geq N$, there exists a symmetric, bounded linear operator $\mathcal{F}_K\colon \widehat{\mathcal{H}}^1\rightarrow \widehat{\mathcal{H}}^{-1}$ and a finite-dimensional subspace $\mathcal{W}_K \subset \widehat{\mathcal{H}}^1$ such that
	\begin{enumerate}
		\item We have $\mathcal{V}_K=\widetilde{\mathcal{V}}_K \oplus \text{span}\{\Psi_{0, K}\} \subset \mathcal{W}_K$, and for any function $\Upsilon_K \in \widetilde{\mathcal{V}}_K\subset \widehat{\mathcal{H}}^{1, \perp}_{\Psi_{0, K}}$, the cluster operator $\mathcal{T}(\Upsilon_K)$ has the property that 
		\begin{align}\label{eq:excitation_complete_2}
			\forall \Phi_K \in  {\mathcal{V}}_K\colon \qquad  \mathcal{T}(\Upsilon_K) \Phi_K  \in \mathcal{W}_K \quad \text{and} \quad \mathcal{T}(\Upsilon_K)^{\dagger} \Phi_K  \in {\mathcal{V}}_K.
		\end{align} 
		
		\item Introducing the restricted operator $\widetilde{\mathcal{F}}_K \colon \mathcal{W}_K \rightarrow \mathcal{W}_K^*$ defined as
		\begin{align}\label{eq:restriction}
			\forall \Phi_K, \Upsilon_K \in \mathcal{W}_K \colon \quad \langle \Upsilon_K, \widetilde{\mathcal{F}}_K\Phi_K \rangle_{\mathcal{W}_K \times \mathcal{W}_K^*}:=\langle \Upsilon_K, \mathcal{F}_K\Phi_K \rangle_{\widehat{\mathcal{H}}^1 \times \widehat{\mathcal{H}}^{-1}},
		\end{align}
        the approximation space ${\mathcal{V}}_K$ is an invariant subspace of $\widetilde{\mathcal{F}}_K$ in the sense that
        \begin{align}\label{eq:Yvon_invariant}
      \forall \Phi_K \in \mathcal{V}_K, ~\forall \Upsilon_K^{\perp} \in \mathcal{V}_K^{\perp}\cap \mathcal{W}_K\colon \qquad       \langle \Upsilon_K^{\perp}, \widetilde{\mathcal{F}}_K\Phi_K \rangle_{\mathcal{W}_K \times \mathcal{W}_K^*}=0.
        \end{align}
        Here, $\mathcal{V}_K^{\perp}= \{\Phi \in \widehat{\mathcal{L}}^2\colon (\Phi, \Upsilon_K)_{\widehat{\mathcal{L}}^2}=0 ~\forall \Upsilon_K \in \mathcal{V}_K\}$.
        
		\item The operator $\mathcal{F}_K$ has continuity constant independent of $K$, and there exists $\Lambda_0 \in \mathbb{R}$ such that $\mathcal{F}_K-\Lambda_0$ is coercive on $\mathcal{W}_K$ with coercivity constant independent of $K$, i.e., $\exists \widetilde{\gamma}>0$ independent of $K$ such that
		\begin{align*}
			\forall \Phi_K  \in \mathcal{W}_K &\colon \quad \Big\langle\Phi_K, (\mathcal{F}_K-\Lambda_0)\Phi_K  \Big\rangle_{\widehat{\mathcal{H}}^1 \times \widehat{\mathcal{H}}^{-1}} \geq \widetilde{\gamma} \Vert \Phi_K\Vert_{\widehat{\mathcal{H}}^1}^2.
		\end{align*}
		
		\item For each $K\geq N$ the operator $\mathcal{U}_K:= H-\mathcal{F}_K$ is a bounded linear mapping from $\widehat{\mathcal{H}}^1$ to $\widehat{\mathcal{L}}^2$ with continuity constant uniformly bounded in $K$ and small in the following sense: denoting by $(\cdot, \cdot)_{\mathcal{F}_K}$ and $\Vert \cdot \Vert_{\mathcal{F}_K}$ the inner product and norm generated by the restricted operator $\widetilde{\mathcal{F}}_K \colon \mathcal{W}_K \rightarrow \mathcal{W}_K^*$, by $\mathbb{P}_{\mathcal{F}_K}\colon \mathcal{W}_K \rightarrow \mathcal{W}_K$ the $( \cdot, \cdot )_{\mathcal{F}_K}$-orthogonal projection operator onto $\mathcal{V}_K$, by $\widetilde{\mathcal{U}}_K \colon \mathcal{W}_K \rightarrow \mathcal{W}_K$ the restriction of the operator $\mathcal{U}_K$, by $\Psi_{{\rm GS}, K}^* \in \mathcal{W}_K$ the normalised ground state eigenfunction of $H$ in $\mathcal{W}_K$, and by $\Theta^{\Pi}_{K, \rm GS} \in \widetilde{\mathcal{V}}_K$, the best approximation with respect to the $\widehat{\mathcal{H}}^1$-inner product of the ground state coupled cluster zero $\Theta_{K, \rm GS}^* \in \widehat{\mathcal{H}}^{1, \perp}_{\Psi_{0, K}}$, we impose the condition that
		\begin{align*}
			\left \Vert \left({\rm I}-\mathbb{P}_{\mathcal{F}_K}\right)\widetilde{\mathcal{U}}_K \mathbb{P}_{\mathcal{F}_K}\right\Vert_{\mathcal{F}_K \to \widehat{\mathcal{L}}^2} &<  \frac{1}{2}\min_{\substack{\Phi_K \in \mathcal{W}_K\\ \Phi_K \in \mathcal{V}_K^{\perp}}} \frac{\vert\langle \Phi_K, (\mathcal{F}_K-\Lambda_0) \Phi_K\rangle_{\widehat{\mathcal{H}}^1 \times \widehat{\mathcal{H}}^{-1}} \vert^{\frac{1}{2}}}{\Vert \Phi_K \Vert_{\widehat{\mathcal{L}}^2}}\\ 
			&\times\min_{\Phi_K \in \{\Psi_{{\rm GS}, K}^*\}^{\perp}\cap\mathcal{W}_K}  \frac{\langle \Phi_K, (H-\mathcal{E}_{\rm GS^*}) \Phi_K\rangle_{\widehat{\mathcal{H}}^1 \times \widehat{\mathcal{H}}^{-1}}}{\Vert \Phi_K \Vert_{\mathcal{F}_K}^2}\\ 
			&\times \left \Vert \left({\rm I}-\mathbb{P}_{\mathcal{F}_K}\right){e^{\mathcal{T}(\Theta^{\Pi}_{K, \rm GS})}}\mathbb{P}_{\mathcal{F}_K}{e^{-\mathcal{T}(\Theta^{\Pi}_{K, \rm GS})}} \mathbb{P}_{\mathcal{F}_K}\right\Vert_{\mathcal{F}_K \to \mathcal{F}_K}^{-1}.
		\end{align*}
	\end{enumerate}
	
\end{mdframed}

\begin{remark}[Preliminary Remarks on the Structure Assumptions]~

A few comments concerning the above structure assumptions are now in order. 

\begin{itemize}
    \item First, it can readily be seen that \textbf{Structure Assumption B.I} is a strong constraint that is not satisfied by the excitation rank-truncated coupled cluster schemes such as CCSD or CCSDT etc. Indeed, in the numerical practice, the only class of coupled cluster approximation spaces that satisfy such a constraint are the Full-CC equations in a finite basis. 

This shortcoming is addressed by the weaker \textbf{Structure Assumption B.II}. We will show later that this alternative assumption is in fact a natural abstraction of combined mean-field excitation-rank truncated coupled cluster equations together with an additional smallness assumption on the operator norm of the mean-field operator. While the smallness assumption is problem-dependent and unlikely to be universally satisfied, preliminary numerical results (shown in Appendix \ref{sec:spaces}) indicate that it is indeed satisfied by the small molecules that we have tested.

\item Second, the idea to split the electronic Hamiltonian as a sum of the mean-field operator and the fluctuation potential is not new in the numerical analysis of the coupled cluster method. Indeed, in the contribution \cite{MR3110488}, as an alternative to Estimate \eqref{eq:loc_mon} of the local monotonicity constant, the authors use this splitting to express the similarity-transformed Hamiltonian in the form 
\begin{align*}
\forall \Theta_K \in \widehat{\mathcal{H}}^{1, \perp}_{\Psi_{0, K}}\colon \qquad&	e^{-\mathcal{T}(\Theta_K)}He^{\mathcal{T}(\Theta_K)}= (\mathcal{F}_K-\Lambda_0)\mathcal{T}(\Theta_K) + \widetilde{U}_{\Theta_K},\\
 \intertext{where}
 \widetilde{U}_{\Theta_K}:=~& e^{-\mathcal{T}(\Theta_K)}He^{\mathcal{T}(\Theta_K)}-(\mathcal{F}_K-\Lambda_0)\mathcal{T}(\Theta_K).
\end{align*}
Using this splitting, the authors show that the Fr\'echet derivative $\mD\mathcal{f}_k(\Theta^*_{K, \rm GS}) \colon \widehat{\mathcal{H}}^{1, \perp}_{\Psi_{0, K}} \rightarrow \big(\widehat{\mathcal{H}}^{1, \perp}_{\Psi_{0, K}}\big)^*$ is a coercive operator provided that the local Lipschitz constant of the mapping $\Theta_K \mapsto \widetilde{U}_{\Theta_K}$ at $\Theta_K = \Theta^*_{K, \rm GS}$ is smaller than the coercivity constant $\widetilde{\gamma}$ of $\mathcal{F}_K -\Lambda_0$ (see Point~3 in \textbf{Structure Assumption B.II} above). 

As we shall see in the subsequent analysis, the approach in the present contribution is different. Unfortunately, it does not seem straightforward to test numerically the validity of the smallness assumption on the local Lipschitz constant of $\widetilde{U}_{\Theta_K}$, and we are therefore unable to perform a comparative study with the approach taken in the present contribution.
\end{itemize}

\end{remark}

\begin{figure}[ht]
	\begin{center}
		\includegraphics[width=1\textwidth, trim={0cm, 1cm, 0cm, 1cm},clip=true]{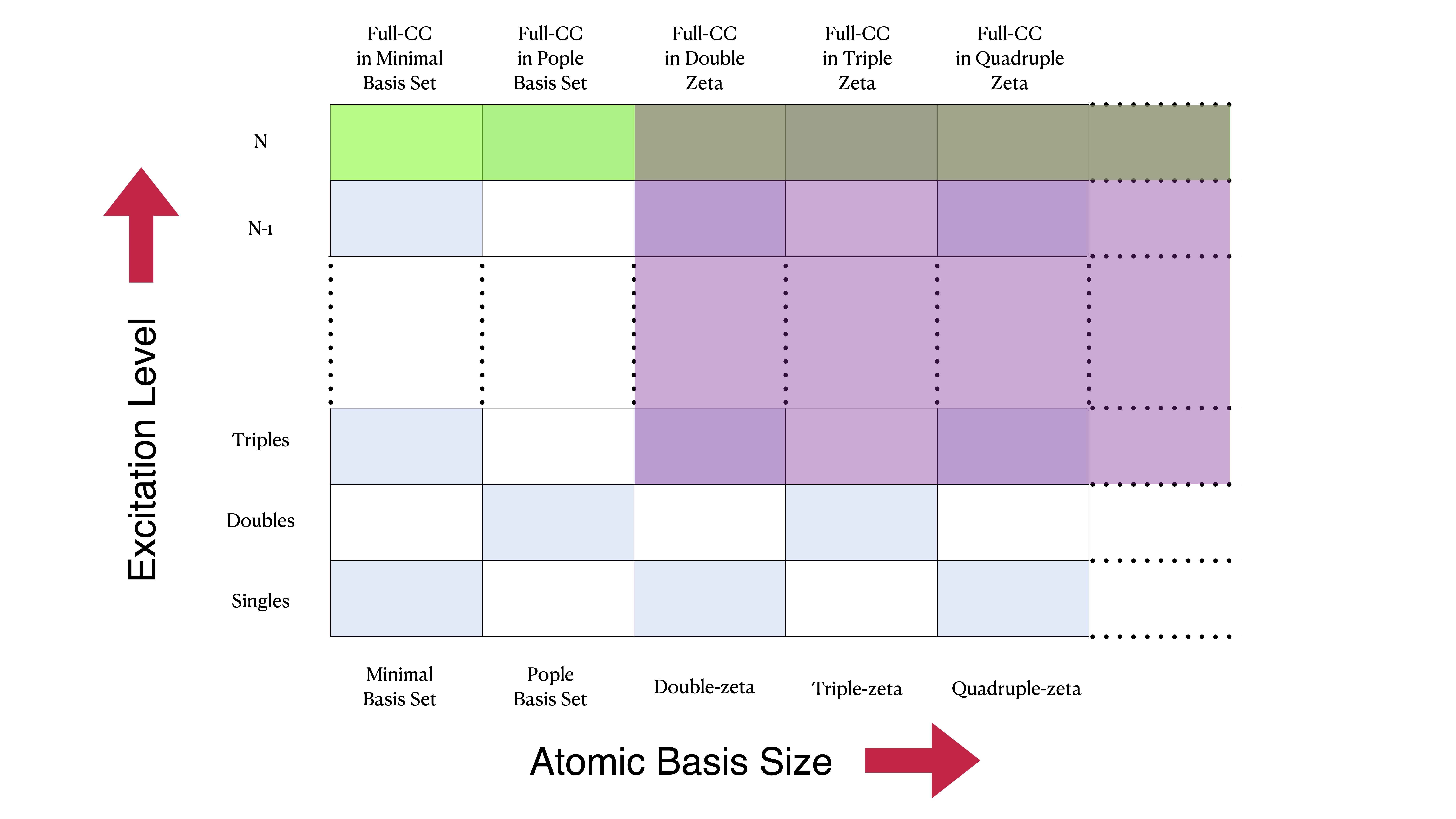}
	\end{center}
	\caption{Graphical depiction of region of validity of \textbf{Structure Assumption B.I} (in green) and \textbf{Structure Assumption B.II} (in magenta). Note that the exact size of the magenta region of validity, i.e., whether it begins at the double-zeta level or quadruple-zeta level etc., or whether it includes quadruples or triples etc., depends on the properties of the mean-field operator (c.f., the proof of Lemma \ref{lem:inf-sup}).}
	\label{fig:2}
\end{figure}

Equipped with either \textbf{Structure Assumption B.I} or \textbf{Structure Assumption B.II}, we are now ready to state the precise discrete inf-sup condition that we wish to prove.

\vspace{5mm}
\begin{lemma}[Discrete inf-sup condition on Fr\'echet Derivative of Coupled Cluster Function]\label{lem:inf-sup}~
	
	Let the sequence of reference determinants $\{\Psi_{0, K}\}_{K\geq N}$ and the sequence of approximation spaces  $\{\widetilde{\mathcal{V}}_K\}_{K\geq N}$ be constructed as in Notation \ref{def:seq_ref} and Notation \ref{def:V_K} respectively, and assume that {\textbf{Assumptions A.I-A.III}} hold.
	
	Additionally, for each $K\geq N$, let $\{\Psi_{0, K}\}^{\perp}$ be the $\widehat{\mathcal{L}}^2$ orthogonal complement of $\Psi_{0, K}$ in $\widehat{\mathcal{L}}^2$, let the infinite-dimensional subspace $\widehat{\mathcal{H}}^{1, \perp}_{\Psi_{0, K}}$ be defined as $\widehat{\mathcal{H}}^{1, \perp}_{\Psi_{0, K}}:= \{\Psi_{0, K}\}^{\perp} \cap \widehat{\mathcal{H}}^1$, let the coupled cluster function $\mathcal{f}_K \colon \widehat{\mathcal{H}}^{1, \perp}_{\Psi_{0, K}} \rightarrow \left(\widehat{\mathcal{H}}^{1, \perp}_{\Psi_{0, K}}\right)^*$ be defined through Equation \eqref{eq:12bis}, where we adopt the convention of using the subscript $K$ to signify the dependency of this coupled cluster function on $\Psi_{0, K}$, let $\Theta_{{K}, \rm GS}^*\in \widehat{\mathcal{H}}^{1, \perp}_{\Psi_{0, K}}$ denote the zero of the coupled cluster function $\mathcal{f}_K$ corresponding to the $\widehat{\mathcal{L}}^2$-normalised ground state eigenfunction $\Psi^*_{\rm GS}$ of the electronic Hamiltonian, i.e.,  
 \begin{align*}
     e^{\mathcal{T}(\Theta_{K, \rm GS}^*)}\Psi_{0, K}= \frac{1}{\left(\Psi^*_{\rm GS}, \Psi_{0, K}\right)_{\widehat{\mathcal{L}}^2}} \Psi^*_{\rm GS},
 \end{align*}
 and let $\mD \mathcal{f}_K(\Theta^*_{K, \rm GS}) \colon \widehat{\mathcal{H}}^{1, \perp}_{\Psi_{0, K}} \rightarrow \left(\widehat{\mathcal{H}}^{1, \perp}_{\Psi_{0, K}}\right)^*$ denote the Fr\'echet derivative of the coupled cluster function $\mathcal{f}_K$ evaluated at $\Theta^*_{K, \rm GS}$.
	
	Under the constraint that either \textbf{Structure Assumption B.I} or \textbf{Structure Assumption B.II} holds, we have the existence of a $K_0\geq N$  and a constant $\gamma >0$ such that
	\begin{align}\label{eq:inf-sup_old}
		\forall K\geq K_0  \colon ~\inf_{\Upsilon_K \in \widetilde{\mathcal{V}}_K} \sup_{\Phi_K \in \widetilde{\mathcal{V}}_K} \frac{\left\langle  \Upsilon_K, \mD\mathcal{f}_K(\Theta^*_{K, \rm GS})\Phi_K\right\rangle_{\widehat{\mathcal{H}}^{1, \perp}_{\Psi_{0, K}}\times \left(\widehat{\mathcal{H}}^{1, \perp}_{\Psi_{0, K}}\right)^*}}{\Vert \Upsilon_K\Vert_{\widehat{\mathcal{H}}^1}\Vert \Phi_K\Vert_{\widehat{\mathcal{H}}^1}} \geq \gamma.
	\end{align}    
\end{lemma}

We will provide separate proofs for each of the two structural assumptions. Before proceeding to these proofs however, let us first demonstrate a simple, auxiliary result which will considerably simplify our arguments.

\begin{lemma}[Auxiliary result to simplify the proof of Lemma \ref{lem:inf-sup}]\label{lem:inf-sup_aux}~
	
	Consider the setting and statement of Lemma \ref{lem:inf-sup}, let $\mathcal{E}^*_{\rm GS}$ denote the ground state energy of the electronic Hamiltonian $H \colon \widehat{\mathcal{H}}^1\rightarrow \widehat{\mathcal{H}}^{-1}$, and for all $K\geq N$ let $\Theta^{\Pi}_{K, \rm GS}\in \widetilde{\mathcal{V}}_K$ denote the best approximation with respect to the $\widehat{\mathcal{H}}^1$-inner product of the exact ground state zero $\Theta_{K, \rm GS}^* \in \widehat{\mathcal{H}}^{1, \perp}_{\Psi_{0, K}}$ of the coupled cluster function $\mathcal{f}_K\colon \widehat{\mathcal{H}}^{1, \perp}_{\Psi_{0, K}} \rightarrow \left(\widehat{\mathcal{H}}^{1, \perp}_{\Psi_{0, K}}\right)^*$. Then the discrete inf-sup condition \eqref{eq:inf-sup_old} holds if there exists $\widehat{K}_0\geq N$ and a constant $\widehat{\gamma} >0$ such that
	\begin{align}\label{eq:inf-sup}
		\forall K\geq \widehat{K}_0  \colon ~\inf_{\Upsilon_K \in \widetilde{\mathcal{V}}_K} \sup_{\Phi_K \in \widetilde{\mathcal{V}}_K} \frac{\left\langle  \Upsilon_K,  e^{-\mathcal{T}(\Theta^{\Pi}_{K, \rm GS})}(H-\mathcal{E}^*_{\rm GS}) e^{\mathcal{T}(\Theta^{\Pi}_{K, \rm GS})}\Phi_K\right\rangle_{\widehat{\mathcal{H}}^1\times \mathcal{H}^{-1}} }{\Vert \Upsilon_K\Vert_{\widehat{\mathcal{H}}^1}\Vert \Phi_K\Vert_{\widehat{\mathcal{H}}^1}} > \widehat{\gamma}.
	\end{align} 
\end{lemma}
\begin{proof}
	Let $K \geq N$ be fixed. It has been demonstrated in detail in \cite[Corollary 27]{Hassan_CC} that the Fr\'echet derivative $\mD \mathcal{f}_K(\Theta^*_{K, \rm GS}) \colon \widehat{\mathcal{H}}^{1, \perp}_{\Psi_{0, K}} \rightarrow \left(\widehat{\mathcal{H}}^{1, \perp}_{\Psi_{0, K}}\right)^*$ of the coupled cluster function $\mathcal{f}_K\colon \widehat{\mathcal{H}}^{1, \perp}_{\Psi_{0, K}} \rightarrow \left(\widehat{\mathcal{H}}^{1, \perp}_{\Psi_{0, K}}\right)^*$ evaluated at $\Theta^*_{K, \rm GS}$ has the expression
	\begin{align*}
		\forall \Upsilon, \Phi\in \widehat{\mathcal{H}}^{1, \perp}_{\Psi_{0, K}}\colon \quad \left\langle  \Upsilon, \mD\mathcal{f}_K(\Theta^*_{K, \rm GS})\Phi\right\rangle_{\widehat{\mathcal{H}}^{1, \perp}_{\Psi_{0, K}} \times \left(\widehat{\mathcal{H}}^{1, \perp}_{\Psi_{0, K}} \right)^*} = \langle\Upsilon, e^{-\mathcal{T}(\Theta^*_{K, \rm GS})}(H-\mathcal{E}^*_{\rm GS}) e^{\mathcal{T}(\Theta^*_{K, \rm GS})}\Phi\rangle_{\widehat{\mathcal{H}}^1\times \widehat{\mathcal{H}}^{-1}}.
	\end{align*}

	On the other hand, Lemma \ref{lem:approx} implies that the best approximation $\Theta^{\Pi}_{K, \rm GS}\in \widetilde{\mathcal{V}}_K$ of the exact ground state zero $\Theta_{K, \rm GS}^* \in \widehat{\mathcal{H}}^{1, \perp}_{\Psi_{0, K}}$ satisfies
	\begin{align}\label{eq:Yvon_33}
		\lim_{K \to \infty}\Vert \Theta^{\Pi}_{K, \rm GS} -\Theta_{K, \rm GS}^* \Vert_{\widehat{\mathcal{H}}^1}=0.
	\end{align}
	
	Recalling therefore from Theorem \ref{thm:1} that the mapping $\widehat{\mathcal{H}}^{1, \perp}_{\Psi_{0, K}} \ni \Theta \mapsto \mathcal{T}(\Theta)$ is bounded with continuity constant depending only on $N$ and $\Vert \Psi_{0, K}\Vert_{\widehat{\mathcal{H}}^1}$, and using the fact that the exponential is a $\mathscr{C}^{\infty}$ mapping on the algebra of cluster operators (see Theorem \ref{thm:2}) yields the required result.
\end{proof}

In view of Lemma \ref{lem:inf-sup_aux}, it suffices to establish the validity of the estimate \eqref{eq:inf-sup} in order to prove the discrete inf-sup Lemma \ref{lem:inf-sup}. The primary motivation for proceeding in this fashion is that the underlying exponential cluster operators appearing in Estimate \eqref{eq:inf-sup} are constructed from functions in the finite-dimensional approximation spaces $\{\widetilde{\mathcal{V}}_K\}_{K\geq N}$. Consequently, we can immediately take advantage of the relations \eqref{eq:excitation_complete} and \eqref{eq:excitation_complete_2} appearing in \textbf{Structure Assumption B.I} and \textbf{Structure Assumption B.II} respectively.

\vspace{2mm}
We begin with the simpler case of \textbf{Structure Assumption B.I}. Since this proof closely resembles that of \cite[Theorem 42]{Hassan_CC}, we shall be succinct in our arguments.

\vspace{3mm}
\begin{proof}[Proof of Lemma \ref{lem:inf-sup} for \textbf{Structure Assumption B.I}]~
	
	We will establish the validity of the estimate \eqref{eq:inf-sup}. To this end, let $K\geq N$ be fixed, let $\mathbb{P}_{0, K} \colon \widehat{\mathcal{H}}^1 \rightarrow \widehat{\mathcal{H}}^1$ denote the $\widehat{\mathcal{L}}^2$-orthogonal projector onto span$\{\Psi_{0, K}\}$ and let $\mathbb{P}_{0, K}^{\perp} \colon \widehat{\mathcal{H}}^1 \rightarrow \widehat{\mathcal{H}}^1$ denote its complement, i.e., $\mathbb{P}_{0, K}^{\perp}:={\rm I}-\mathbb{P}_{0, K}$. 
	
	Then for any arbitrary $\Upsilon_K \in \widetilde{\mathcal{V}}_K \subset \widehat{\mathcal{H}}^{1, \perp}_{\Psi_{0, K}} $, we may define $\Phi_K \in  \widehat{\mathcal{H}}^{1, \perp}_{\Psi_{0, K}}$ as
	\begin{align}\label{eq:lemma_inf-sup_0}
		\Phi_K:= \mathbb{P}_{0, K}^{\perp}e^{-\mathcal{T}(\Theta^{\Pi}_{K, \rm GS})} e^{-\mathcal{T}(\Theta^{\Pi}_{K, \rm GS})^{\dagger}}\Upsilon_K,
	\end{align} 
 where, as in Lemma \ref{lem:inf-sup_aux}, $\Theta^{\Pi}_{K, \rm GS}\in \widetilde{\mathcal{V}}_K$ denotes the best approximation with respect to the $\widehat{\mathcal{H}}^1$-inner product of the exact ground state zero $\Theta_{K, \rm GS}^* \in \widehat{\mathcal{H}}^{1, \perp}_{\Psi_{0, K}}$ of the coupled cluster function $\mathcal{f}_K\colon \widehat{\mathcal{H}}^{1, \perp}_{\Psi_{0, K}} \rightarrow \left(\widehat{\mathcal{H}}^{1, \perp}_{\Psi_{0, K}}\right)^*$.
 
	We claim that in fact $\Phi_K\in \widetilde{\mathcal{V}}_K$. Indeed, it follows directly from \textbf{Structure Assumption B.I} that $e^{-\mathcal{T}(\Theta^{\Pi}_{K, \rm GS})} e^{-\mathcal{T}(\Theta^{\Pi}_{K, \rm GS})^{\dagger}}\Upsilon_K\in {\mathcal{V}}_K$. Using therefore the fact that $\widetilde{\mathcal{V}}_K$ is $\widehat{\mathcal{L}}^2$-orthogonal to $\text{span}\{\Psi_{0, K}\}$ by construction (recall Notation \ref{def:V_K}) proves the claim.
	
	We can thus deduce that (c.f., Estimate \eqref{eq:inf-sup})
	\begin{align}\nonumber
		&\left\langle  \Upsilon_K,  e^{-\mathcal{T}(\Theta^{\Pi}_{K, \rm GS})} (H-\mathcal{E}^*_{\rm GS})e^{\mathcal{T}(\Theta^{\Pi}_{K, \rm GS})}\Phi_K\right\rangle_{\widehat{\mathcal{H}}^1\times \mathcal{H}^{-1}}\\ \nonumber
		=&  \left\langle  \Upsilon_K,  e^{-\mathcal{T}(\Theta^{\Pi}_{K, \rm GS})}(H-\mathcal{E}^*_{\rm GS}) e^{\mathcal{T}(\Theta^{\Pi}_{K, \rm GS})}\mathbb{P}_{0, K}^{\perp}e^{-\mathcal{T}(\Theta^{\Pi}_{K, \rm GS})} e^{-\mathcal{T}(\Theta^{\Pi}_{K, \rm GS})^{\dagger}}\Upsilon_K\right\rangle_{\widehat{\mathcal{H}}^1\times \mathcal{H}^{-1}}\\ \nonumber
		=&\underbrace{\left\langle  \Upsilon_K,  e^{-\mathcal{T}(\Theta^{\Pi}_{K, \rm GS})}(H-\mathcal{E}^*_{\rm GS}) e^{-\mathcal{T}(\Theta^{\Pi}_{K, \rm GS})^{\dagger}}\Upsilon_K\right\rangle_{\widehat{\mathcal{H}}^1\times \mathcal{H}^{-1}}}_{:=\rm (I)}\\ \label{eq:forgot}
		-& \underbrace{\left\langle  \Upsilon_K,  e^{-\mathcal{T}(\Theta^{\Pi}_{K, \rm GS})}(H-\mathcal{E}^*_{\rm GS}) e^{\mathcal{T}(\Theta^{\Pi}_{K, \rm GS})}\mathbb{P}_{0, K}e^{-\mathcal{T}(\Theta^{\Pi}_{K, \rm GS})} e^{-\mathcal{T}(\Theta^{\Pi}_{K, \rm GS})^{\dagger}}\Upsilon_K\right\rangle_{\widehat{\mathcal{H}}^1\times \mathcal{H}^{-1}}}_{:=\rm (II)}, 
	\end{align}
	where we remind the reader that, as in Lemma \ref{lem:inf-sup_aux}, $\mathcal{E}^*_{\rm GS}$ denotes the ground state eigenvalue of the electronic Hamiltonian $H \colon \widehat{\mathcal{H}}^1\rightarrow \widehat{\mathcal{H}}^{-1}$.

	Let us first consider the term (II). A straightforward calculation reveals that
	\begin{align}\nonumber
		{\rm (II)}&=   \left(\Psi_{0, K}, e^{-\mathcal{T}(\Theta^{\Pi}_{K, \rm GS})} e^{-\mathcal{T}(\Theta^{\Pi}_{K, \rm GS})^{\dagger}}\Upsilon_K\right)_{\widehat{\mathcal{L}}^2}\left\langle  \Upsilon_K,  e^{-\mathcal{T}(\Theta^{\Pi}_{K, \rm GS})}(H-\mathcal{E}^*_{\rm GS}) e^{\mathcal{T}(\Theta^{\Pi}_{K, \rm GS})}\Psi_{0, K}\right\rangle_{\widehat{\mathcal{H}}^1\times \mathcal{H}^{-1}}\\[0.5em] \label{eq:lemma_inf-sup_1}
		&\leq \Vert \Psi_{0, K}\Vert_{\widehat{\mathcal{H}}^1} \Vert e^{-\mathcal{T}(\Theta^{\Pi}_{K, \rm GS})}\Vert_{\widehat{\mathcal{H}}^1 \to \widehat{\mathcal{H}}^1} \Vert e^{-\mathcal{T}(\Theta^{\Pi}_{K, \rm GS})^{\dagger}}\Upsilon_K\Vert^2_{\widehat{\mathcal{H}}^1} \Vert (H-\mathcal{E}^*_{\rm GS}) e^{\mathcal{T}(\Theta^{\Pi}_{K, \rm GS})}\Psi_{0, K}\Vert_{\widehat{\mathcal{H}}^{-1}}.
	\end{align}
	
	Note that since $\Theta_{K, \rm GS}^* $ is by definition the ground state zero of the coupled cluster function $\mathcal{f}_K\colon \widehat{\mathcal{H}}^{1, \perp}_{\Psi_{0, K}} \rightarrow \left(\widehat{\mathcal{H}}^{1, \perp}_{\Psi_{0, K}}\right)^*$, it follows that $H e^{\mathcal{T}(\Theta_{K, \rm GS}^*)} \Psi_{0, K} = \mathcal{E}^*_{\rm GS} e^{\mathcal{T}(\Theta_{K, \rm GS}^*)} \Psi_{0, K}$ (see, e.g., \cite[Theorem 5.3]{MR3021693}) and we therefore have
	\begin{align}\label{eq:lemma_inf-sup_2}
		\Vert (H-\mathcal{E}^*_{\rm GS}) e^{\mathcal{T}(\Theta^{\Pi}_{K, \rm GS})}\Psi_{0, K}\Vert_{\widehat{\mathcal{H}}^{-1}} \leq \Vert (H-\mathcal{E}^*_{\rm GS}) \left(e^{\mathcal{T}(\Theta^{\Pi}_{K, \rm GS})}-e^{\mathcal{T}(\Theta_{K, \rm GS}^*)}\right)\Psi_{0, K}\Vert_{\widehat{\mathcal{H}}^{-1}}.
	\end{align}


	{On the one hand, Lemma \ref{lem:approx} implies that the best approximation $\Theta^{\Pi}_{K, \rm GS}\in \widetilde{\mathcal{V}}_K$ of the exact ground state zero $\Theta_{K, \rm GS}^* \in \widehat{\mathcal{H}}^{1, \perp}_{\Psi_{0, K}}$ satisfies $\lim_{K \to \infty}\Vert \Theta^{\Pi}_{K, \rm GS} -\Theta_{K, \rm GS}^* \Vert_{\widehat{\mathcal{H}}^1}=0$ (recall Equation \eqref{eq:Yvon_33}).} On the other hand, we can recall from Theorem \ref{thm:1} that the mapping $\widehat{\mathcal{H}}^{1, \perp}_{\Psi_{0, K}} \ni \Theta \mapsto \mathcal{T}(\Theta)$ is bounded with continuity constant depending only on $N$ and $\Vert \Psi_{0, K}\Vert_{\widehat{\mathcal{H}}^1}$. Since the exponential is a $\mathscr{C}^{\infty}$ mapping on the algebra of cluster operators (see Theorem \ref{thm:2}), we can take advantage of \textbf{Assumption A.II} and thus deduce from Inequalities \eqref{eq:lemma_inf-sup_1} and \eqref{eq:lemma_inf-sup_2} that there exists a constant ${\epsilon}_K^{\rm (II)}\geq0$ such that 
	\begin{equation}\label{eq:lemma_inf-sup_3}
		{\rm (II)}\leq {\epsilon}_K^{\rm (II)}\Vert e^{-\mathcal{T}(\Theta^{\Pi}_{K, \rm GS})^{\dagger}}\Upsilon_K\Vert_{\widehat{\mathcal{H}}^1}^2 \quad \text{with} \quad \lim_{K \to \infty}{\epsilon}_K^{\rm (II)}=0.
	\end{equation}
	
	It remains to estimate the term (I). To do so, let us first write
	\begin{align*}
		{\rm (I)}=&\left\langle  e^{-\mathcal{T}(\Theta^{\Pi}_{K, \rm GS})^{\dagger}}\Upsilon_K,  (H-\mathcal{E}^*_{\rm GS}) e^{-\mathcal{T}(\Theta^{\Pi}_{K, \rm GS})^{\dagger}}\Upsilon_K\right\rangle_{\widehat{\mathcal{H}}^1\times \mathcal{H}^{-1}}.
	\end{align*}
	Notice that thanks to \textbf{Assumption A.III}, the shifted Hamiltonian $H-\mathcal{E}_{\rm GS}^* \colon \widehat{\mathcal{H}}^1 \rightarrow  \widehat{\mathcal{H}}^{-1}$ is coercive on $\{\Psi^*_{\rm GS}\}^{\perp}\cap \widehat{\mathcal{H}}^1$ with a coercivity constant that we denote by $\gamma_{\rm GS}^*>0$, where we remind the reader that $\Psi^*_{\rm GS} \in \widehat{\mathcal{H}}^1$ is the $\widehat{\mathcal{L}}^2$-normalised ground state eigenfunction of the electronic Hamiltonian $H$. 
	
	In order to take advantage of this coercivity result, let us first observe that the function $e^{-\mathcal{T}(\Theta^*_{K, \rm GS})^{\dagger}}\Upsilon_K \in {\mathcal{V}}_K$ and is $\widehat{\mathcal{L}}^2$-orthogonal to the ground state $\Psi^*_{\rm GS}$. Indeed, since the ground state zero $\Theta_{K, \rm GS}^*$ of the coupled cluster function $\mathcal{f}_K$ satisfies by definition that $e^{\mathcal{T}(\Theta_{K, \rm GS}^*)} \Psi_{0, K}= \frac{1}{(\Psi_{0, K}, \Psi_{\rm GS}^*)_{\widehat{\mathcal{L}}^2} }\Psi_{\rm GS}^*$, we deduce that
	\begin{align*}
		\frac{1}{(\Psi_{0, K}, \Psi_{\rm GS}^*)_{\widehat{\mathcal{L}}^2}}\left(e^{-\mathcal{T}(\Theta^*_{K, \rm GS})^{\dagger}}\Upsilon_K, \Psi^*_{\rm GS}\right)_{\widehat{\mathcal{L}}^2}=\left(e^{-\mathcal{T}(\Theta^*_{K, \rm GS})^{\dagger}}\Upsilon_K, e^{\mathcal{T}(\Theta_{K, \rm GS}^*)} \Psi_{0, K}\right)_{\widehat{\mathcal{L}}^2}= \left(\Upsilon_K, \Psi_{0, K}\right)_{\widehat{\mathcal{L}}^2}=0,
	\end{align*}
	where the last step follows since $\Upsilon_K \in \widetilde{\mathcal{V}}_K$, which is $\widehat{\mathcal{L}}^2$-orthogonal to $\Psi_{0, K}$ (see Notation~\ref{def:V_K}).
	
	Consequently, using again the fact that $\lim_{K \to \infty}\Vert \Theta^{\Pi}_{K, \rm GS} -\Theta_{K, \rm GS}^* \Vert_{\widehat{\mathcal{H}}^1}=0$ (see Lemma \ref{lem:approx}), we deduce, for all $K$ sufficiently large, the existence of a constant $\gamma_{\rm GS}^{(K)}>0$ such that 
	\begin{equation}\label{eq:lemma_inf-sup_4}
		{\rm (I)}\geq \gamma_{\rm GS}^{(K)} \Vert e^{-\mathcal{T}(\Theta^{\Pi}_{K, \rm GS})^{\dagger}}\Upsilon_K\Vert_{\widehat{\mathcal{H}}^1}^2 \quad \text{and} \quad \lim_{K \to \infty}\gamma_{\rm GS}^{(K)}= \gamma_{\rm GS}^*.
	\end{equation}
	
	Combining now Inequalities \eqref{eq:lemma_inf-sup_3} and \eqref{eq:lemma_inf-sup_4} and recalling our choice of test function $\Phi_K$ given by Equation \eqref{eq:lemma_inf-sup_0}, we see that we have shown the existence of constant $\widetilde{\gamma}>0$ and $\widehat{K}_0$ such that for all $K\geq \widehat{K}_0$ and all $\Upsilon_K \in \widetilde{\mathcal{V}}_K$ it holds that
	\begin{align*}
		\sup_{\Phi_K \in \widetilde{\mathcal{V}}_K} \frac{\left\langle  \Upsilon_K,  e^{-\mathcal{T}(\Theta^{\Pi}_{K, \rm GS})}(H-\mathcal{E}^*_{\rm GS}) e^{\mathcal{T}(\Theta^{\Pi}_{K, \rm GS})}\Phi_K\right\rangle_{\widehat{\mathcal{H}}^1\times \mathcal{H}^{-1}} }{\Vert \Upsilon_K\Vert_{\widehat{\mathcal{H}}^1}\Vert \Phi_K\Vert_{\widehat{\mathcal{H}}^1}}&\geq \widetilde{\gamma}\frac{\Vert e^{-\mathcal{T}(\Theta^{\Pi}_{K, \rm GS})^{\dagger}}\Upsilon_K\Vert_{\widehat{\mathcal{H}}^1}^2}{\Vert \Upsilon_K\Vert_{\widehat{\mathcal{H}}^1}\Vert \mathbb{P}_{0, K}^{\perp}e^{-\mathcal{T}(\Theta^{\Pi}_{K, \rm GS})} e^{-\mathcal{T}(\Theta^{\Pi}_{K, \rm GS})^{\dagger}}\Upsilon_K\Vert_{\widehat{\mathcal{H}}^1}}\\[1em]
		&\geq\widetilde{\gamma}\frac{1}{\Vert e^{\mathcal{T}(\Theta^{\Pi}_{K, \rm GS})^{\dagger}}\Vert_{\widehat{\mathcal{H}}^1 \to \widehat{\mathcal{H}}^1}\Vert \mathbb{P}_{0, K}^{\perp}e^{-\mathcal{T}(\Theta^{\Pi}_{K, \rm GS})}\Vert_{\widehat{\mathcal{H}}^1 \to \widehat{\mathcal{H}}^1}}.
	\end{align*}
 
	Taking now the infimum over all $\Upsilon_K \in \widetilde{\mathcal{V}}_K$ therefore completes the proof. 
	
\end{proof}

\begin{remark}[Method of proof of Lemma \ref{lem:inf-sup} for \textbf{Structure Assumption B.I} ]\label{rem:B.I}~
	
	A few comments on the above proof are now in order. First, notice that the main reason for introducing \textbf{Structure Assumption B.I} was to allow us to construct an appropriate test function belonging to the finite-dimensional approximation space $\widetilde{\mathcal{V}}_K$. Second, many of our subsequent arguments relied on the approximability of the exact ground state zero $\Theta^*_{K, \rm GS} \in \widehat{\mathcal{H}}^{1, \perp}_{\Psi_{0, K}}$ of the coupled cluster function in the approximation space $\widetilde{\mathcal{V}}_K$; indeed, the same approximability property was used in the proof of Lemma~\ref{lem:inf-sup_aux}. Finally, we used the coercivity of the shifted electronic Hamiltonian $H-\mathcal{E}^*_{\rm FCI}$ on $\{\Psi^*_{\rm GS}\}^{\perp}\cap \widehat{\mathcal{H}}^1$ where $\Psi^*_{\rm GS}$ denotes the $\widehat{\mathcal{L}}^2$-normalised ground state eigenfunction of $H \colon \widehat{\mathcal{H}}^1\rightarrow \widehat{\mathcal{H}}^{-1}$ with corresponding eigenvalue $\mathcal{E}^*_{\rm GS}$. In particular, the question of whether the discrete inf-sup condition \eqref{eq:inf-sup} holds for a particular choice of finite-dimensional approximation space $\widetilde{\mathcal{V}}_K$ depends on how well the exact ground state coupled cluster zero $\Theta^*_{K, \rm GS} $ is approximated in $\widetilde{\mathcal{V}}_K$. Finally, let us point out that the numerator $\widetilde{\gamma}>0$ in the discrete inf-sup constant that we derived has the property that $\lim_{K \to \infty}\widetilde{\gamma}=\gamma_{\rm GS}^*$, where $\gamma^*_{\rm GS}$ denotes the coercivity constant of the shifted Hamiltonian $H-\mathcal{E}_{\rm GS}^* \colon \widehat{\mathcal{H}}^1 \rightarrow  \widehat{\mathcal{H}}^{-1}$ on $\{\Psi^*_{\rm GS}\}^{\perp}\cap \widehat{\mathcal{H}}^1$. 
\end{remark}

As mentioned previously, the so-called Full-CC approximation spaces satisfy \textbf{Structure Assumption B.I}. This is the subject of further discussion in Appendix \ref{sec:spaces}.

We next turn our attention to the proof of Lemma \ref{lem:inf-sup} in the case of our second structural assumption. This proof is considerably more technical. \vspace{3mm}

\begin{proof}[Proof of Lemma \ref{lem:inf-sup} for \textbf{Structure Assumption B.II}]~
	
	As before, we will make use of Lemma \ref{lem:inf-sup_aux} and establish the validity of the estimate \eqref{eq:inf-sup}. To this end, for any $K\geq N$ we denote by $\mathbb{P}_{0, K} \colon \widehat{\mathcal{L}}^2 \rightarrow \widehat{\mathcal{L}}^2$ the $\widehat{\mathcal{L}}^2$-orthogonal projector onto span$\{\Psi_{0, K}\}$ and by $\mathbb{P}_{0, K}^{\perp} \colon \widehat{\mathcal{H}}^1 \rightarrow \widehat{\mathcal{H}}^1$ its complement, i.e., $\mathbb{P}_{0, K}^{\perp}:={\rm I}-\mathbb{P}_{0, K}$. Note that since $\mathbb{P}_{0, K}$ has finite-dimensional range, both $\mathbb{P}_{0, K}$ and $\mathbb{P}_{0, K}^{\perp}$ extend to bounded linear mappings from $\widehat{\mathcal{H}}^1$ to $\widehat{\mathcal{H}}^1$. It will also be useful to recall from Notation \eqref{def:V_K} that for each $K\geq N$, we denote $\mathcal{V}_K=\widetilde{V}_K \oplus \text{span}\{\Psi_{0, K}\}$. \\
	
	We begin our proof by demonstrating two auxiliary results that follow directly from \textbf{Structure Assumption B.II}. The first of these minor results is the observation that for each $K \geq N$, the shifted mean-field operator $\mathcal{F}_K -\Lambda_0 \colon \widehat{\mathcal{H}}^1 \rightarrow \widehat{\mathcal{H}}^{-1}$ induces an inner-product, denoted $(\cdot, \cdot)_{\mathcal{F}_K}$, on the finite-dimensional space $\mathcal{W}_K \subset \widehat{\mathcal{H}}^1$ through the relation
	\begin{align}
		\forall K\geq N, ~ \forall \Upsilon_K, \Phi_K \in \mathcal{W}_K \colon \qquad (\Upsilon_K, \Phi_K)_{\mathcal{F}_K}:= \big\langle \Upsilon_K, (\mathcal{F}_K-\Lambda_0)\Phi_K\big\rangle_{\widehat{\mathcal{H}}^1 \times \widehat{\mathcal{H}}^{-1}}.
	\end{align}
	Indeed, this is simply a consequence of the symmetry of $\mathcal{F}_K$ and its coercivity on $\mathcal{W}_K$. More importantly however, the fact that both the continuity constant and coercivity constant of $\mathcal{F}_K$ are uniformly bounded in $K$ immediately implies that the norm $\Vert \cdot \Vert_{\mathcal{F}_K}$ induced by this new inner-product $(\cdot, \cdot)_{\mathcal{F}_K}$ on the finite-dimensional space $\mathcal{W}_K$ is equivalent to the usual $\widehat{\mathcal{H}}^1$ norm with equivalence constants independent of $K$.
	
	Consider now, for any $K\geq N$, the inner-product $(\cdot, \cdot)_{\mathcal{F}_K} \colon \mathcal{W}_K \times \mathcal{W}_K \rightarrow \mathbb{R}$ and recall that the approximation space $\mathcal{V}_K \subset \mathcal{W}_K$ by assumption. Let us denote by $\mathbb{P}_{\mathcal{F}_K} \colon \mathcal{W}_K \rightarrow \mathcal{W}_K$, the $(\cdot, \cdot)_{\mathcal{F}_K}$-orthogonal projection operator onto ${\mathcal{V}}_K$ and let us denote by $\mathbb{P}^{\perp}_{\mathcal{F}_K} \colon \mathcal{W}_K \rightarrow \mathcal{W}_K$, its complement, i.e., $\mathbb{P}^{\perp}_{\mathcal{F}_K}  ={\rm I}-\mathbb{P}_{\mathcal{F}_K} $. We now claim that these projection operators $\mathbb{P}_{\mathcal{F}_K}$ and $\mathbb{P}^{\perp}_{\mathcal{F}_K}$ are, in fact, $\widehat{\mathcal{L}}^2$-orthogonal.
	
	To see this, first recall the restricted operator $\widetilde{\mathcal{F}}_K \colon \mathcal{W}_K \rightarrow \mathcal{W}_K^*$ defined through Equation \eqref{eq:restriction}. Next, observe that since $\mathcal{W}_K$ is finite-dimensional, the Riesz representation theorem implies the existence of an isomorphism $\mathcal{i}_K\colon \mathcal{W}_K^* \to \mathcal{W}_K$ such that for all $\Phi_K, \Upsilon_K \in \mathcal{W}_K$ it holds that
    \begin{align*}
     \langle\Upsilon_K,  \mathcal{F}_K\Phi_K\rangle_{\widehat{\mathcal{H}}^{1} \times \widehat{\mathcal{H}}^{-1}}  =\langle \Upsilon_K, \widetilde{\mathcal{F}}_K\Phi_K\rangle_{\mathcal{W}_K \times \mathcal{W}_K^*}= (\Upsilon_K, \mathcal{i}_K\widetilde{\mathcal{F}}_K\Phi_K)_{\widehat{\mathcal{L}}^2}.
    \end{align*}
    Notice that since $\mathcal{F}_K \colon \widehat{\mathcal{H}}^1 \rightarrow \widehat{\mathcal{H}}^{-1}$ is symmetric, the mapping $\mathcal{i}_K\widetilde{\mathcal{F}}_K \colon \mathcal{W}_K \rightarrow \mathcal{W}_K$ is symmetric as well. We claim that the mapping $\mathcal{i}_K\widetilde{\mathcal{F}}_K -\Lambda_0 \colon \mathcal{W}_K \rightarrow \mathcal{W}_K$ is also invertible. To see this, it suffices to appeal to the coercivity hypothesis in \textbf{Structure Assumption B.II} and note that for all $\Phi_K \in \mathcal{W}_K$ we have
\begin{align*}
    (\Phi_K, (\mathcal{i}_K\widetilde{\mathcal{F}}_K -\Lambda_0)\Phi_K)_{\widehat{\mathcal{L}}^2}&=(\Phi_K, \mathcal{i}_K\widetilde{\mathcal{F}}_K \Phi_K)_{\widehat{\mathcal{L}}^2}-\Lambda_0(\Phi_K,\Phi_K)_{\widehat{\mathcal{L}}^2}\\
    &= \langle\Phi_K, {\mathcal{F}}_K \Phi_K\rangle_{\widehat{\mathcal{H}}^1\times \widehat{\mathcal{H}}^{-1}} -  \Lambda_0(\Phi_K,\Phi_K)_{\widehat{\mathcal{L}}^2}\\
    &=\langle\Phi_K, ({\mathcal{F}}_K -\Lambda_0)\Phi_K\rangle_{\widehat{\mathcal{H}}^1\times \widehat{\mathcal{H}}^{-1}}\\
    &\geq \widetilde{\gamma} \Vert \Phi_K\Vert^2_{\widehat{\mathcal{H}}^1}\geq\widetilde{\gamma} \Vert \Phi_K\Vert^2_{\widehat{\mathcal{L}}^2}. 
\end{align*}

Next, recall from \textbf{Structure Assumption B.II} that the approximation space $\mathcal{V}_K$ is an invariant subspace of $\widetilde{\mathcal{F}}_K$ in the sense of Equation \eqref{eq:Yvon_invariant}. This implies that for all $\Phi_K \in \mathcal{V}_K$ and all $\Upsilon_K^\perp \in \mathcal{V}_K^{\perp} \cap \mathcal{W}_K$:
\begin{align*}
       0=\langle \Upsilon_K^{\perp}, \widetilde{\mathcal{F}}_K\Phi_K \rangle_{\mathcal{W}_K \times \mathcal{W}_K^*}= (\Upsilon_K^\perp, \mathcal{i}_K\widetilde{\mathcal{F}}_K\Phi_K)_{\widehat{\mathcal{L}}^2}.
        \end{align*}
    Consequently, $\mathcal{V}_K$ is also an invariant subspace of the mapping $\mathcal{i}_K\widetilde{\mathcal{F}}_K \colon \mathcal{W}_K \rightarrow \mathcal{W}_K$, i.e., 
    \begin{align*}
\text{Range}\big(\mathcal{i}_K\widetilde{\mathcal{F}}_K\vert_{\mathcal{V}_K}\big) \subset \mathcal{V}_K.
    \end{align*}
Since $\mathcal{i}_K\widetilde{\mathcal{F}}_K -\Lambda_0\colon \mathcal{W}_K \rightarrow \mathcal{W}_K$ is bijective, we immediately deduce that the restricted mapping
\begin{align*}
    \mathcal{i}_K\widetilde{\mathcal{F}}_K -\Lambda_0&\colon \mathcal{V}_K \rightarrow \mathcal{V}_K \quad \text{is invertible}.
\end{align*}
Next, observe that the $(\cdot, \cdot)_{\mathcal{F}_K}$-orthogonal projector $\mathbb{P}_{\mathcal{F}_K}$ onto ${\mathcal{V}}_K$ satisfies for all $\Upsilon_K \in \mathcal{W}_K, \Phi_K \in \mathcal{V}_K$
\begin{align*}
    0=\langle\Upsilon_K-\mathbb{P}_{\mathcal{F}_K}\Upsilon_K,  ({\mathcal{F}}_K -\Lambda_0)\Phi_K\rangle_{\widehat{\mathcal{H}}^1 \times \widehat{\mathcal{H}}^{-1}}= \big(\Upsilon_K-\mathbb{P}_{\mathcal{F}_K}\Upsilon_K,  (\mathcal{i}_K\widetilde{\mathcal{F}}_K -\Lambda_0)\Phi_K\big)_{\widehat{\mathcal{L}}^2}.
\end{align*}
But $\mathcal{i}_K\widetilde{\mathcal{F}}_K -\Lambda_0 \colon \mathcal{V}_K \rightarrow \mathcal{V}_K$ is invertible. Consequently, for all $\Upsilon_K \in \mathcal{W}_K, \widetilde{\Phi}_K \in \mathcal{V}_K$ we have
\begin{align}\label{eq:Yvon_L2_orthogonality}
    0=\big(\Upsilon_K-\mathbb{P}_{\mathcal{F}_K}\Upsilon_K,  \widetilde{\Phi}_K\big)_{\widehat{\mathcal{L}}^2}.
\end{align}
Thus, $\mathbb{P}_{\mathcal{F}_K}$ coincides with the $\widehat{\mathcal{L}}^2$-orthogonal projector from $\mathcal{W}_K$ onto $\mathcal{V}_K$.

	Equipped with the above two auxiliary results, let us now return to our goal of showing the validity of Estimate \eqref{eq:inf-sup} from  Lemma \ref{lem:inf-sup_aux}. For any arbitrary $\Upsilon_K \in \widetilde{\mathcal{V}}_K \subset \widehat{\mathcal{H}}^{1, \perp}_{\Psi_{0, K}} $, let us now choose the test function $\Phi_K \in  \widetilde{\mathcal{V}}_K$ as
	\begin{align}\label{eq:lemma_inf-sup_5}
		\Phi_K:= \mathbb{P}_{0, K}^{\perp}\mathbb{P}_{\mathcal{F}_K}e^{-\mathcal{T}(\Theta^{\Pi}_{K, \rm GS})} e^{-\mathcal{T}(\Theta^{\Pi}_{K, \rm GS})^{\dagger}}\Upsilon_K.
	\end{align} 
  where, as in Lemma \ref{lem:inf-sup_aux}, $\Theta^{\Pi}_{K, \rm GS}\in \widetilde{\mathcal{V}}_K$ denotes the best approximation with respect to the $\widehat{\mathcal{H}}^1$ inner product of the exact ground state zero $\Theta_{K, \rm GS}^* \in \widehat{\mathcal{H}}^{1, \perp}_{\Psi_{0, K}}$ of the coupled cluster function $\mathcal{f}_K\colon \widehat{\mathcal{H}}^{1, \perp}_{\Psi_{0, K}} \rightarrow \left(\widehat{\mathcal{H}}^{1, \perp}_{\Psi_{0, K}}\right)^*$.
  
	Note that $\Phi_K \in \widetilde{\mathcal{V}}_K$ is well-defined since by \textbf{Structural Assumption B.II}, the function \[e^{-\mathcal{T}(\Theta^{\Pi}_{K, \rm GS})} e^{-\mathcal{T}(\Theta^{\Pi}_{K, \rm GS})^{\dagger}}\Upsilon_K \in \mathcal{W}_K,\] and applying the projection operator $\mathbb{P}_{\mathcal{F}_K}$ yields an element of $\mathcal{V}_K$.
	
	It follows that for any arbitrary $\Upsilon_K \in \widetilde{\mathcal{V}}_K$, with the above choice of test function $\Phi_K \in \widetilde{\mathcal{V}}_K$, we have that (c.f., Estimate \eqref{eq:inf-sup})
	\begin{align}\nonumber
		&\left\langle  \Upsilon_K,  e^{-\mathcal{T}(\Theta^{\Pi}_{K, \rm GS})} (H-\mathcal{E}^*_{\rm GS})e^{\mathcal{T}(\Theta^{\Pi}_{K, \rm GS})}\Phi_K\right\rangle_{\widehat{\mathcal{H}}^1\times \mathcal{H}^{-1}}\\ \nonumber
		=&  \left\langle  \Upsilon_K,  e^{-\mathcal{T}(\Theta^{\Pi}_{K, \rm GS})}(H-\mathcal{E}^*_{\rm GS}) e^{\mathcal{T}(\Theta^{\Pi}_{K, \rm GS})}\mathbb{P}_{0, K}^{\perp}\mathbb{P}_{\mathcal{F}_K}e^{-\mathcal{T}(\Theta^{\Pi}_{K, \rm GS})} e^{-\mathcal{T}(\Theta^{\Pi}_{K, \rm GS})^{\dagger}}\Upsilon_K\right\rangle_{\widehat{\mathcal{H}}^1\times \mathcal{H}^{-1}}\\ \nonumber
		=&  \left\langle  \Upsilon_K,  e^{-\mathcal{T}(\Theta^{\Pi}_{K, \rm GS})}(H-\mathcal{E}^*_{\rm GS}) e^{\mathcal{T}(\Theta^{\Pi}_{K, \rm GS})}\mathbb{P}_{\mathcal{F}_K}e^{-\mathcal{T}(\Theta^{\Pi}_{K, \rm GS})} e^{-\mathcal{T}(\Theta^{\Pi}_{K, \rm GS})^{\dagger}}\Upsilon_K\right\rangle_{\widehat{\mathcal{H}}^1\times \mathcal{H}^{-1}}\\ \nonumber
		-&\left\langle  \Upsilon_K,  e^{-\mathcal{T}(\Theta^{\Pi}_{K, \rm GS})}(H-\mathcal{E}^*_{\rm GS}) e^{\mathcal{T}(\Theta^{\Pi}_{K, \rm GS})}\mathbb{P}_{0, K}\mathbb{P}_{\mathcal{F}_K}e^{-\mathcal{T}(\Theta^{\Pi}_{K, \rm GS})} e^{-\mathcal{T}(\Theta^{\Pi}_{K, \rm GS})^{\dagger}}\Upsilon_K\right\rangle_{\widehat{\mathcal{H}}^1\times \mathcal{H}^{-1}}\\ \nonumber
		=&\underbrace{\left\langle  \Upsilon_K, e^{-\mathcal{T}(\Theta^{\Pi}_{K, \rm GS})} (H-\mathcal{E}^*_{\rm GS}) \mathbb{P}_{\mathcal{F}_K}e^{\mathcal{T}(\Theta^{\Pi}_{K, \rm GS})}\mathbb{P}_{\mathcal{F}_K}e^{-\mathcal{T}(\Theta^{\Pi}_{K, \rm GS})} e^{-\mathcal{T}(\Theta^{\Pi}_{K, \rm GS})^{\dagger}}\Upsilon_K\right\rangle_{\widehat{\mathcal{H}}^1\times \mathcal{H}^{-1}}}_{:=\rm (I*)}\\  \label{eq:lemma_inf-sup_6}
		+& \underbrace{\left\langle  \Upsilon_K,  e^{-\mathcal{T}(\Theta^{\Pi}_{K, \rm GS})}(H-\mathcal{E}^*_{\rm GS}) \mathbb{P}^{\perp}_{\mathcal{F}_K}e^{\mathcal{T}(\Theta^{\Pi}_{K, \rm GS})}\mathbb{P}_{\mathcal{F}_K}e^{-\mathcal{T}(\Theta^{\Pi}_{K, \rm GS})} e^{-\mathcal{T}(\Theta^{\Pi}_{K, \rm GS})^{\dagger}}\Upsilon_K\right\rangle_{\widehat{\mathcal{H}}^1\times \mathcal{H}^{-1}}}_{:=\rm (II*)}\\ \nonumber
		-& \underbrace{\left\langle  \Upsilon_K,  e^{-\mathcal{T}(\Theta^{\Pi}_{K, \rm GS})}(H-\mathcal{E}^*_{\rm GS}) e^{\mathcal{T}(\Theta^{\Pi}_{K, \rm GS})}\mathbb{P}_{0, K}\mathbb{P}_{\mathcal{F}_K}e^{-\mathcal{T}(\Theta^{\Pi}_{K, \rm GS})} e^{-\mathcal{T}(\Theta^{\Pi}_{K, \rm GS})^{\dagger}}\Upsilon_K\right\rangle_{\widehat{\mathcal{H}}^1\times \mathcal{H}^{-1}}}_{:= \rm (III*)},
	\end{align}
	where we remind the reader that, as in Lemma \ref{lem:inf-sup_aux}, $\mathcal{E}^*_{\rm GS}$ denotes the ground state eigenvalue of the electronic Hamiltonian $H \colon \widehat{\mathcal{H}}^1\rightarrow \widehat{\mathcal{H}}^{-1}$.\vspace{2mm}
	
	Intuitively, the key idea of the remainder of this proof is to argue that the term (I*) is positive while the terms (II*) and (III*) are `small'-- at least asymptotically. We proceed one term at a time. \vspace{0mm}
	
	\textbf{Estimation of the term (I*)}~
	
	We begin by considering the term (I*). As a first step, we claim that
    \begin{align}\nonumber
		\mathbb{P}_{\mathcal{F}_K}e^{\mathcal{T}(\Theta^{\Pi}_{K, \rm GS})}\mathbb{P}_{\mathcal{F}_K}&\colon \mathcal{V}_K \rightarrow \mathcal{V}_K ~\text{ is an invertible mapping with }\\[0.5em] \label{eq:lemma_inf-sup_7}
		\left(\mathbb{P}_{\mathcal{F}_K}e^{\mathcal{T}(\Theta^{\Pi}_{K, \rm GS})}\mathbb{P}_{\mathcal{F}_K} \right)^{-1} &= \mathbb{P}_{\mathcal{F}_K}e^{-\mathcal{T}(\Theta^{\Pi}_{K, \rm GS})}\mathbb{P}_{\mathcal{F}_K}.
	\end{align}
 
	Indeed, a direct calculation shows that
	\begin{align}\nonumber
		\left(\mathbb{P}_{\mathcal{F}_K}e^{\mathcal{T}(\Theta^{\Pi}_{K, \rm GS})}\mathbb{P}_{\mathcal{F}_K}\right)   \left(\mathbb{P}_{\mathcal{F}_K}e^{-\mathcal{T}(\Theta^{\Pi}_{K, \rm GS})}\mathbb{P}_{\mathcal{F}_K}\right)&=  \mathbb{P}_{\mathcal{F}_K}e^{\mathcal{T}(\Theta^{\Pi}_{K, \rm GS})}\mathbb{P}_{\mathcal{F}_K}e^{-\mathcal{T}(\Theta^{\Pi}_{K, \rm GS})}\mathbb{P}_{\mathcal{F}_K}\\ \label{eq:lemma_inf-sup_8}
		&= \mathbb{P}_{\mathcal{F}_K}e^{\mathcal{T}(\Theta^{\Pi}_{K, \rm GS})}e^{-\mathcal{T}(\Theta^{\Pi}_{K, \rm GS})}\mathbb{P}_{\mathcal{F}_K}\\ \nonumber
		&- \underbrace{\mathbb{P}_{\mathcal{F}_K}e^{\mathcal{T}(\Theta^{\Pi}_{K, \rm GS})}\mathbb{P}^{\perp}_{\mathcal{F}_K}e^{-\mathcal{T}(\Theta^{\Pi}_{K, \rm GS})}\mathbb{P}_{\mathcal{F}_K}}_{:= \rm (IA*)}.
	\end{align}
 
	We claim that the term (IA*) is zero. Indeed, suppose there exists $\widehat{\Phi}_K \in \mathcal{V}_K$ such that \[\mathbb{P}_{\mathcal{F}_K}e^{\mathcal{T}(\Theta^{\Pi}_{K, \rm GS})}\mathbb{P}^{\perp}_{\mathcal{F}_K}e^{-\mathcal{T}(\Theta^{\Pi}_{K, \rm GS})}\mathbb{P}_{\mathcal{F}_K} \widehat{\Phi}_K =\mathbb{P}_{\mathcal{F}_K}e^{\mathcal{T}(\Theta^{\Pi}_{K, \rm GS})}\mathbb{P}^{\perp}_{\mathcal{F}_K}e^{-\mathcal{T}(\Theta^{\Pi}_{K, \rm GS})} \widehat{\Phi}_K\neq 0.\] It follows that there must exist some $\widehat{\Upsilon}_K \in \mathcal{V}_K$ such that 
	\begin{align*}
		0\neq \left(\widehat{\Upsilon}_K, \mathbb{P}_{\mathcal{F}_K}e^{\mathcal{T}(\Theta^{\Pi}_{K, \rm GS})}\mathbb{P}^{\perp}_{\mathcal{F}_K}e^{-\mathcal{T}(\Theta^{\Pi}_{K, \rm GS})}\widehat{\Phi}_K\right)_{\widehat{\mathcal{L}}^2}= \left(e^{\mathcal{T}(\Theta^{\Pi}_{K, \rm GS})^{\dagger}}\widehat{\Upsilon}_K, \mathbb{P}^{\perp}_{\mathcal{F}_K}e^{-\mathcal{T}(\Theta^{\Pi}_{K, \rm GS})}\widehat{\Phi}_K\right)_{\widehat{\mathcal{L}}^2}.
	\end{align*}
	But thanks to \textbf{Structure Assumption B.II}, it holds that $e^{\mathcal{T}(\Theta^{\Pi}_{K, \rm GS})^{\dagger}}\widehat{\Upsilon}_K \in \mathcal{V}_K$. And since, the projection operator $\mathbb{P}^{\perp}_{\mathcal{F}_K}$ is $\widehat{\mathcal{L}}^2$-orthogonal, the inner product on the right must be zero. Recalling Equation \eqref{eq:lemma_inf-sup_8}, we see that our claim concerning the invertibility of the projected exponential cluster operator over $\mathcal{V}_K$ readily follows.
	
	Let us now return to our earlier aim of estimating the term (I*) appearing in Equation \eqref{eq:lemma_inf-sup_6}. Using the invertibility result \eqref{eq:lemma_inf-sup_7} and taking advantage of the fact that $ e^{-\mathcal{T}(\Theta^{\Pi}_{K, \rm GS})^{\dagger}}\Upsilon_K \in \mathcal{V}_K$, which is again a consequence of \textbf{Structure Assumption B.II}, we deduce that
	\begin{align*}
		{\rm (I*)}&= \left\langle  \Upsilon_K, e^{-\mathcal{T}(\Theta^{\Pi}_{K, \rm GS})} (H-\mathcal{E}^*_{\rm GS}) \mathbb{P}_{\mathcal{F}_K}e^{\mathcal{T}(\Theta^{\Pi}_{K, \rm GS})}\mathbb{P}_{\mathcal{F}_K}e^{-\mathcal{T}(\Theta^{\Pi}_{K, \rm GS})} e^{-\mathcal{T}(\Theta^{\Pi}_{K, \rm GS})^{\dagger}}\Upsilon_K\right\rangle_{\widehat{\mathcal{H}}^1\times \mathcal{H}^{-1}}\\
		&=\left\langle  \Upsilon_K, e^{-\mathcal{T}(\Theta^{\Pi}_{K, \rm GS})} (H-\mathcal{E}^*_{\rm GS}) \mathbb{P}_{\mathcal{F}_K}e^{\mathcal{T}(\Theta^{\Pi}_{K, \rm GS})}\mathbb{P}_{\mathcal{F}_K}e^{-\mathcal{T}(\Theta^{\Pi}_{K, \rm GS})} \mathbb{P}_{\mathcal{F}_K}e^{-\mathcal{T}(\Theta^{\Pi}_{K, \rm GS})^{\dagger}}\Upsilon_K\right\rangle_{\widehat{\mathcal{H}}^1\times \mathcal{H}^{-1}}\\
		&=\left\langle  \Upsilon_K, e^{-\mathcal{T}(\Theta^{\Pi}_{K, \rm GS})} (H-\mathcal{E}^*_{\rm GS}) \mathbb{P}_{\mathcal{F}_K}e^{-\mathcal{T}(\Theta^{\Pi}_{K, \rm GS})^{\dagger}}\Upsilon_K\right\rangle_{\widehat{\mathcal{H}}^1\times \mathcal{H}^{-1}}\\
		&=\left\langle  e^{-\mathcal{T}(\Theta^{\Pi}_{K, \rm GS})^{\dagger}}\Upsilon_K, (H-\mathcal{E}^*_{\rm GS})e^{-\mathcal{T}(\Theta^{\Pi}_{K, \rm GS})^{\dagger}}\Upsilon_K\right\rangle_{\widehat{\mathcal{H}}^1\times \mathcal{H}^{-1}}.
	\end{align*}
	
	Next, let $(\mathcal{E}^*_{{\rm GS}, K}, \Psi_{{\rm GS}, K}^*)$ denote a ground state (normalised) eigenpair of the electronic Hamiltonian in the finite-dimensional approximation space $\mathcal{W}_K$. Recall that the approximation spaces $\{\mathcal{V}_K\}_{K\geq N}$ are dense in $\widehat{\mathcal{H}}^1$ by definition (see Notation \ref{def:V_K}), and that \textbf{Structure Assumption B.II} imposes that for each $K\geq N$, we have $\mathcal{V}_K \subset \mathcal{W}_K$. We can therefore conclude that the sequence of finite-dimensional spaces $\{\mathcal{W}_K\}_{K\geq N}$ are also dense in $\widehat{\mathcal{H}}^1$. Consequently, in view of \textbf{Assumption A.III}, standard approximability results for linear eigenvalue problems (see, e.g., \cite[Example 5.9]{MR0716134}) imply that the eigenvalue $\mathcal{E}^*_{{\rm GS}, K}$ must be simple for sufficiently large $K\geq N$, and with the correct choice of normalisation, $\Psi_{{\rm GS}, K}^* \to \Psi_{{\rm GS}}^*$ as $K \to \infty$ in the $\widehat{\mathcal{H}}^1$ sense. This implies in particular that for all $K$ sufficiently large, the shifted Hamiltonian $H-\mathcal{E}_{\rm GS}^*$ is coercive on $\{\Psi_{{\rm GS}, K}^*\}^{\perp}\cap \mathcal{W}_K$ under the equivalent $\Vert \cdot \Vert_{\mathcal{F}_K}$ norm with uniformly bounded (in $K$) coercivity constant $\Gamma^*_{\rm GS} > 0$.

    Next, notice that the convergence result from Lemma \ref{lem:approx} implies that
	\begin{align*}
		\lim_{K \to \infty} \left\Vert \frac{1}{\big(\Psi_{{\rm GS}, K}^*, \Psi_{0, K}\big)_{\widehat{\mathcal{L}}}}\Psi_{{\rm GS}, K}^* - e^{\mathcal{T}(\Theta^{\Pi}_{K, \rm GS})}\Psi_{0, K}\right\Vert_{\widehat{\mathcal{L}}^2}&=0.
	\end{align*}
	Since, additionally, the test function $e^{-\mathcal{T}(\Theta^{\Pi}_{K, \rm GS})^{\dagger}}\Upsilon_K$ appearing in the simplification of the term~(I*) is $\widehat{\mathcal{L}}^2$-orthogonal to $e^{\mathcal{T}(\Theta^{\Pi}_{K, \rm GS})}\Psi_{0, K}$ (recall that $\Upsilon_K \in \widetilde{\mathcal{V}}_K \perp \Psi_{0, K}$), we can deduce the existence of a constant $\Gamma^{\mathcal{F}_K}_{\rm GS}>0$ such that for all $K\geq N$ sufficiently large
	\begin{equation}\label{eq:lemma_inf-sup_9}
		{\rm (I*)}\geq \Gamma^{\mathcal{F}_K}_{\rm GS} \Vert e^{-\mathcal{T}(\Theta^{\Pi}_{K, \rm GS})^{\dagger}}\Upsilon_K\Vert_{\mathcal{F}_K}^2 \quad \text{and} \quad \lim_{K \to \infty}\Gamma^{\mathcal{F}_K}_{\rm GS}= \Gamma^{*}_{\rm GS}.
	\end{equation}
	
	\vspace{1mm}
	
	\textbf{Estimation of the term (II*)}~
	
	We now turn our attention to the term (II*). We begin by using once again \textbf{Structure Assumption B.II} and employing the splitting of the electronic Hamiltonian $H$ given by 
	\begin{align*}
		\forall K\geq N \colon \quad  H= \mathcal{F}_K + \mathcal{U}_K \quad \text{where }~ \mathcal{U}_K:= H-\mathcal{F}_K \colon \widehat{\mathcal{H}}^1 \rightarrow \widehat{\mathcal{L}}^2 ~\text{is a bounded linear operator}.
	\end{align*}
	Employing this splitting allows us to write the term (II*) as
	\begin{align*}
		{\rm (II*)} &=\left\langle  \Upsilon_K,  e^{-\mathcal{T}(\Theta^{\Pi}_{K, \rm GS})}(H-\mathcal{E}^*_{\rm GS}) \mathbb{P}^{\perp}_{\mathcal{F}_K}e^{\mathcal{T}(\Theta^{\Pi}_{K, \rm GS})}\mathbb{P}_{\mathcal{F}_K}e^{-\mathcal{T}(\Theta^{\Pi}_{K, \rm GS})} e^{-\mathcal{T}(\Theta^{\Pi}_{K, \rm GS})^{\dagger}}\Upsilon_K\right\rangle_{\widehat{\mathcal{H}}^1\times \mathcal{H}^{-1}}\\
		&=\underbrace{\left\langle  \Upsilon_K,  e^{-\mathcal{T}(\Theta^{\Pi}_{K, \rm GS})}(\mathcal{F}_K-\mathcal{E}^*_{\rm GS}) \mathbb{P}^{\perp}_{\mathcal{F}_K}e^{\mathcal{T}(\Theta^{\Pi}_{K, \rm GS})}\mathbb{P}_{\mathcal{F}_K}e^{-\mathcal{T}(\Theta^{\Pi}_{K, \rm GS})} e^{-\mathcal{T}(\Theta^{\Pi}_{K, \rm GS})^{\dagger}}\Upsilon_K\right\rangle_{\widehat{\mathcal{H}}^1\times \mathcal{H}^{-1}}}_{= \rm (II**)}\\ 
		&+\left\langle  \Upsilon_K,  e^{-\mathcal{T}(\Theta^{\Pi}_{K, \rm GS})}\mathcal{U}_K \mathbb{P}^{\perp}_{\mathcal{F}_K}e^{\mathcal{T}(\Theta^{\Pi}_{K, \rm GS})}\mathbb{P}_{\mathcal{F}_K}e^{-\mathcal{T}(\Theta^{\Pi}_{K, \rm GS})} e^{-\mathcal{T}(\Theta^{\Pi}_{K, \rm GS})^{\dagger}}\Upsilon_K\right\rangle_{\widehat{\mathcal{H}}^1\times \mathcal{H}^{-1}}.
	\end{align*}
	We claim that the term (II**) is in fact zero. To see this, note that \textbf{Structure Assumption B.II} implies that $e^{-\mathcal{T}(\Theta^{\Pi}_{K, \rm GS})^{\dagger}}\Upsilon_K \in \mathcal{V}_K$. Since $\mathbb{P}_{\mathcal{F}_K}$ is the $( \cdot, \cdot )_{\mathcal{F}_K}$-orthogonal projector onto $\mathcal{V}_K$, we can deduce that
    \begin{align*}
       {\rm(II**)} &= (\Lambda_0-\mathcal{E}^*_{\rm GS}) \left\langle  e^{-\mathcal{T}(\Theta^{\Pi}_{K, \rm GS})^{\dagger}}\Upsilon_K,\mathbb{P}^{\perp}_{\mathcal{F}_K}e^{\mathcal{T}(\Theta^{\Pi}_{K, \rm GS})}\mathbb{P}_{\mathcal{F}_K}e^{-\mathcal{T}(\Theta^{\Pi}_{K, \rm GS})} e^{-\mathcal{T}(\Theta^{\Pi}_{K, \rm GS})^{\dagger}}\Upsilon_K\right\rangle_{\widehat{\mathcal{H}}^1\times \mathcal{H}^{-1}}\\
       &=(\Lambda_0-\mathcal{E}^*_{\rm GS}) \left(  e^{-\mathcal{T}(\Theta^{\Pi}_{K, \rm GS})^{\dagger}}\Upsilon_K,\mathbb{P}^{\perp}_{\mathcal{F}_K}e^{\mathcal{T}(\Theta^{\Pi}_{K, \rm GS})}\mathbb{P}_{\mathcal{F}_K}e^{-\mathcal{T}(\Theta^{\Pi}_{K, \rm GS})} e^{-\mathcal{T}(\Theta^{\Pi}_{K, \rm GS})^{\dagger}}\Upsilon_K\right)_{\widehat{\mathcal{L}}^2},
    \end{align*}
    where the second step follows from Equation \eqref{eq:Yvon_duality}. Since the projection operator $\mathbb{P}_{\mathcal{F}_K}$ is also $( \cdot, \cdot )_{\widehat{\mathcal{L}}^2}$-orthogonal (recall Equation \eqref{eq:Yvon_L2_orthogonality}), we obtain that (II**)$=0$ as claimed, and hence,
    \begin{align*}
		{\rm (II*)} = \left\langle  \Upsilon_K,  e^{-\mathcal{T}(\Theta^{\Pi}_{K, \rm GS})}\mathcal{U}_K \mathbb{P}^{\perp}_{\mathcal{F}_K}e^{\mathcal{T}(\Theta^{\Pi}_{K, \rm GS})}\mathbb{P}_{\mathcal{F}_K}e^{-\mathcal{T}(\Theta^{\Pi}_{K, \rm GS})} e^{-\mathcal{T}(\Theta^{\Pi}_{K, \rm GS})^{\dagger}}\Upsilon_K\right\rangle_{\widehat{\mathcal{H}}^1\times \mathcal{H}^{-1}}. 
	\end{align*}
	
	Appealing once again to Equation \eqref{eq:Yvon_duality}, a direct calculation now yields
	\begin{align*}
		{\rm (II*)} &\geq - \left \Vert \mathbb{P}^{\perp}_{\mathcal{F}_K}\mathcal{U}_K  e^{-\mathcal{T}(\Theta^{\Pi}_{K, \rm GS})^{\dagger}}\Upsilon_K \right\Vert_{\widehat{\mathcal{L}}^2} \left \Vert\mathbb{P}^{\perp}_{\mathcal{F}_K}e^{\mathcal{T}(\Theta^{\Pi}_{K, \rm GS})}\mathbb{P}_{\mathcal{F}_K}e^{-\mathcal{T}(\Theta^{\Pi}_{K, \rm GS})} e^{-\mathcal{T}(\Theta^{\Pi}_{K, \rm GS})^{\dagger}}\Upsilon_K \right\Vert_{\widehat{\mathcal{L}}^2}\\[0.5em]
		&\geq -\left \Vert \mathbb{P}^{\perp}_{\mathcal{F}_K}\widetilde{\mathcal{U}}_K \mathbb{P}_{\mathcal{F}_K}\right\Vert_{\mathcal{F}_K \to \widehat{\mathcal{L}}^2} \left \Vert e^{-\mathcal{T}(\Theta^{\Pi}_{K, \rm GS})^{\dagger}}\Upsilon_K \right\Vert_{\mathcal{F}_K} \left \Vert \mathbb{P}^{\perp}_{\mathcal{F}_K}e^{\mathcal{T}(\Theta^{\Pi}_{K, \rm GS})}\mathbb{P}_{\mathcal{F}_K}e^{-\mathcal{T}(\Theta^{\Pi}_{K, \rm GS})} e^{-\mathcal{T}(\Theta^{\Pi}_{K, \rm GS})^{\dagger}}\Upsilon_K\right\Vert_{\widehat{\mathcal{L}}^2},
	\end{align*}
	where $\widetilde{\mathcal{U}}_K \colon \mathcal{W}_K \rightarrow \mathcal{W}_K$ denotes the restriction of the formally infinite-dimensional operator $\mathcal{U}_K \colon \widehat{\mathcal{H}}_1 \rightarrow \widehat{\mathcal{L}}_2$ to the finite-dimensional space $\mathcal{W}_K$, and we have denoted by $\Vert \cdot \Vert_{\mathcal{F}_K \to \widehat{\mathcal{L}}^2}$, the operator norm on the space of bounded linear operators acting from $\mathcal{W}_K$, equipped with the $\widehat{\mathcal{H}}^1$-equivalent $\Vert \cdot \Vert_{\mathcal{F}_K}$ norm to the space $\mathcal{W}_K$ equipped with the usual $\Vert \cdot \Vert_{\widehat{\mathcal{L}}^2}$ norm. Notice that this is permitted since the function $e^{-\mathcal{T}(\Theta^{\Pi}_{K, \rm GS})^{\dagger}}\Upsilon_K \in \mathcal{V}_K \subset\mathcal{W}_K$ for all $K\geq N$ as a consequence of \textbf{Structure Assumption B.II}.

	To proceed, let us consider an arbitrary element $\widehat{\Phi}_K \in \mathcal{W}_K$ for any $K\geq N$. It is easy to see that we can write
	\begin{align}\label{eq:lemma_inf-sup_10}
		\left \Vert \mathbb{P}^{\perp}_{\mathcal{F}_K}\widehat{\Phi}_K\right\Vert^2_{\widehat{\mathcal{L}}^2} \leq \frac{1}{\Lambda_{\min}^{\mathcal{F}_K}} \left \Vert \mathbb{P}^{\perp}_{\mathcal{F}_K}\widehat{\Phi}_K\right\Vert^2_{\mathcal{F}_K},
	\end{align}
	where $\Lambda_{\min}^{\mathcal{F}_K} >0$ denotes the minimum of the spectrum of the shifted, restricted mean-field operator $\widetilde{\mathcal{F}}_K-\Lambda_0$ in the subspace $\text{ran}\;\mathbb{P}^{\perp}_{\mathcal{F}_K}$, which can be expressed as
 \begin{align*}
		\Lambda_{\rm min}^{\mathcal{F}_K}= \min_{\substack{\Phi_K \in \mathcal{W}_K\\ \Phi_K \in \mathcal{V}_K^{\perp}}} \frac{\langle \Phi_K, (\mathcal{F}_K-\Lambda_0) \Phi_K\rangle_{\widehat{\mathcal{H}}^1 \times \widehat{\mathcal{H}}^{-1}}}{\Vert \Phi_K \Vert_{\widehat{\mathcal{L}}^2}^2 }, \quad \text{and we have that}\quad\mathcal{W}_K = \underbrace{\text{ran}\;\mathbb{P}_{\mathcal{F}_K}}_{:= \mathcal{V}_K} \oplus \text{ ran}\;\mathbb{P}^{\perp}_{\mathcal{F}_K}.
	\end{align*}
	
	Returning now to the term (II*) and using Inequality \eqref{eq:lemma_inf-sup_10} together with some elementary simplifications, we deduce that for all $K\geq N$ it holds that
	\begin{equation}\label{eq:lemma_inf-sup_11}
		{\rm (II*)} \geq - \frac{\left \Vert \mathbb{P}^{\perp}_{\mathcal{F}_K}\widetilde{\mathcal{U}}_K \mathbb{P}_{\mathcal{F}_K}\right\Vert_{\mathcal{F}_K \to \widehat{\mathcal{L}}^2} \left \Vert \mathbb{P}^{\perp}_{\mathcal{F}_K}{e^{\mathcal{T}(\Theta^{\Pi}_{K, \rm GS})}}\mathbb{P}_{\mathcal{F}_K}{e^{-\mathcal{T}(\Theta^{\Pi}_{K, \rm GS})}} \mathbb{P}_{\mathcal{F}_K}\right\Vert_{\mathcal{F}_K \to \mathcal{F}_K}}{\sqrt{\Lambda_{\min}^{\mathcal{F}_K} }}\left \Vert e^{-\mathcal{T}(\Theta^{\Pi}_{K, \rm GS})^{\dagger}}\Upsilon_K \right\Vert^2_{\mathcal{F}_K}.
	\end{equation}
	
	\vspace{2mm}
	
	\textbf{Estimation of the term (III*)}~
	
	It remains to estimate the final term (III*). To do so, we will recall the estimation of the very similar term (II) appearing in Equation \eqref{eq:forgot} from the earlier proof of Lemma \ref{lem:inf-sup} under \textbf{Structure Assumption B.I}. Following exactly the same arguments as before (which we do not repeat for the sake of brevity), we deduce for all $K\geq N$ sufficiently large, the existence of a constant $\epsilon_K^{\rm (III)}$ such that
	\begin{equation}\label{eq:lemma_inf-sup_12}
		{\rm (III*)}\geq - \epsilon_K^{(\rm III)} \left \Vert e^{-\mathcal{T}(\Theta^{\Pi}_{K, \rm GS})^{\dagger}}\Upsilon_K \right\Vert^2_{\mathcal{F}_K}  \quad \text{and} \quad \lim_{K \to \infty}{\epsilon}_K^{\rm (III)}=0.
	\end{equation}

	Combining now Equation \eqref{eq:lemma_inf-sup_6} with Inequalities \eqref{eq:lemma_inf-sup_9}, \eqref{eq:lemma_inf-sup_11}, and \eqref{eq:lemma_inf-sup_12}, recalling our choice of test function $\Phi_K$ given by Equation \eqref{eq:lemma_inf-sup_5}, and making use of the assumption on the continuity constant of $\mathcal{U}_K$ from \textbf{Structure Assumption B.II}, we see that we have shown the existence of a constant $\widetilde{\gamma}>0$ and $\widehat{K}_0$ such that for all $K\geq \widehat{K}_0\geq N$ and all $\Upsilon_K \in \widetilde{\mathcal{V}}_K$ it holds that\vspace{2mm}
	\begin{align*}
		\sup_{\Phi_K \in \widetilde{\mathcal{V}}_K} &\frac{\left\langle  \Upsilon_K,  e^{-\mathcal{T}(\Theta^{\Pi}_{K, \rm GS})}(H-\mathcal{E}^*_{\rm GS}) e^{\mathcal{T}(\Theta^{\Pi}_{K, \rm GS})}\Phi_K\right\rangle_{\widehat{\mathcal{H}}^1\times \mathcal{H}^{-1}} }{\Vert \Upsilon_K\Vert_{\mathcal{F}_K}\Vert \Phi_K\Vert_{\mathcal{F}_K}}\\[1em] 
		&\geq \widetilde{\gamma}\frac{\Vert e^{-\mathcal{T}(\Theta^{\Pi}_{K, \rm GS})^{\dagger}}\Upsilon_K\Vert_{\mathcal{F}_K}^2}{\Vert \Upsilon_K\Vert_{\mathcal{F}_K}\Vert \mathbb{P}_{0, K}^{\perp}\mathbb{P}_{\mathcal{F}_K}e^{-\mathcal{T}(\Theta^{\Pi}_{K, \rm GS})} e^{-\mathcal{T}(\Theta^{\Pi}_{K, \rm GS})^{\dagger}}\Upsilon_K\Vert_{\mathcal{F}_K}}\\[1em]
		&= \widetilde{\gamma}\frac{\Vert \mathbb{P}_{\mathcal{F}_K} e^{-\mathcal{T}(\Theta^{\Pi}_{K, \rm GS})^{\dagger}}\Upsilon_K\Vert_{\mathcal{F}_K}}{\Vert \Upsilon_K\Vert_{\mathcal{F}_K}} \frac{\Vert e^{-\mathcal{T}(\Theta^{\Pi}_{K, \rm GS})^{\dagger}}\Upsilon_K\Vert_{\mathcal{F}_K}}{\Vert \mathbb{P}_{0, K}^{\perp}\mathbb{P}_{\mathcal{F}_K}e^{-\mathcal{T}(\Theta^{\Pi}_{K, \rm GS})} e^{-\mathcal{T}(\Theta^{\Pi}_{K, \rm GS})^{\dagger}}\Upsilon_K\Vert_{\mathcal{F}_K}}\\[1em]
		&\geq \widetilde{\gamma}\frac{1}{\Vert \mathbb{P}_{\mathcal{F}_K}e^{\mathcal{T}(\Theta^{\Pi}_{K, \rm GS})^{\dagger}}\mathbb{P}_{\mathcal{F}_K}\Vert_{\mathcal{F}_K \to \mathcal{F}_K} }\frac{ \Vert e^{-\mathcal{T}(\Theta^{\Pi}_{K, \rm GS})^{\dagger}}\Upsilon_K\Vert_{\mathcal{F}_K}}{\Vert \mathbb{P}_{0, K}^{\perp}\mathbb{P}_{\mathcal{F}_K}e^{-\mathcal{T}(\Theta^{\Pi}_{K, \rm GS})}\mathbb{P}_{\mathcal{F}_K}\Vert_{\mathcal{F}_K \to \mathcal{F}_K}\Vert  e^{-\mathcal{T}(\Theta^{\Pi}_{K, \rm GS})^{\dagger}}\Upsilon_K\Vert_{\mathcal{F}_K}}\\
		&=\widetilde{\gamma}\frac{1}{\Vert \mathbb{P}_{\mathcal{F}_K}e^{\mathcal{T}(\Theta^{\Pi}_{K, \rm GS})^{\dagger}}\mathbb{P}_{\mathcal{F}_K}\Vert_{\mathcal{W}_K \to \mathcal{W}_K}\Vert \mathbb{P}_{0, K}^{\perp} \mathbb{P}_{\mathcal{F}_K}e^{-\mathcal{T}(\Theta^{\Pi}_{K, \rm GS})} \mathbb{P}_{\mathcal{F}_K}\Vert_{\mathcal{F}_K \to \mathcal{F}_K}}.
	\end{align*}
	
	Here, the second to last step uses the invertibility of the projected exponential cluster operator, i.e., $\mathbb{P}_{\mathcal{F}_K}e^{\mathcal{T}(\Theta^{\Pi}_{K, \rm GS})^{\dagger}}\mathbb{P}_{\mathcal{F}_K}$ as a mapping from the approximation space $\mathcal{V}_K$ to $\mathcal{V}_K$. This invertibility is a consequence of the fact that $e^{\mathcal{T}(\Theta^{\Pi}_{K, \rm GS})^{\dagger}} \Phi_K \in \mathcal{V}_K$ for any $\Phi_K \in \mathcal{V}_K$, which itself follows from \textbf{Structure Assumption B.II}.
	
	Taking now the infimum over all $\Upsilon_K \in \widetilde{\mathcal{V}}_K$ completes the proof. 
	
\end{proof}

A number of remarks on the above proof of Lemma \ref{lem:inf-sup} for \textbf{Structure Assumption B.II} are in order.

\begin{remark}[The role of \textbf{Structure Assumption B.II} in the proof of Lemma \ref{lem:inf-sup}]\label{rem:B.II(1)}~
	
	Consider the above proof of Lemma \ref{lem:inf-sup}. It is instructive to briefly discuss exactly where each condition appearing in \textbf{Structure Assumption B.II} is used in our proof. 
	
	\begin{itemize}
		\item The first point in  \textbf{Structure Assumption B.II}, pertaining to the closedness of the approximation space $\mathcal{V}_K$ under the action of excitation and de-excitation operators is used to ensure that the chosen test function $\Phi_K$, which is constructed by applying certain exponential cluster operators to an element $\Upsilon_K \in \mathcal{V}_K$ and then taking a projection, is indeed an element of the approximation space $\mathcal{V}_K$. Additionally, the so-called de-excitation completeness condition, i.e., $e^{-\mathcal{T}(\Theta^{\Pi}_{K, \rm GS})^{\dagger}}\Upsilon_K \in \mathcal{V}_K$ for all $\Upsilon_K$ is used to argue the invertibility of the exponential cluster operator in the approximation space, which is a crucial step in our proof.
		
		\item The fact that the approximation space $\mathcal{V}_K$ is an invariant subspace of the restricted mean-field operator $\widetilde{\mathcal{F}}_K \colon \mathcal{W}_K\rightarrow \mathcal{W}_K$ has two uses. On the one hand, it implies that the projection operator $\mathbb{P}_{\mathcal{F}_K}$ onto $\mathcal{V}_K$, which is a priori defined with respect to the inner product generated by the mean-field operator $\mathcal{F}_K$, is simultaneously $\widehat{\mathcal{L}}^2$-orthogonal. On the other hand, this invariance property plays a key role in our estimate of the term {\rm(II*)} as it allows us to (eventually) conclude that the term {\rm(II*)} is small.
		
		\item The fact that the mean-field operator $\mathcal{F}_K-\Lambda_0$ is symmetric and coercive on the finite-dimensional space $\mathcal{W}_K$ with coercivity and continuity constants independent of $K$ allows us to introduce a new inner-product $(\cdot, \cdot)_{\mathcal{F}_K}$ on $\mathcal{W}_K$ which is equivalent to the usual $\widehat{\mathcal{H}}^1$ inner product. This allows us (see above) to construct a projection operator $\mathbb{P}_{\mathcal{F}_K}$ which is simultaneously $\widehat{\mathcal{L}}^2$-orthogonal and $(\cdot, \cdot)_{\mathcal{F}_K}$ orthogonal. These simultaneous orthogonalities aid us in estimating the term {\rm(II*)} using, what basically amount to, spectral arguments. 
		
		\item Finally, the splitting of the electronic Hamiltonian $H$ in terms of the mean-field operator $\mathcal{F}_K$ and $\mathcal{U}_K \colon \widehat{\mathcal{H}}^1 \rightarrow \widehat{\mathcal{L}}^2$ is used crucially in the estimate of term {\rm(II*)}. The smallness assumption on the continuity constant of $\mathcal{U}_K$ is used to argue that our final bound on the term {\rm (II*)} is indeed small.
	\end{itemize}
\end{remark}
\vspace{0mm}

\begin{remark}[The role of other assumptions in the proof of Lemma \ref{lem:inf-sup}]\label{rem:B.II(2)}~
	
	Consider again the above proof of Lemma \ref{lem:inf-sup}. In addition to \textbf{Structure Assumption B.II}, we also use \textbf{Assumptions A.I-A.III} and of course the density of the approximation spaces $\{\mathcal{V}_K\}_{K\geq N}$ in $\widehat{\mathcal{H}}^1$. These results are used to make asymptotic arguments pertaining to the approximability of the exact ground state zero $\Theta^*_{K, \rm GS} \in \widehat{\mathcal{H}}^{1, \perp}_{\Psi_{0, K}}$ of the coupled cluster function in the approximation space $\widetilde{\mathcal{V}}_K$. More precisely, we use these assumptions to argue that the term {\rm (I*)} appearing in the above proof is indeed positive and that the term {\rm (III*)} goes to zero asymptotically as $K \to \infty$.
	
	\item From a close study of the method of our proof, it can be deduced that the density of the approximation spaces is not a necessary condition. Indeed, for a given choice of approximation space $\mathcal{V}_K$, the question of whether the discrete inf-sup condition \eqref{eq:inf-sup} holds depends (in addition to \textbf{Structure Assumption B.II}) on how well the exact ground state coupled cluster zero $\Theta^*_{K, \rm GS} $ is approximated in~$\widetilde{\mathcal{V}}_K$.
	
\end{remark}

Having completed our second proof of Lemma \ref{lem:inf-sup}, let us finally turn to the question of how restrictive \textbf{Structure Assumption B.II} really is.

\begin{remark}[Restrictiveness of \textbf{Structure Assumption B.II} ]\label{rem:B.II(3)}~
	
	As we discuss in more detail in Appendix \ref{sec:spaces}, the first two points of this assumption are satisfied by standard coupled cluster discretisations based on an initial Hartree-Fock computation in a finite basis with the so-called $N$-particle Hartree-Fock operator (defined below in Appendix \ref{sec:HF}) playing the role of the mean-field operator $\mathcal{F}_K \colon \widehat{\mathcal{H}}^1 \rightarrow \widehat{\mathcal{H}}^{-1}$ and it's lowest eigenvalue $\Lambda_0$ being used to create the shift. We emphasise that this is true not just for coupled cluster methods based on the canonical Hartree-Fock orbitals but even for rotated orbitals as long as the rotations maintain the $\mL^2$-orthogonality of occupied and virtual spaces. 
	
	The third point is also satisfied by such Hartree-Fock discretisations, provided that we assume convergence of the discrete Hartree-Fock ground state towards some limiting Slater determinant in the complete basis set limit. As far as we are aware, this complete basis set convergence of the Hartree-Fock method has not been rigorously established in the literature although it seems reasonable to assume.
	
	The fourth condition in \textbf{Structure Assumption B.II} is by far the most restrictive. Indeed, while it can be shown (again, see Appendix \ref{sec:HF} below) that the $N$-particle Hartree-Fock operator $\mathcal{F}_K$ can be used to split the electronic Hamiltonian $H$ as $H= \mathcal{F}_K + \mathcal{U}_K$ with $\mathcal{U}_K \colon \widehat{\mathcal{H}}^1 \rightarrow \widehat{\mathcal{L}}^2$, the smallness of the continuity constant of $\mathcal{U}_K$ is problem-dependent and not a universal property. To be more precise, recalling the above proof of Lemma \ref{lem:inf-sup} for \textbf{Structure Assumption B.II}, we see that the precise condition that must be satisfied is (c.f., Point 4) in \textbf{Structure Assumption B.II})
	\begin{equation}\label{eq:lemma_inf-sup_13}
		\left \Vert \mathbb{P}^{\perp}_{\mathcal{F}_K}\widetilde{\mathcal{U}}_K \mathbb{P}_{\mathcal{F}_K}\right\Vert_{\mathcal{F}_K \to \widehat{\mathcal{L}}^2}< \frac{1}{2}\frac{\sqrt{\Lambda_{\min}^{\mathcal{F}_K}}\;\Gamma^{*}_{\rm GS}}{\left \Vert \mathbb{P}^{\perp}_{\mathcal{F}_K}{e^{\mathcal{T}(\Theta^{\Pi}_{K, \rm GS})}}\mathbb{P}_{\mathcal{F}_K}{e^{-\mathcal{T}(\Theta^{\Pi}_{K, \rm GS})}} \mathbb{P}_{\mathcal{F}_K}\right\Vert_{\mathcal{F}_K \to \mathcal{F}_K}},
	\end{equation}
	where we remind the reader that the constant $\Gamma_{\rm GS}^*$ is the coercivity constant of the shifted electronic Hamiltonian  $H-\mathcal{E}_{\rm GS}^*$ on the finite-dimensional space $\{\Psi_{{\rm GS}, K}^*\}^{\perp}\cap \mathcal{W}_K$, while the constant $\Lambda_{\rm min}^{\mathcal{F}_K}>0$ denotes the minimum of the spectrum of the shifted, restricted mean-field operator $\widetilde{\mathcal{F}}_K-\Lambda_0$ in the range of the projection operator $\mathbb{P}^{\perp}_{\mathcal{F}_K} \colon \mathcal{W}_K \rightarrow \mathcal{W}_K$, i.e.,
	\begin{align*}
		\Lambda_{\rm min}^{\mathcal{F}_K}= \min_{\substack{\Phi_K \in \mathcal{W}_K\\ \Phi_K \in \mathcal{V}_K^{\perp}}} \frac{\langle \Phi_K, (\mathcal{F}_K-\Lambda_0) \Phi_K\rangle_{\widehat{\mathcal{H}}^1 \times \widehat{\mathcal{H}}^{-1}}}{\Vert \Phi_K \Vert_{\widehat{\mathcal{L}}^2}^2 }.
	\end{align*}
	
	The coercivity constant $\Gamma_{\rm GS}^*$ depends on the spectral gap of the specific molecular system being studied while the operator norm appearing in Inequality \eqref{eq:lemma_inf-sup_13} depends on the norm of the exact ground state coupled cluster zero $\Theta_{K, \rm GS}$ which is also molecule-dependent. Moreover, even though we have assumed the density of the approximation spaces $\{\mathcal{V}_K\}_{K\geq N}$ in $\widehat{\mathcal{H}}^1$, it is \underline{not the case} that the constant $\Lambda_{\rm min}^{\mathcal{F}_K} \to \infty$ as $K \to \infty$. This is due to the fact that the Laplacian, considered as an operator on the unbounded domain $\mathbb{R}^3$, has an essential spectrum consisting of $[0, \infty)$. Thus any mean-field operator ${\mathcal{F}}_K \colon \widehat{\mathcal{H}}^1 \rightarrow \widehat{\mathcal{H}}^{-1}$ constructed according to the stipulations of \textbf{Structure Assumption B.II} (and, in particular, the $N$-particle Hartree-Fock operator) will also posses an essential spectrum that contains the interval $[0, \infty)$. Having said this, for the specific case of the $N$-particle Hartree-Fock operator and the so-called truncated coupled cluster equations which are based on excitation-rank truncations (such as CCSD, CCSDT etc.), lower bounds can be derived for $\Lambda_{\rm min}^{\mathcal{F}_K}$ in terms of the so-called Hartree-Fock excitation energies. In particular, thanks to the density of the approximation spaces $\{\mathcal{V}_K\}_{K \geq N}$ in $\widehat{\mathcal{H}}^1$, we can deduce that for $K$ sufficiently large, a simple upper bound for $\Lambda_{\rm min}^{\mathcal{F}_K}$ is given by the sum of the first $N$ excitation energies of the $N$-particle Hartree-Fock operator $\mathcal{F}_K$, which is typically a large quantity (see Appendix \ref{sec:spaces}).
	
	The upshot of the above discussion is that, for an arbitrary molecular system, we are, unfortunately, unable to guarantee that the continuity constant of the operator $\mathcal{U}_K$ is sufficiently small, even in the asymptotic limit $K\to\infty$. We are therefore also unable to guarantee unconditional validity of the discrete inf-sup estimate \eqref{eq:inf-sup_old} in the asymptotic limit $K \to\infty$ for arbitrary molecules. Nevertheless, as we discuss in the sequel, our numerical results indicate that the smallness condition \eqref{eq:lemma_inf-sup_13} is often met for small molecules, particularly if the approximation space $\mathcal{V}_K$ is sufficiently rich.
\end{remark}

\vspace{2mm}

\subsection{Main Results on the Well-posedness of the Discrete Coupled Cluster Equations}\label{sec:4c}~

Throughout this subsection, we assume the setting of Sections \ref{sec:4a} and \ref{sec:4b}. Equipped with the technical lemmas that we have proved in these section, we are now ready to state and prove the main result of the present article.

\begin{theorem}[Local Well-posedness of the Discrete Coupled Cluster Equations]\label{thm:Galerkin_truncated_CC}~
	
	Let the sequence of reference determinants $\{\Psi_{0, K}\}_{K\geq N}$ and the sequence of approximation spaces  $\{\widetilde{\mathcal{V}}_K\}_{K\geq N}$ be constructed as in Notation \ref{def:seq_ref} and Notation \ref{def:V_K} respectively, and assume that {\textbf{Assumptions A.I-A.III}} and either \textbf{Structure Assumption B.I} or \textbf{Structure Assumption B.II} holds.
	
	Additionally, for each $K\geq N$, let $\{\Psi_{0, K}\}^{\perp}$ be the $\widehat{\mathcal{L}}^2$ orthogonal complement of $\Psi_{0, K}$ in $\widehat{\mathcal{L}}^2$, let the infinite-dimensional subspace $\widehat{\mathcal{H}}^{1, \perp}_{\Psi_{0, K}}$ be defined as $\widehat{\mathcal{H}}^{1, \perp}_{\Psi_{0, K}}:= \{\Psi_{0, K}\}^{\perp} \cap \widehat{\mathcal{H}}^1$, let the coupled cluster function $\mathcal{f}_K \colon \widehat{\mathcal{H}}^{1, \perp}_{\Psi_{0, K}} \rightarrow \left(\widehat{\mathcal{H}}^{1, \perp}_{\Psi_{0, K}}\right)^*$ be defined through Equations~\eqref{eq:12bis}, and let $\Theta_{{K}, \rm GS}^*$ denote the zero of the coupled cluster function $\mathcal{f}_K$ corresponding to the ground state eigenfunction $\Psi^*_{\rm GS} \in \widehat{\mathcal{H}}^1$ of the electronic Hamiltonian, i.e.,
 \begin{align*}
     e^{\mathcal{T}(\Theta_{K, \rm GS}^*)}\Psi_{0, K}= \frac{1}{\left(\Psi^*_{\rm GS}, \Psi_{0, K}\right)_{\widehat{\mathcal{L}}^2}} \Psi^*_{\rm GS}.
 \end{align*}

	
	Consider now for each $K\geq N$, the discrete coupled cluster problem which consists of seeking $\Theta_K \in \widetilde{\mathcal{V}}_K$ such that (c.f., Equation \eqref{eq:CC_discrete_2})
	\begin{equation}\label{eq:CC_discrete_2*}
		\forall \Phi_K \in \widetilde{\mathcal{V}}_K\colon \qquad \left\langle \Phi_K, \mathcal{f}_K(\Theta_K)\right\rangle_{\widehat{\mathcal{H}}^{1, \perp}_{\Psi_{0, K}} \times \left(\widehat{\mathcal{H}}^{1, \perp}_{\Psi_{0, K}}\right)^*}=0.
	\end{equation}
	
	There exists $\delta > 0$ and $\overline{{K}_0^*}\geq N$ such that for all $K\geq \overline{K_0^*}$, there exists a unique  $\Theta_K \in \widetilde{\mathcal{V}}_K \cap \mathbb{B}_{\delta}(\Theta^*_{K, \rm GS})$ that solves the discrete coupled cluster equation \eqref{eq:CC_discrete_2*}, and we have the quasi-optimality result
	\begin{align}\label{eq:quasi_optimality}
		\exists C>0, ~ \forall K\geq \overline{K_0^*} \colon \qquad \Vert \Theta_K - \Theta^*_{K, \rm GS}\Vert_{\widehat{\mathcal{H}}^1} &\leq  C \inf_{\Phi_K \in \widetilde{\mathcal{V}}_K} \Vert \Phi_K - \Theta^*_{K, \rm GS}\Vert_{\widehat{\mathcal{H}}^1}, \\
		\intertext{as well as the residual-based error estimate}
		\forall K\geq \overline{K_0^*} \colon \qquad  \Vert \Theta_K - \Theta^*_{K, \rm GS}\Vert_{\widehat{\mathcal{H}}^1} &\leq  2 \Vert \mD \mathcal{f}_K\left(\Theta_K\right)^{-1}\Vert_{\left(\widehat{\mathcal{H}}^{1, \perp}_{\Psi_{0, K}}\right)^*\to \widehat{\mathcal{H}}^{1, \perp}_{\Psi_{0, K}}}\; \Vert \mathcal{f}_K\left(\Theta_K\right)\Vert_{\left(\widehat{\mathcal{H}}^{1, \perp}_{\Psi_{0, K}}\right)^*}. \label{eq:residual_error}
	\end{align}
\end{theorem}
\begin{proof}
	The proof that we construct is based on the arguments developed by Caloz and Rappaz \cite[Theorems 6.1, 6.2, and 7.1]{MR1470227} for conforming Galerkin discretisations of non-linear equations but adapted to the present setting. Throughout this proof, we will assume the setting of, and use results from, Sections \ref{sec:4a} and \ref{sec:4b}. It will be useful, in particular, to recall that $\Psi_0^* \in \widehat{\mathcal{H}}^1$ denotes the limiting reference determinant defined in \textbf{Assumption A.II}. Moreover, denoting by $\{\Psi_{0}^*\}^{\perp}$, the $\widehat{\mathcal{L}}^2$ orthogonal complement of $\Psi_{0}^*$ in $\widehat{\mathcal{L}}^2$ and defining the infinite-dimensional subspace $\widehat{\mathcal{H}}^{1, \perp}_{\Psi_{0}^*}$ as $\widehat{\mathcal{H}}^{1, \perp}_{\Psi_{0}^*}:= \{\Psi_{0}^*\}^{\perp} \cap \widehat{\mathcal{H}}^1$, we can introduce the coupled cluster function $\mathcal{f} \colon \widehat{\mathcal{H}}^{1, \perp}_{\Psi_{0}^*} \rightarrow \left(\widehat{\mathcal{H}}^{1, \perp}_{\Psi_{0}^*}\right)^*$ defined through Equation~\eqref{eq:12bis}. In this context, we denote by $\Theta_{\rm GS}^*$ the zero of the coupled cluster function $\mathcal{f}$ corresponding to the ground state eigenfunction of the electronic Hamiltonian.
	
	We begin by defining for every $K\geq N$, the bilinear form $\mathcal{b}_K \colon \widehat{\mathcal{H}}^{1, \perp}_{\Psi_{0, K}} \times \widehat{\mathcal{H}}^{1, \perp}_{\Psi_{0, K}} \rightarrow \mathbb{R}$ given by
	\begin{align}\label{eq:bilinear}
		\forall \Phi, \Upsilon \in \widehat{\mathcal{H}}^{1, \perp}_{\Psi_{0, K}} \colon \qquad \mathcal{b}_K(\Phi, \Upsilon) = \langle \Upsilon, \mD \mathcal{f}_K(\Theta_{K, \rm GS}^*)\Phi\rangle_{\widehat{\mathcal{H}}^{1, \perp}_{\Psi_{0, K}} \times \left(\widehat{\mathcal{H}}^{1, \perp}_{\Psi_{0, K}}\right)^*},
	\end{align}
	where we remind the reader that $ \mD \mathcal{f}_K(\Theta_{K, \rm GS}^*) \colon \widehat{\mathcal{H}}^{1, \perp}_{\Psi_{0, K}}  \rightarrow \left(\widehat{\mathcal{H}}^{1, \perp}_{\Psi_{0, K}} \right)^*$ is the Fr\'echet derivative of the coupled cluster function $\mathcal{f}_k$ at $\Theta_{K, \rm GS}^*$. 
	
	It follows directly from Lemma \ref{lem:inf-sup} that there exists $K_0 \geq N$ such that for all $K\geq K_0$, the bilinear form $\mathcal{b}_K$ satisfies a discrete inf-sup condition on the approximation space $\widetilde{\mathcal{V}}_K$. 
	
	Consequently, for any $K\geq K_0$ we can introduce projection operators $\Pi_K^{\rm L}\colon \widehat{\mathcal{H}}^{1, \perp}_{\Psi_{0, K}} \rightarrow \widetilde{\mathcal{V}}_K$ and $\Pi_K^{\rm R}\colon$ $\widehat{\mathcal{H}}^{1, \perp}_{\Psi_{0, K}} \rightarrow \widetilde{\mathcal{V}}_K$ through the relations
	\begin{align}\label{eq:proof_proj_1}
		\text{For any }\Phi \in \widehat{\mathcal{H}}^{1, \perp}_{\Psi_{0, K}}, ~\forall \Upsilon_K \in \widetilde{\mathcal{V}}_K \colon \qquad \mathcal{b}_K\left(\Phi - \Pi_K^{\rm L}\Phi, \Upsilon_K\right)=0,\\[0.5em]
		\text{For any }\Upsilon \in \widehat{\mathcal{H}}^{1, \perp}_{\Psi_{0, K}}, ~\forall \Phi_K \in \widetilde{\mathcal{V}}_K \colon \qquad \mathcal{b}_K\left(\Phi_K, \Upsilon-\Pi_K^{\rm R}\Upsilon\right)=0. \label{eq:proof_proj_2}
	\end{align}
	Note that both projection operators are well-defined thanks to the validity of the discrete inf-sup condition.
	
	Equipped with the projections operators $\Pi_K^{\rm L}, \Pi_K^{\rm R}$, we next define for each $K\geq K_0$, the non-linear mapping $\mathcal{G}_{K} \colon \widehat{\mathcal{H}}^{1, \perp}_{\Psi_{0, K}}  \rightarrow (\widehat{\mathcal{H}}^{1, \perp}_{\Psi_{0, K}} )^{*}$ given by
	\begin{equation}\label{eq:discrete_cc_func}
		\forall \Phi, \Upsilon \in \widehat{\mathcal{H}}^{1, \perp}_{\Psi_{0, K}} \colon \quad \left\langle \Upsilon, \mathcal{G}_K(\Phi)\right\rangle_{\widehat{\mathcal{H}}^{1, \perp}_{\Psi_{0, K}}\times \left(\widehat{\mathcal{H}}^{1, \perp}_{\Psi_{0, K}}\right)^*}\;:=  \left\langle \Pi_K^{\rm R}\Upsilon, \mathcal{f}_K(\Phi)\right\rangle_{\widehat{\mathcal{H}}^{1, \perp}_{\Psi_{0, K}}\times \left(\widehat{\mathcal{H}}^{1, \perp}_{\Psi_{0, K}}\right)^*}\; + \mathcal{b}_K \left(\Phi, \Upsilon-\Pi_K^{\rm R}\Upsilon\right).
	\end{equation}
	
	We now claim that for any $K\geq K_0$, the function $\Theta_K \in \widetilde{\mathcal{V}}_K$ solves the discrete coupled cluster equations~\eqref{eq:CC_discrete_2*} if and only if $\Theta_K$ solves the infinite-dimensional non-linear equation
	\begin{align}\label{eq:equivalent}
		\forall \Upsilon \in \widehat{\mathcal{H}}^{1, \perp}_{\Psi_{0, K}} \colon\qquad \left\langle \Upsilon, \mathcal{G}_K(\Theta_K)\right\rangle_{\widehat{\mathcal{H}}^{1, \perp}_{\Psi_{0, K}}\times \left(\widehat{\mathcal{H}}^{1, \perp}_{\Psi_{0, K}}\right)^*}=0.
	\end{align}
	
	Indeed, assume $\Theta_K \in \widetilde{\mathcal{V}}_K$ solves the discrete coupled cluster equations \eqref{eq:CC_discrete_2*}. Then it follows directly from the definition of the projection operator $\Pi_{K}^{\rm R}$ given by Equation \eqref{eq:proof_proj_2} that $\Theta_K$ also solves Equation~\eqref{eq:equivalent}. Conversely, let $\Theta_K \in \widehat{\mathcal{H}}^{1, \perp}_{\Psi_{0, K}} $ be a solution to Equation \eqref{eq:equivalent}. We must show that $\Theta_K \in \widetilde{\mathcal{V}}_K$. 
	
	To this end, notice that Equations \eqref{eq:discrete_cc_func} and \eqref{eq:equivalent} together imply that for all $\Upsilon \in \widehat{\mathcal{H}}^{1, \perp}_{\Psi_{0, K}}$ it holds that
	\begin{align*}
		0= \left\langle \Upsilon- \Pi_K^{\rm R}\Upsilon, \mathcal{G}_K(\Theta_K)\right\rangle_{\widehat{\mathcal{H}}^{1, \perp}_{\Psi_{0, K}} \times \left(\widehat{\mathcal{H}}^{1, \perp}_{\Psi_{0, K}}\right)^*}=\mathcal{b}_K(\Theta_K, \Upsilon - \Pi_K^{\rm R}\Upsilon).
	\end{align*}
	Thanks to Definitions \eqref{eq:proof_proj_1} and \eqref{eq:proof_proj_2} of the projection operators $\Pi_K^{\rm L}$ and $\Pi_K^{\rm R}$ respectively, we can then deduce that for all $\Upsilon \in \widehat{\mathcal{H}}^{1, \perp}_{\Psi_{0, K}}$ it holds that
	\begin{align*}
		0=\mathcal{b}_K(\Theta_K, \Upsilon - \Pi_K^{\rm R}\Upsilon)=\mathcal{b}_K(\Theta_K-\Pi_K^{\rm L}\Theta_K, \Upsilon).
	\end{align*}
	Recalling now the definition of bilinear form $\mathcal{b}_K\colon \colon \widehat{\mathcal{H}}^{1, \perp}_{\Psi_{0, K}} \times \widehat{\mathcal{H}}^{1, \perp}_{\Psi_{0, K}} \rightarrow \mathbb{R}$ given by Equation \eqref{eq:bilinear} and using the fact that the Fr\'echet derivative  $ \mD \mathcal{f}_K(\Theta_{K, \rm GS}^*) \colon \widehat{\mathcal{H}}^{1, \perp}_{\Psi_{0, K}}  \rightarrow \left(\widehat{\mathcal{H}}^{1, \perp}_{\Psi_{0, K}} \right)^*$ is an isomorphism (see Theorem \ref{thm:CC_der_inv}), we conclude that $\Theta_K=\Pi_K^{\rm L}\Phi$, or equivalently, $\Theta_K \in \widetilde{\mathcal{V}}_K$ as claimed. In particular, $\Theta_K\in \widetilde{\mathcal{V}}_K$ is a solution to the discrete coupled cluster equations \eqref{eq:CC_discrete_2*}.
	
	\vspace{3mm}
	We will now analyse the (local) well-posedness of the equivalent Equation \eqref{eq:equivalent}. To do so, we first prove four properties of the sequence of non-linear functions $\{\mathcal{G}_K\}_{K\geq K_0}$. \vspace{3mm}

	\noindent \textbf{Claim One:} The mapping $\widehat{\mathcal{H}}^{1, \perp}_{\Psi_{0, K}} \ni \Theta \mapsto \mathcal{G}_K(\Theta)$ is of class $\mathscr{C}^{1}$ for all $K\geq K_0$.

	To demonstrate this first claim, we use the fact (see, e.g., \cite[Proposition 26]{Hassan_CC}) that for any $K\geq N$ the coupled cluster function $\mathcal{f}_K \colon \widehat{\mathcal{H}}^{1, \perp}_{\Psi_{0, K}}\rightarrow \left(\widehat{\mathcal{H}}^{1, \perp}_{\Psi_{0, K}}\right)^*$ is a $\mathscr{C}^{\infty}$ mapping. It then follows directly from Equation \eqref{eq:discrete_cc_func} that $\mathcal{G}_K$ is also a $\mathscr{C}^{\infty}$ mapping for any $K\geq K_0$. In particular, for any $\Theta \in \widehat{\mathcal{H}}^{1, \perp}_{\Psi_{0, K}}$ the Fr\'echet derivative $\mD\mathcal{G}_K(\Theta)\colon  \widehat{\mathcal{H}}^{1, \perp}_{\Psi_{0, K}} \rightarrow \left(\widehat{\mathcal{H}}^{1, \perp}_{\Psi_{0, K}}\right)^*$ is given by
	\begin{align}\label{eq:Frechet_G}
		\forall \Phi, \Upsilon \in \widehat{\mathcal{H}}^{1, \perp}_{\Psi_{0, K}} \colon \quad \left\langle \Upsilon, \mD \mathcal{G}_K(\Theta)\Phi\right\rangle_{\widehat{\mathcal{H}}^{1, \perp}_{\Psi_{0, K}} \times \left(\widehat{\mathcal{H}}^{1, \perp}_{\Psi_{0, K}}\right)^*}:=&  \left\langle \Pi_K^{\rm R}\Upsilon, \mD \mathcal{f}_K(\Theta)\Phi\right\rangle_{\widehat{\mathcal{H}}^{1, \perp}_{\Psi_{0, K}} \times\left(\widehat{\mathcal{H}}^{1, \perp}_{\Psi_{0, K}}\right)^*}\\
  +& \mathcal{b}_K \left(\Phi, \Upsilon-\Pi_K^{\rm R}\Upsilon\right). \nonumber
	\end{align}

	\vspace{3mm}

	\noindent \textbf{Claim Two:} There exists $\delta_0 >0$ and $L_0> 0$ such that for all $K \geq K_0$ it holds that 
	\begin{align*}
		\forall \Theta \in \widehat{\mathcal{H}}^{1, \perp}_{\Psi_{0, K}}  \cap \mathbb{B}_{\delta_0}(\Theta_{K, \rm GS}^*) \colon \quad \Vert \mD \mathcal{G}_K(\Theta_{K, \rm GS}^*) - \mD \mathcal{G}_K(\Theta)\Vert_{\widehat{\mathcal{H}}^{1, \perp}_{\Psi_{0, K}}\to \left(\widehat{\mathcal{H}}^{1, \perp}_{\Psi_{0, K}}\right)^*} \leq L_0 \Vert \Theta_{K, \rm GS}^*- \Theta\Vert_{\widehat{\mathcal{H}}^1}.
	\end{align*}
	
	To demonstrate the second claim, we begin by using Equation \eqref{eq:Frechet_G} to deduce that for all $K\geq K_0$ and all $\Theta, \Upsilon, \Phi \in \widehat{\mathcal{H}}^{1, \perp}_{\Psi_{0, K}}$ it holds that
	\begin{align}\nonumber
		\left\langle \Upsilon, \left(\mD \mathcal{G}_K(\Theta_{K, \rm GS}^*)-\mD \mathcal{G}_K(\Theta)\right)\Phi\right\rangle_{\widehat{\mathcal{H}}^{1, \perp}_{\Psi_{0, K}} \times \left(\widehat{\mathcal{H}}^{1, \perp}_{\Psi_{0, K}}\right)^*}&=  \left\langle \Pi_K^{\rm R}\Upsilon, \left(\mD \mathcal{f}_K(\Theta_{K, \rm GS}^*) -\mD \mathcal{f}_K(\Theta)\right)\Phi\right\rangle_{\widehat{\mathcal{H}}^{1, \perp}_{\Psi_{0, K}} \times \left(\widehat{\mathcal{H}}^{1, \perp}_{\Psi_{0, K}}\right)^*}\\[0.5em]
		&\leq \Vert \Pi_K^{\rm R}\Upsilon\Vert_{\widehat{\mathcal{H}}^{1}} \Vert \left(\mD \mathcal{f}_K(\Theta_{K, \rm GS}^*) -\mD \mathcal{f}_K(\Theta)\right)\Phi\Vert_{\left(\widehat{\mathcal{H}}^{1, \perp}_{\Psi_{0, K}}\right)^*}. \label{eq:main_thm_1}
	\end{align}
	
	To proceed, we must bound the term $\Vert \Pi_K^{\rm R}\Upsilon\Vert_{\widehat{\mathcal{H}}^{1}}$. To this end, we observe that thanks to Equation~\eqref{eq:proof_proj_1}, for every $K\geq K_0$ and any $\Upsilon \in \widehat{\mathcal{H}}^{1, \perp}_{\Psi_{0, K}}$, the function $\Pi_K^{\rm R}\Upsilon$ is a solution to the following adjoint problem:
	\begin{align*}
		\forall \Phi_K \in \widetilde{\mathcal{V}}_K \colon \quad \mathcal{b}_K\left(\Phi_K, \Pi_K^{\rm R}\Upsilon\right)=\mathcal{b}_K\left(\Phi_K, \Upsilon\right).
	\end{align*}

	Recalling the definition of the bilinear form $\mathcal{b}_{K}$ given by Equation \eqref{eq:bilinear} and using Lemma \ref{lem:inf-sup}, which implies that the Fr\'echet derivative $\mD \mathcal{f}_K(\Theta_{K, \rm GS}^*) $ satisfies a discrete inf-sup condition on $\widetilde{\mathcal{V}}_K$ with constant $\gamma >0$, a simple calculation yields that
	\begin{align*}
		\forall K\geq K_0, ~\forall \Upsilon \in  \widehat{\mathcal{H}}^{1, \perp}_{\Psi_{0, K}} \colon \qquad    \Vert \Pi_K^{\rm R}\Upsilon\Vert_{\widehat{\mathcal{H}}^{1}} \leq \frac{\Vert \mD \mathcal{f}_K(\Theta_{K, \rm GS}^*)\Vert_{\widehat{\mathcal{H}}^{1, \perp}_{\Psi_{0, K}} \to \left(\widehat{\mathcal{H}}^{1, \perp}_{\Psi_{0, K}}\right)^*} }{\gamma}\Vert \Upsilon\Vert_{\widehat{\mathcal{H}}^{1}}.
	\end{align*}
 
	Using now the expression for the Fr\'echet derivative $\mD \mathcal{f}_K(\Theta_{K, \rm GS}^*) \colon \widehat{\mathcal{H}}^{1, \perp}_{\Psi_{0, K}}\rightarrow \left(\widehat{\mathcal{H}}^{1, \perp}_{\Psi_{0, K}}\right)^*$ found in \cite[Proposition 26]{Hassan_CC}, we further deduce that for all $K\geq K_0$ and all $\Upsilon \in  \widehat{\mathcal{H}}^{1, \perp}_{\Psi_{0, K}} $ it holds that
	\begin{align*}
		\Vert \Pi_K^{\rm R}\Upsilon\Vert_{\widehat{\mathcal{H}}^{1}} \leq \frac{\Vert e^{-\mathcal{T}(\Theta_{K, \rm GS}^*)}\Vert_{\widehat{\mathcal{H}}^{-1} \to \widehat{\mathcal{H}}^{-1}} \Vert H-\mathcal{E}^*_{\rm GS}\Vert_{\widehat{\mathcal{H}}^{1} \to \widehat{\mathcal{H}}^{-1}}\Vert e^{\mathcal{T}(\Theta_{K, \rm GS}^*)}\Vert_{\widehat{\mathcal{H}}^{1} \to \widehat{\mathcal{H}}^{1}}}{\gamma}\Vert \Upsilon\Vert_{\widehat{\mathcal{H}}^{1}}.
	\end{align*}
 
	Finally, making use of the uniform boundedness of the sequence of coupled cluster zeros $\{\Theta^*_{K, \rm GS}\}_{K \geq N}$ given by Lemma \ref{lem:approx}, we deduce the existence of a constant $\breve{\beta}>0$ such that for all $K\geq K_0$ and all $\Upsilon \in  \widehat{\mathcal{H}}^{1, \perp}_{\Psi_{0, K}} $ it holds that
	\begin{align}\label{eq:main_thm_2}
		\Vert \Pi_K^{\rm R}\Upsilon\Vert_{\widehat{\mathcal{H}}^{1}} \leq \frac{\breve{\beta}}{\gamma}\Vert \Upsilon\Vert_{\widehat{\mathcal{H}}^{1}},
	\end{align}
	where we emphasise that both $\breve{\beta}, \gamma >0$ are independent of $K$.
	
	In view of Inequalities \eqref{eq:main_thm_1} and \eqref{eq:main_thm_2}, it remains to prove that for every $K\geq K_0$ the Fr\'echet derivative $\mD \mathcal{f}_K$ of the coupled cluster function $\mathcal{f}_K$ is locally Lipschitz continuous in a neighbourhood around $\Theta_{K, \rm GS}^*$ with Lipschitz constant uniformly bounded in $K$. While this local Lipschitz continuity has been established in prior results (see, e.g., \cite[Proposition 26]{Hassan_CC}), the Lipschitz constants obtained therein have not been shown to be uniform in $K$. We now state a brief argument to bridge this gap. 
 
 To do so, we will make use of the expression for the Fr\'echet derivative of the coupled cluster function found in our previous contribution \cite[Proposition 26]{Hassan_CC}. Using this expression, we deduce that for all $K\geq K_0$ and all $\Theta, \Phi \in \widehat{\mathcal{H}}^{1, \perp}_{\Psi_{0, K}}$ it holds that
	\begin{align*}
		&\Vert \left(\mD \mathcal{f}_K(\Theta_{K, \rm GS}^*) -\mD \mathcal{f}_K(\Theta)\right)\Phi\Vert_{\left(\widehat{\mathcal{H}}^{1, \perp}_{\Psi_{0, K}}\right)^*}\\
        \leq &\Vert \left(\mD \mathcal{f}_K(\Theta_{K, \rm GS}^*) -\mD \mathcal{f}_K(\Theta)\right)\Phi\Vert_{\widehat{\mathcal{H}}^{-1}}
		\\[0.25em]
  \leq& \left \Vert e^{-\mathcal{T}(\Theta_{K, \rm GS}^*)}\comm\Big{H}{\mathcal{T}(\Phi) } e^{\mathcal{T}(\Theta_{K, \rm GS}^*)}\Psi_{0, K}-e^{-\mathcal{T}(\Theta)}\comm\Big{H}{\mathcal{T}(\Phi)} e^{\mathcal{T}(\Theta)}\Psi_{0, K}\right \Vert_{\widehat{\mathcal{H}}^{-1}},
	\end{align*}
	where $\mathcal{T}(\Phi), \mathcal{T}(\Theta)$ denote, as usual the cluster operators generated by $\Phi, \Theta \in \widehat{\mathcal{H}}^{1, \perp}_{\Psi_{0, K}}$ respectively, and $[\cdot, \cdot]$ denotes the usual commutator. 
	
	Adding and subtracting suitable terms now yields
	\begin{align}\label{eq:main_thm_1*}
		\Vert \left(\mD \mathcal{f}_K(\Theta_{K, \rm GS}^*) -\mD \mathcal{f}_K(\Theta)\right)\Phi\Vert_{\widehat{\mathcal{H}}^{-1}}
		\leq& \underbrace{\left \Vert  \left(e^{-\mathcal{T}(\Theta_{K, \rm GS}^*)} -e^{-\mathcal{T}(\Theta)}\right)\comm\Big{H}{\mathcal{T}(\Phi) } e^{\mathcal{T}(\Theta_{K, \rm GS}^*)}\Psi_{0, K}\right \Vert_{\widehat{\mathcal{H}}^{-1}}}_{:=\rm (I)}\\[0.5em]
		+&\underbrace{\left \Vert  e^{-\mathcal{T}(\Theta)}\comm\Big{H}{\mathcal{T}(\Phi) } \left(e^{\mathcal{T}(\Theta_{K, \rm GS}^*)}-e^{-\mathcal{T}(\Theta)}\right)\Psi_{0, K}\right \Vert_{\widehat{\mathcal{H}}^{-1}}}_{:=\rm (II)}. \nonumber
	\end{align}
	
	Let us first consider the term (I). We immediately observe that
	\begin{align}\label{eq:main_thm_2*}
		{\rm (I)} \leq \underbrace{\left \Vert  e^{-\mathcal{T}(\Theta_{K, \rm GS}^*)} -e^{-\mathcal{T}(\Theta)}\right \Vert_{\widehat{\mathcal{H}}^{-1} \to \widehat{\mathcal{H}}^{-1} }}_{:= \rm (IA)} \; \underbrace{\left \Vert\comm\Big{H}{\mathcal{T}(\Phi) }\right \Vert_{\widehat{\mathcal{H}}^{1} \to \widehat{\mathcal{H}}^{-1}}}_{:= \rm (IB)}  \; \underbrace{\left \Vert  e^{\mathcal{T}(\Theta_{K, \rm GS}^*)}\Psi_{0, K}\right \Vert_{\widehat{\mathcal{H}}^{1}}}_{:= \rm (IC)}.
	\end{align}
 
	We proceed term-by-term beginning with (IC). We first notice that thanks to Theorem \ref{thm:1} concerning the continuity properties of cluster operators, there exists $\beta_{\widehat{\mathcal{H}}}>0$ depending only on $N$ and $\Vert \Psi_{0, K}\Vert_{\widehat{\mathcal{H}}^1}$ such that
	\begin{align}\label{eq:main_thm_3}
		{\rm (IC)}=\left \Vert  e^{\mathcal{T}(\Theta_{K, \rm GS}^*)}\Psi_{0, K}\right \Vert_{\widehat{\mathcal{H}}^{1}}\leq e^{\beta_{\widehat{\mathcal{H}}} \Vert \Theta_{K, \rm GS}^*\Vert_{\widehat{\mathcal{H}}^{1}} }\Vert \Psi_{0, K}\Vert_{\widehat{\mathcal{H}}^{1}}.
	\end{align}
	
	Using again Theorem \ref{thm:1} and appealing to the continuity of the electronic Hamiltonian $H \colon \widehat{\mathcal{H}}^{1} \rightarrow \widehat{\mathcal{H}}^{-1}$, we further deduce the existence of a constant $\widehat{\beta}>0$ depending only on $N$, $\Vert \Psi_{0, K}\Vert_{\widehat{\mathcal{H}}^1}$ and the continuity constant of $H$ such that 
	\begin{align}
		{\rm (IB)} =\left \Vert\comm\Big{H}{\mathcal{T}(\Phi) }\right \Vert_{\widehat{\mathcal{H}}^{1} \to \widehat{\mathcal{H}}^{-1}} \leq \widehat{\beta} \Vert \Phi \Vert_{\widehat{\mathcal{H}}^{1}}.
	\end{align}
	
	Finally, relying on the series expansion of the exponential, we deduce the existence of a constant $\overline{\beta}_K>0$ depending only $\Vert \mathcal{T}(\Theta^*_{K, \rm GS})\Vert_{\widehat{\mathcal{H}}^{-1} \to \widehat{\mathcal{H}}^{-1}}$ such that for all $\Theta \in \widehat{\mathcal{H}}^{1, \perp}_{\Psi_{0, K}}$ that satisfy
	\begin{align} \label{eq:main_thm_4}
		\Vert \mathcal{T}(\Theta) - \mathcal{T}(\Theta_{K, \rm GS}^*)\Vert_{\widehat{\mathcal{H}}^{-1} \to \widehat{\mathcal{H}}^{-1}} &< 1,\\
		\intertext{it holds that}
		{\rm (IA)}= \left \Vert  e^{-\mathcal{T}(\Theta_{K, \rm GS}^*)} -e^{-\mathcal{T}(\Theta)}\right \Vert_{\widehat{\mathcal{H}}^{-1} \to \widehat{\mathcal{H}}^{-1} } &\leq \overline{\beta}_K \Vert \mathcal{T}(\Theta) - \mathcal{T}(\Theta_{K, \rm GS}^*)\Vert_{\widehat{\mathcal{H}}^{1}}. \label{eq:main_thm_5}
	\end{align}
 
 
	On the other hand, the linearity and continuity properties of cluster operators given by Theorem~\ref{thm:1} imply that for any $\Theta \in \widehat{\mathcal{H}}^{1, \perp}_{\Psi_{0, K}}$ it holds that
	\begin{align}\label{eq:main_thm_6}
		\Vert \mathcal{T}(\Theta) - \mathcal{T}(\Theta_{K, \rm GS}^*)\Vert_{\widehat{\mathcal{H}}^{-1} \to \widehat{\mathcal{H}}^{-1}} \leq \beta_{\widehat{\mathcal{H}}}\Vert \Theta - \Theta_{K, \rm GS}^*\Vert_{\widehat{\mathcal{H}}^{1}},
	\end{align}
 where, as stated earlier, $\beta_{\widehat{\mathcal{H}}}$ depends only on $N$ and $\Psi_{0, K}$.
 
	In view of Inequalities \eqref{eq:main_thm_2*}-\eqref{eq:main_thm_6}, we deduce that for any $K\geq K_0$ and all $\Theta \in \widehat{\mathcal{H}}^{1, \perp}_{\Psi_{0, K}} \cap \mathbb{B}_{1/\beta_{\widehat{\mathcal{H}}^1}}(\Theta_{K, \rm GS}^*)$ it holds that
	\begin{equation*}
		{\rm (I)}\leq \beta_{\widehat{\mathcal{H}}} \overline{\beta}_K \widehat{\beta}\;  e^{\beta_{\widehat{\mathcal{H}}}\Vert \Theta_{K, \rm GS}^*\Vert_{\widehat{\mathcal{H}}^{1} }}\;\Vert \Psi_{0, K}\Vert_{\widehat{\mathcal{H}}^{1}}\Vert \Phi \Vert_{\widehat{\mathcal{H}}^{1}}\Vert \Theta - \Theta_{K, \rm GS}^*\Vert_{\widehat{\mathcal{H}}^{1}},
	\end{equation*}
    where we emphasise that the constants $\beta_{\widehat{\mathcal{H}}^1}, \overline{\beta}_K, \widehat{\beta}>0$ all depend only on $N$, $\Vert \Psi_{0, K}\Vert_{\widehat{\mathcal{H}}^1}$ and $\Vert \Theta_{K, \rm GS}^*\Vert_{\widehat{\mathcal{H}}^{1} }$. Appealing now to the uniform boundedness of the sequence $\{\Theta^*_{K, \rm GS}\}_{K\geq N}$ as given in Lemma \ref{lem:approx}, we finally deduce the existence of a constant $L_{\rm (I)}>0$ that is independent of $K$ such that for all $K \geq K_0$ and all $\Theta \in \widehat{\mathcal{H}}^{1, \perp}_{\Psi_{0, K}} \cap \mathbb{B}_{1/\beta_{\widehat{\mathcal{H}}^1}}(\Theta_{K, \rm GS}^*)$ it holds that
    \begin{align*}
        {\rm (I)}\leq L_{\rm (I)}\Vert \Theta - \Theta_{K, \rm GS}^*\Vert_{\widehat{\mathcal{H}}^{1}}.
    \end{align*}

	A very similar argument can be used to bound the term (II) appearing in Inequality \eqref{eq:main_thm_1*} so that Inequality \eqref{eq:main_thm_1*} allows us to deduce the required local Lipschitz continuity of the Fr\'echet derivative $\mD \mathcal{f}_K$ at $\Theta_{K, \rm GS}^*$ with Lipschitz constant uniformly bounded in $K$.

	\vspace{3mm}
	
	\noindent \textbf{Claim Three:} The sequence of non-linear maps $\{\mathcal{G}_K\}_{K\geq K_0}$ defined through Equation \eqref{eq:discrete_cc_func} satisfy
	\begin{align*}
		\lim_{K \to \infty} \Vert \mathcal{G}_K(\Theta_{K, \rm GS}^*) \Vert_{(\widehat{\mathcal{H}}^{1, \perp}_{\Psi_{0, K}} )^{*}}=0.
	\end{align*}
	
	To demonstrate this claim, we will use the definition of the non-linear mapping $\mathcal{G}_K\colon \widehat{\mathcal{H}}^{1, \perp}_{\Psi_{0, K}}  \rightarrow (\widehat{\mathcal{H}}^{1, \perp}_{\Psi_{0, K}} )^{*}$ given by Equation \eqref{eq:discrete_cc_func} together with the fact that $\Theta_{K, \rm GS}^* \in \widehat{\mathcal{H}}^{1, \perp}_{\Psi_{0, K}}$ is by definition, a zero of the coupled cluster function $\mathcal{f}_K \colon \widehat{\mathcal{H}}^{1, \perp}_{\Psi_{0, K}}  \rightarrow (\widehat{\mathcal{H}}^{1, \perp}_{\Psi_{0, K}} )^{*}$. Combining these two facts yields that for all $K \geq K_0$ and all $\Upsilon \in \widehat{\mathcal{H}}^{1, \perp}_{\Psi_{0, K}}$ we have that
	\begin{align}\nonumber
		\left\langle \Upsilon, \mathcal{G}_K(\Theta_{K, \rm GS}^*)\right\rangle_{\widehat{\mathcal{H}}^{1, \perp}_{\Psi_{0, K}}  \times \left(\widehat{\mathcal{H}}^{1, \perp}_{\Psi_{0, K}} \right)^*}&= \mathcal{b}_K \left(\Theta_{K, \rm GS}^*, \Upsilon-\Pi_K^{\rm R}\Upsilon\right)\\ \label{eq:main_thm_7}
  &=\mathcal{b}_K \left(\Theta_{K, \rm GS}^* - \Pi_K^{\rm L}\Theta_{K, \rm GS}^*, \Upsilon\right)\\[0.25em] 
		&\leq \Vert \mD \mathcal{f}_K(\Theta_{K, \rm GS}^*)\Vert_{\widehat{\mathcal{H}}^{1, \perp}_{\Psi_{0, K}} \to \left(\widehat{\mathcal{H}}^{1, \perp}_{\Psi_{0, K}}\right)^*}  \Vert \Theta_{K, \rm GS}^*- \Pi_K^{\rm L}\Theta_{K, \rm GS}^*\Vert_{\widehat{\mathcal{H}}^1} \Vert \Upsilon  \Vert_{\widehat{\mathcal{H}}^1}. \nonumber
	\end{align}
	where the second step follows from the definition of the projection operators $\Pi_K^L$ and $\Pi_K^R$ given by Equations \eqref{eq:proof_proj_1}  and \eqref{eq:proof_proj_2} respectively, and the third step follows from the definition of the bilinear form $\mathcal{b}_K\colon \widehat{\mathcal{H}}^{1, \perp}_{\Psi_{0, K}}  \times \widehat{\mathcal{H}}^{1, \perp}_{\Psi_{0, K}}  \rightarrow \mathbb{R}$ given by Equation~\eqref{eq:bilinear}.
	
	We first focus on bounding the term involving the projection operator $\Pi_K^{\rm L}$. To this end, we recall that by assumption, the Fr\'echet derivative $\mD \mathcal{f}_K(\Theta_{K, \rm GS}^*) $-- and hence the bilinear form $\mathcal{b}_K$-- satisfies a discrete inf-sup condition on $\widetilde{\mathcal{V}}_K$. Consequently, recalling further the definition of the projection operator $\Pi_K^{\rm L}$ given by Equation \eqref{eq:proof_proj_1}, we can use the well-known C\'ea's lemma to deduce that for all $K \geq K_0$ and any $\Phi \in \widehat{\mathcal{H}}^{1, \perp}_{\Psi_{0, K}}$ it holds that
	\begin{align}\label{eq:main_thm_8}
		\Vert \Theta_{K, \rm GS}^*- \Pi_K^{\rm L}\Theta_{K, \rm GS}^*\Vert_{\widehat{\mathcal{H}}^1} \leq \left(1 + \frac{\Vert \mD \mathcal{f}_K(\Theta_{K, \rm GS}^*)\Vert_{\widehat{\mathcal{H}}^{1, \perp}_{\Psi_{0, K}}\to \left(\widehat{\mathcal{H}}^{1, \perp}_{\Psi_{0, K}}\right)^* }}{\gamma}\right) \inf_{\widetilde{\Phi}_K \in \widetilde{\mathcal{V}}_K} \Vert \Theta_{K, \rm GS}^* - \widetilde{\Phi}_K\Vert_{\widehat{\mathcal{H}}^1}.
	\end{align}
	
	Consider now Inequalities \eqref{eq:main_thm_7} and \eqref{eq:main_thm_8}. We have already argued (see the proof of \textbf{Claim Two} above) that the operator norm $\Vert \mD \mathcal{f}_K(\Theta_{K, \rm GS}^*)\Vert_{\widehat{\mathcal{H}}^{1, \perp}_{\Psi_{0, K}} \to \left(\widehat{\mathcal{H}}^{1, \perp}_{\Psi_{0, K}}\right)^*} $ is uniformly bounded above in $K$. Recalling therefore Lemma \ref{lem:approx} which implies that the $\widehat{\mathcal{H}}^1$-best approximation error of $\Theta_{K, \rm GS}^*$ in the approximation space $\widetilde{\mathcal{V}}_K$ goes to zero as $K \to \infty$ now yields the required result.
	
	\vspace{3mm}
	
	\noindent \textbf{Claim Four:} For every $K\geq K_0$ it holds that the Fr\'echet derivative $\mD \mathcal{G}_K(\Theta_{K, \rm GS}^*) \colon \widehat{\mathcal{H}}^{1, \perp}_{\Psi_{0, K}}  \rightarrow (\widehat{\mathcal{H}}^{1, \perp}_{\Psi_{0, K}} )^{*}$ is an isomorphism with continuity and inf-sup constants uniformly bounded in $K$.
	
	To demonstrate this final claim, we make use of Equation \eqref{eq:Frechet_G} which implies that for all $K\geq K_0$ it holds that
	\begin{align*}
		\forall \Phi, \Upsilon \in \widehat{\mathcal{H}}^{1, \perp}_{\Psi_{0, K}} \colon \quad \left\langle \Upsilon, \mD \mathcal{G}_K(\Theta_{K, \rm GS}^*)\Phi\right\rangle_{\widehat{\mathcal{H}}^{1, \perp}_{\Psi_{0, K}} \times \left(\widehat{\mathcal{H}}^{1, \perp}_{\Psi_{0, K}}\right)^*}=& \left\langle \Pi_K^{\rm R}\Upsilon, \mD \mathcal{f}_K(\Theta_{K, \rm GS}^*)\Phi\right\rangle_{\widehat{\mathcal{H}}^{1, \perp}_{\Psi_{0, K}} \times \left(\widehat{\mathcal{H}}^{1, \perp}_{\Psi_{0, K}}\right)^*}\\[0.25em]
  +& \mathcal{b}_K \left(\Phi, \Upsilon-\Pi_K^{\rm R}\Upsilon\right)\\[0.5em]
		=& \left\langle \Upsilon, \mD \mathcal{f}_K(\Theta_{K, \rm GS}^*)\Phi\right\rangle_{\widehat{\mathcal{H}}^{1, \perp}_{\Psi_{0, K}} \times \left(\widehat{\mathcal{H}}^{1, \perp}_{\Psi_{0, K}}\right)^*},
	\end{align*}
	where the second step follows from the definition of the bilinear form $\mathcal{b}_K$ given by Equation \eqref{eq:bilinear}. Thus, we have that $\mD \mathcal{G}_K(\Theta_{K, \rm GS}^*)  = \mD \mathcal{f}_k(\Theta_{K, \rm GS}^*) $ and the isomorphism property immediately follows from Theorem \ref{thm:CC_der_inv}.

	To demonstrate the remaining portion of \textbf{Claim Four}, note that we have already shown in the proof of \textbf{Claim Two} that the continuity constant of the mapping $\mD \mathcal{f}_K(\Theta_{K, \rm GS}^*) \colon \widehat{\mathcal{H}}^{1, \perp}_{\Psi_{0, K}}  \rightarrow (\widehat{\mathcal{H}}^{1, \perp}_{\Psi_{0, K}} )^{*}$ is uniformly bounded in $K\geq K_0$. Similar arguments allow us to conclude that the inf-sup constant of $\mD \mathcal{f}_K(\Theta_{K, \rm GS}^*)$ (given by Theorem \ref{thm:CC_der_inv}) is also uniformly bounded in $K\geq K_0$. For the sake of brevity and to avoid repeating tedious arguments, we omit these details.

	\vspace{3mm}
	
	Equipped with \textbf{Claims One} through \textbf{Four}, we are now ready to conclude our proof of the present theorem. We seek to apply the inverse function theorem for Banach spaces, namely, Theorem \ref{Caloz_Rappaz_2} to the non-linear mapping $\mathcal{G}_K \colon \widehat{\mathcal{H}}^{1, \perp}_{\Psi_{0, K}}  \rightarrow (\widehat{\mathcal{H}}^{1, \perp}_{\Psi_{0, K}} )^{*}$ for $K$ sufficiently large. To this end, let us define for each $K\geq K_0$ the quantities
	\begin{align*}
		\epsilon_K :=& \Vert \mathcal{G}_K(\Theta_{K, \rm GS}^*)\Vert_{\left(\widehat{\mathcal{H}}^{1, \perp}_{\Psi_{0, K}} \right)^{*}},\\[1em]
		\gamma_{\mathcal{G}}^*:=& \Vert \mD \mathcal{G}_K(\Theta_{K, \rm GS}^*)^{-1}\Vert_{\left(\widehat{\mathcal{H}}^{1, \perp}_{\Psi_{0, K}} \right)^{*}\to \widehat{\mathcal{H}}^{1, \perp}_{\Psi_{0, K}}},\\[1em]
		\forall\alpha \in [0, \delta_0)\colon \qquad \mL(\alpha):=& \sup_{\Theta \in \overline{\mathbb{B}_{\alpha}(\Theta_{K, \rm GS}^*)}} \Vert \mD \mathcal{G}_K(\Theta_{K, \rm GS}^*)-\mD \mathcal{G}_K(\Theta)\Vert_{\widehat{\mathcal{H}}^{1, \perp}_{\Psi_{0, K}} \to \left(\widehat{\mathcal{H}}^{1, \perp}_{\Psi_{0, K}} \right)^{*} }.
	\end{align*}
	
	We now observe that
	\begin{itemize}
		\item Thanks to \textbf{Claim One}, for all $K \geq {K_0}$, the mapping $\mathcal{G}_K \colon \widehat{\mathcal{H}}^{1, \perp}_{\Psi_{0, K}}  \rightarrow (\widehat{\mathcal{H}}^{1, \perp}_{\Psi_{0, K}} )^{*}$  is of regularity class $\mathscr{C}^1$. 
		
		\item Thanks to \textbf{Claim Two}, for all $K \geq {K_0}$ and $\alpha \in [0, \delta_0)$ with ${\delta_0} >0$ independent of $K$, the constant $\mL(\alpha)$ is bounded above, uniformly  in $K$.
		
		\item Thanks to \textbf{Claim Three}, the constant $\epsilon_K$ defined above satisfies $\lim_{K\to \infty} \epsilon_K =0.$
		
		\item Thanks to \textbf{Claim Four}, for all $K \geq {K_0}$, the Fr\'echet derivative $\mD \mathcal{G}_K \colon \widehat{\mathcal{H}}^{1, \perp}_{\Psi_{0, K}}  \rightarrow (\widehat{\mathcal{H}}^{1, \perp}_{\Psi_{0, K}} )^{*}$ is an isomorphism and the constant $\gamma^*_{\mathcal{G}}$ defined above is uniformly bounded above in $K$.
	\end{itemize}
	
	Consequently, there exists $\overline{K_0} \geq N$ such that for all $K \geq \overline{K_0}$ it holds that
	\begin{align*}
		2 \gamma_{\mathcal{G}}^* \mL(2\gamma_{\mathcal{G}}^* \epsilon_K) < 1,
	\end{align*}
	and the conclusions of Theorem \ref{Caloz_Rappaz_2} apply: for all $K \geq \overline{K_0}$ the closed ball $\overline{\mathbb{B}_{2\gamma_{\mathcal{G}}^*\epsilon_K}\left(\Theta_{K, \rm GS}^*\right)} \subset \widehat{\mathcal{H}}^{1, \perp}_{\Psi_{0, K}} $ contains a unique solution~$\Theta_{K}$ to the equation
	\begin{align*}
		\mathcal{G}_K(\Theta_{K})=0,
	\end{align*}
	the Fr\'echet derivative $\mD \mathcal{G}_K(\Theta_{K}) \colon \widehat{\mathcal{H}}^{1, \perp}_{\Psi_{0, K}} \rightarrow \left(\widehat{\mathcal{H}}^{1, \perp}_{\Psi_{0, K}} \right)^*$ is an isomorphism with
	\begin{align}\label{eq:main_thm_8*}
		\Vert \mD \mathcal{G}_K(\Theta_{K})^{-1}\Vert_{\left(\widehat{\mathcal{H}}^{1, \perp}_{\Psi_{0, K}} \right)^*\to \widehat{\mathcal{H}}^{1, \perp}_{\Psi_{0, K}}  } \leq 2 \gamma_{\mathcal{G}}^*,
	\end{align}
	and for all $\Theta \in \overline{\mathbb{B}_{2\gamma_{\mathcal{G}}^*\epsilon_K}(\Theta^*_{K})}\cap \widehat{\mathcal{H}}^{1, \perp}_{\Psi_{0, K}} $, we have the error estimate
	\begin{align}\label{eq:main_thm_9}
		\Vert \Theta_{K}-\Theta\Vert_{\widehat{\mathcal{H}}^1} \leq 2 \gamma_{\mathcal{G}}^* \Vert \mathcal{G}_K(\Theta)\Vert_{\left(\widehat{\mathcal{H}}^{1, \perp}_{\Psi_{0, K}} \right)^*}.
	\end{align}
	Taking now Inequality \eqref{eq:main_thm_9} and applying the earlier Inequalities \eqref{eq:main_thm_7} and \eqref{eq:main_thm_8} immediately yields the quasi-optimality result \eqref{eq:quasi_optimality}.
	
	It remains to establish the residual-based error estimate \eqref{eq:residual_error}. To this end, let us define for each $K \geq \overline{K_0}$ the quantities
	\begin{align*}
		\widetilde{\epsilon_K} :=& \Vert \mathcal{f}_K(\Theta_K)\Vert_{\left(\widehat{\mathcal{H}}^{1, \perp}_{\Psi_{0, K}} \right)^{*}},\\[1em]
		\gamma_{\mathcal{f}}^*:=& \Vert \mD \mathcal{f}_K(\Theta_{K})^{-1}\Vert_{\left(\widehat{\mathcal{H}}^{1, \perp}_{\Psi_{0, K}} \right)^{*}\to \widehat{\mathcal{H}}^{1, \perp}_{\Psi_{0, K}}},\\[1em]
		\forall\alpha \in [0, \delta_0)\colon \qquad \widetilde{\mL(\alpha)}:=& \sup_{\Theta \in \overline{\mathbb{B}_{\alpha}(\Theta_{K})}} \Vert \mD \mathcal{f}_K(\Theta_{K})-\mD \mathcal{f}_K(\Theta)\Vert_{\widehat{\mathcal{H}}^{1, \perp}_{\Psi_{0, K}} \to \left(\widehat{\mathcal{H}}^{1, \perp}_{\Psi_{0, K}} \right)^{*} }.
	\end{align*} 
	
	To proceed, let us first observe that the quasi-optimality result \eqref{eq:quasi_optimality} in combination with the approximability result given by Lemma \ref{lem:approx} together imply that
	\begin{align}\label{eq:main_thm_10}
		\lim_{K \to \infty} \Vert \Theta_K - \Theta_{K, \rm GS}^*\Vert_{\widehat{\mathcal{H}}^1}=0.
	\end{align}
	
	Next, we claim that we also have the consistency result
	\begin{align}\label{eq:thm_main_consist}
		\lim_{K \to \infty} \widetilde{\epsilon_K} = \lim_{K \to \infty} \Vert \mathcal{f}_K(\Theta_K)\Vert_{\left(\widehat{\mathcal{H}}^{1, \perp}_{\Psi_{0, K}} \right)^{*}}=0.
	\end{align}
	
	Indeed, recalling on the one hand that $\Theta_{K, \rm GS}^* \in \widehat{\mathcal{H}}^{1, \perp}_{\Psi_{0, K}}$ is a zero of the coupled cluster function $\mathcal{f}_K$, and on the other hand that $\mathcal{f}_K$ is a $\mathscr{C}^{\infty}$ mapping, we may appeal to the mean-value inequality for Banach spaces (see, e.g., \cite[Theorem 3.2.7]{MR2323436}) to deduce that for all $K\geq \overline{K_0}$ it holds that
	\begin{align*}
		\Vert \mathcal{f}_K(\Theta_K)-\mathcal{f}_K(\Theta_{K, \rm GS}^*)\Vert_{(\widehat{\mathcal{H}}^{1, \perp}_{\Psi_{0, K}} )^{*}} \leq & \sup_{t \in [0, 1]} \Vert \mD \mathcal{f}_K\left(\Theta_{K, \rm GS}^* + t( \Theta_K-\Theta_{K, \rm GS}^*)\right) \left(\Theta_K - \Theta_{K, \rm GS}^*\right)\Vert_{(\widehat{\mathcal{H}}^{1, \perp}_{\Psi_{0, K}} )^{*}}\\
		\leq & \sup_{t \in [0, 1]} \Vert \mD \mathcal{f}_K\left(\Theta_{K, \rm GS}^* + t( \Theta_K-\Theta_{K, \rm GS}^*)\right) \Vert_{\widehat{\mathcal{H}}^{1, \perp}_{\Psi_{0, K}}\to (\widehat{\mathcal{H}}^{1, \perp}_{\Psi_{0, K}} )^{*}} \\
		\cdot &\Vert \Theta_K - \Theta_{K, \rm GS}^*\Vert_{\widehat{\mathcal{H}}^{1, \perp}_{\Psi_{0, K}}}.
	\end{align*}
 
	Notice now that for any $t\in[0, 1]$ and all $K\geq \overline{K_0}$, the function $\Theta_{K, \rm GS}^* + t( \Theta_K-\Theta_{K, \rm GS}^*) \in \overline{\mathbb{B}_{\alpha}(\Theta_{K, \rm GS}^*)}$ for any $\alpha \in [0, \delta_0)$. We therefore deduce from \textbf{Claim Two} that the Fr\'echet derivative $\mD \mathcal{f}_K$ is locally Lipschitz continuous in a closed ball containing the set $\{\Theta_{K, \rm GS}^* + t( \Theta_K-\Theta_{K, \rm GS}^*) \colon~ t \in [0, 1]\}$ with Lipschitz constant uniformly bounded in $K$. The consistency result \eqref{eq:thm_main_consist} then immediately follows from Equation \eqref{eq:main_thm_10}.

	We now observe that
	\begin{itemize}
		\item For all $K \geq \overline{K_0}$, the mapping $\mathcal{f}_K \colon \widehat{\mathcal{H}}^{1, \perp}_{\Psi_{0, K}}  \rightarrow (\widehat{\mathcal{H}}^{1, \perp}_{\Psi_{0, K}} )^{*}$  is of regularity class $\mathscr{C}^1$ (see, e.g., \cite[Proposition 26]{Hassan_CC}. 
		
		\item Thanks to \textbf{Claim Two} and the $\widehat{\mathcal{H}}^1$-convergence of $\Theta_{K}$ to $\Theta_{K, \rm GS}^*$, there exists $\widetilde{\widetilde{K_0}} \geq \overline{K_0}$ and $\widetilde{\delta_0}>0$ sufficiently small and independent of $K$, such that for all  $\alpha \in [0, \delta_0)$ the constant $\widetilde{\mL(\alpha)}$ is uniformly bounded above in $K$.
		
		\item Thanks to Equation \eqref{eq:thm_main_consist}, the constant $\widetilde{\epsilon_K}$ defined above satisfies $\lim_{K\to \infty} \epsilon_K =0.$
		
		\item Thanks to Inequality \eqref{eq:main_thm_8*}, the constant $\gamma_{\mathcal{f}}^*$ defined above is uniformly bounded above in $K$.
	\end{itemize}
	
	We can therefore conclude that there exists $\widehat{K_0} \geq N$ such that for all $K \geq \widehat{K_0}$ it holds that
	\begin{align*}
		2 \gamma_{\mathcal{f}}^* \mL(2\gamma_{\mathcal{f}}^* \widetilde{\epsilon_K}) < 1.
	\end{align*}
	
	As a consequence, the conclusions of Theorem \ref{Caloz_Rappaz_2} are once again applicable. In particular, for all $K \geq \widehat{K_0}$ the closed ball $\overline{\mathbb{B}_{2\gamma_{\mathcal{f}}^*\widetilde{\epsilon_K}}\left(\Theta_{K}\right)} \subset \widehat{\mathcal{H}}^{1, \perp}_{\Psi_{0, K}} $ contains a unique solution~$\widetilde{\Theta_{K}}$ to the equation
	\begin{align*}
		\mathcal{f}_K(\widetilde{\Theta_{K}})=0,
	\end{align*}
	and we have the error estimate
	\begin{align*}
		\Vert \widetilde{\Theta_{K}}-\Theta_K\Vert_{\widehat{\mathcal{H}}^1} \leq 2 \gamma_{\mathcal{f}}^* \Vert \mathcal{f}_K(\Theta_K)\Vert_{(\widehat{\mathcal{H}}^{1, \perp}_{\Psi_{0, K}} )^*}.
	\end{align*}
	Recalling that $\Theta_{K, \rm GS}^*$ is the locally unique zero of the coupled cluster function $\mathcal{f}_K$ (see Theorem \ref{thm:CC_der_inv} and Lemma \ref{lem:approx}), we deduce that in fact $\widetilde{\Theta_K}= \Theta_{K, \rm GS}^*$. The required residual-based error estimate thus follows.

\end{proof}


 
	\section*{Acknowledgements} 
	
	The authors thank Yipeng Wang for useful discussions and for help with the exposition on the Hartree-Fock method in the appendix. The first author also warmly thanks Antoine Levitt for several helpful discussions on the essential spectra of $N$-body Schrödinger operators. This project has received funding from the European Research Council (ERC) under the European Union’s Horizon 2020 research and innovation program (Grant Agreement No. 810367). The first author also acknowledges support from the Deutsche Forschungsgemeinschaft (DFG) through the Emmy Noether Programme (Project No. 555300205). The majority of this work was done while the first author was employed as a post-doctoral researcher at the Laboratoire Jacques-Louis Lions (LJLL) at Sorbonne Université. 

    \bibliographystyle{abbrv.bst}
	\bibliography{refs.bib}

 \newpage
 
\appendix

\section{The Hartree-Fock Method}\label{sec:HF}~
	
	The Hartree-Fock method is a basic numerical algorithm for approximating the ground state of the electronic Hamiltonian $H \colon \widehat{\mathcal{H}}^1 \rightarrow \widehat{\mathcal{H}}^{-1}$ defined through Equation \eqref{eq:Hamiltonian}. The essential idea of the Hartree-Fock approximation is to replace the infinite-dimensional tensor space $\widehat{\mathcal{H}}^1$ of antisymmetric functions-- over which the minimum expectation value of the electronic Hamiltonian is computed-- with a more computationally tractable subset, in this case the set of \emph{single Slater determinants}. As in the main text of this article, we assume throughout this section that the total nuclear charge $Z= \sum_{\alpha=1}^M Z_{\alpha} > N-1$.

	\begin{definition}[Hartree-Fock Minimisation Set]\label{def:HF_space}~
		
		We define the Hartree-Fock minimisation set ${\mS}_{\rm HF}$ as the subset of $\widehat{\mathcal{H}}^1$ given by
		\begin{align*}
			{\mS}_{\rm HF}:= \Big\{\Phi \in  \widehat{\mathcal{H}}^1 \colon \text{ \rm there exist $\mL^2$-orthonormal } \{\phi_i\}_{i=1}^N \subset \mH^1(\R^3)  \\
   \text{ \rm with }\Phi(\bold{x}_1, \ldots, \bold{x}_N)=\frac{1}{\sqrt{N!}} \text{\rm det}\; \big(\phi_{i}(\bold{x}_j)\big)_{i, j=1}^N \Big\}.
		\end{align*}
	\end{definition}

	\vspace{3mm}
	
	\textbf{Hartree-Fock Approximation of Minimisation Problem \eqref{eq:Ground_State}}~
	
	Let the Hartree-Fock minimisation set ${\mS}_{\rm HF}$ be defined through Definition \ref{def:HF_space}. We seek the solution(s) $\mathcal{E}^*_{\rm HF}\in \mathbb{R}$ to the following minimisation problem
		\begin{align}
			\label{eq:HF}
			\mathcal{E}_{\rm HF}^*:= \min_{\Psi\in \mS_{\rm HF}} \langle\Psi, {H} \Psi \rangle_{\widehat{\mathcal{H}}^1 \times \widehat{\mathcal{H}}^{-1}}.
		\end{align}

 \vspace{0mm}
	Several remarks are in order.

 \vspace{3mm}
	First, in Equation \eqref{eq:HF}, the normalisation constraints have been removed since each element of the set ${\mS}_{\rm HF}$ is normalised by construction. Second, any function $\Psi^*_{\rm HF} \in \mS_{\rm HF}$ that achieves the minimum in Equation \eqref{eq:HF} is called a Hartree-Fock ground state and obviously satisfies
	\begin{align}
	\forall {\Psi} \in \mS_{\rm HF}\colon \quad	\langle {\Psi}, {H} \Psi^*_{\rm HF} \rangle_{\widehat{\mathcal{H}}^1 \times \widehat{\mathcal{H}}^{-1}}= \mathcal{E}_{\rm HF}^*\langle{\Psi}, \Psi^*_{\rm HF} \rangle_{\widehat{\mathcal{H}}^1 \times \widehat{\mathcal{H}}^{-1}}.
	\end{align}
	
	Third, concerning the well-posedness of the minimisation problem \eqref{eq:HF}, the existence of an infimum of the Hartree-Fock energy $\mathcal{E}^*_{\rm HF}$ is relatively straightforward to establish but the existence of a minimising Slater determinant is a difficult problem that was finally solved by Lieb and Simon \cite{Simon}. This result was generalised in a celebrated paper by P.L. Lions \cite{MR879032}. The \emph{uniqueness} of such minimisers is still an open question however. 
	
	Fourth, a useful simplification occurs in the expression of the Hamiltonian ${H}$ when it acts on an element of ${\mS}_{\rm HF}$ and is projected against another Slater determinant. Indeed, given $\Phi:= \frac{1}{\sqrt{N!}} \text{det}\; \big(\phi_{i}(\bold{x}_j)\big)_{i, j=1}^N \in \mS_{\rm HF}$ we have
	\begin{equation}\label{eq:Fock_energy}
		\begin{split}
			\langle\Phi, {H} \Phi \rangle_{\widehat{\mathcal{H}}^1 \times \widehat{\mathcal{H}}^{-1}} = -&\frac{1}{2} \sum_{j=1}^N \int_{\mathbb{R}^{3}}\vert \nabla\phi_j(\bold{x})\vert^2 \; d\bold{x} + \sum_{j =1}^{N} \sum_{\alpha =1}^{M} \int_{\mathbb{R}^{3}}\frac{-Z_{\alpha}}{\vert \bold{x}_{\alpha}- \bold{x}\vert }\vert \phi_j(\bold{x})\vert^2\; d\bold{x}\\
			+& \frac{1}{2}\sum_{j =1}^{N}\sum_{i =1}^{N} \int_{\mathbb{R}^{3}}\int_{\mathbb{R}^{3}}\frac{\vert\phi_j(\bold{x})\vert^2\vert\phi_i(\bold{y})\vert^2}{\vert \bold{x}- \bold{y}\vert } \; d\bold{x}d\bold{y}\\
   -&\frac{1}{2}\sum_{j =1}^{N}\sum_{i =1}^{N} \int_{\mathbb{R}^{3}}\int_{\mathbb{R}^{3}}\frac{\phi_i(\bold{x})\phi_j(\bold{x})\phi_i(\bold{y})\phi_j(\bold{y})}{\vert \bold{x}- \bold{y}\vert } \; d\bold{x}d\bold{y}.
		\end{split}
	\end{equation}
	
	Equation \eqref{eq:Fock_energy} suggests that we can introduce, in a natural manner, the so-called single-particle Hartree-Fock operator ${\rm F}_{\Phi} \colon \mH^1(\R^3) \rightarrow \mH^{-1}(\R^3)$.

	\begin{definition}[Single-Particle Hartree-Fock Operator ]\label{def:Fock_single}~
		
		For any $\Phi=\frac{1}{\sqrt{N!}} \text{\rm det}\; \big(\phi_{i}(\bold{x}_j)\big)_{i, j=1}^N  \in~ {\mS}_{\rm HF}\subset \widehat{\mathcal{H}}^1$, we define the single-particle Hartree-Fock (or mean-field) operator ${\rm F}_{\Phi} \colon \mH^1(\R^3)\rightarrow \mH^{-1}(\R^3)$ as the mapping with the property that 
		\begin{align*}
			\forall \psi \in \mH^1(\R^3) \colon \quad {\rm F}_{\Phi} \psi (\bold{x}) := &-\frac{1}{2} \Delta \psi(\bold{x})+ \sum_{\alpha =1}^{M} \frac{-Z_{\alpha}}{\vert \bold{x}_{\alpha}- \bold{x} \vert}\psi(\bold{x}) +  \sum_{j=1}^N\int_{\mathbb{R}^{3}}\frac{\vert\phi_j(\bold{y})\vert^2}{\vert \bold{x}- \bold{y}\vert }\; d\bold{y} \psi(\bold{x})\\
			&-\sum_{j =1}^{N} \int_{\mathbb{R}^{3}}\frac{\phi_j(\bold{x})\phi_j(\bold{y})}{\vert \bold{x}- \bold{y}\vert} \psi(\bold{y})\; d\bold{y} \quad \text{\rm for a.e. } \bold{x} \in \mathbb{R}^3.
		\end{align*}
	\end{definition}

    It is readily verified from Kato-Rellich theory (see, e.g., \cite[Chapter 13]{MR1361167}) that the single-particle Hartree-Fock operator (constructed from any determinant $\Phi \in \mS_{\rm HF}$) is self-adjoint on $\mL^2(\mathbb{R}^3)$ with domain $\mH^2(\R^3)$ and form domain $\mH^1(\R^3)$. Additionally, since the minimum of the Hartree-Fock energy exists, the single-particle Fock operator constructed from the Hartree-Fock determinant $\Psi_{\rm HF} \in \mS_{\rm HF}$ and modified with a suitable shift is a coercive operator\footnote{Similar to the electronic Hamiltonian ${H}$, the single-particle operator ${\rm F}_{\Phi}, \Phi \in \mS_{\rm HF}$ can be shown to satisfy an ellipticity estimate. This is simply a consequence of the Hardy inequality.}.
 
	The Euler-Lagrange equations for the Hartree-Fock minimisation problem \eqref{eq:Fock_energy}, i.e., the first-order optimality conditions then read as follows:

 \vspace{2mm}
	\textbf{Infinite-Dimensional Hartree-Fock Eigenvalue Problem}~
	
	For any determinant $\Phi \in \mS_{\rm HF}$, let the single-particle Hartree-Fock operator ${\rm F}_{\Phi} \colon \mH^1(\R^3) \rightarrow \mH^{-1}(\R^3)$ be defined through Definition \ref{def:Fock_single}. We seek $\Psi=\frac{1}{\sqrt{N!}} \text{\rm det}\; \big(\psi_{i}(\bold{x}_j)\big)_{i, j=1}^N  \in~{\mS}_{\rm HF}\subset \widehat{\mathcal{H}}^1$ and $\{\lambda_i\}_{i=1}^N \subset \mathbb{R}$ with the property that for all $i=1, \ldots, N$ it holds that
	\begin{equation}\label{eq:HF_eigenvalue}
		{\rm F}_{\Psi} \psi_{i} = \lambda_i \psi_i.
	\end{equation}
	
	Equation \eqref{eq:HF_eigenvalue} is a non-linear eigenvalue problem, the Hartree-Fock operator being dependent on the sought-after eigenfunctions $\{\psi_i\}_{i=1}^N$. Partial results about the analysis of this eigenvalue problem are available, and in particular, it is known that any minimiser of the Hartree-Fock minimisation problem \eqref{eq:Fock_energy} satisfies Equation \eqref{eq:HF_eigenvalue} with all~$\lambda_i < 0$ (see, e.g., \cite{MR879032}).

\vspace{2mm}
 \begin{Convention}[Ordering of Eigenvalues]~
 
Throughout the remainder of this article, we adopt the convention that the eigenvalues of any self-adjoint operator-- including, in particular, any single-particle Hartree-Fock operator ${\rm F}_{\Psi}, ~\Psi \in \mS_{\rm HF}$ or discretisations thereof-- are ordered \underline{ascendingly}. Thus, in the case of the eigenvalue problem \eqref{eq:HF_eigenvalue}, we have $\lambda_1 \leq \lambda_2 \leq \lambda_3 \leq\ldots$.
 \end{Convention}

	Before proceeding further, let us point out that the single-particle Hartree-Fock operator ${\rm F}_{\Phi} \colon \mH^1(\R^3) $ $\rightarrow \mH^{-1}(\R^3)$ can be used to build an $N$-particle Hartree-Fock operator that acts on the tensor space~$\widehat{\mathcal{H}}^1$. Indeed, we have the following definition. \newpage
	
	\begin{definition}[N-Particle Hartree-Fock Operator ]\label{def:Fock_N}~		
		
		For any $\Phi=\frac{1}{\sqrt{N!}} \text{\rm det}\; \big(\phi_{i}(\bold{x}_j)\big)_{i, j=1}^N  \in~ {\mS}_{\rm HF}\subset \widehat{\mathcal{H}}^1$, we define the $N$-particle Fock operator ${\mathcal{F}}_{\Phi} \colon \widehat{\mathcal{H}}^1\rightarrow \widehat{\mathcal{H}}^{-1}$ as the mapping with the property that for any $\Psi=\frac{1}{\sqrt{N!}} \text{\rm det} \big(\psi_{k_i}(\bold{x}_j)\big)_{i, j=1}^N   \in {\mS}_{\rm HF}$ it holds that
		\begin{align*}
			\big(\mathcal{F}_{\Phi}\Psi \big)(\bold{x}_1, \ldots, \bold{x}_N)  =\frac{1}{\sqrt{N!}}\sum_{j=1}^N \sum_{\pi \in {\mS}(N)} (-1)^{\rm sgn(\pi)}  \psi_{\pi (k_1)}(\bold{x}_1)\otimes \psi_{\pi(k_2)}(\bold{x}_2)\otimes \ldots \\
   \ldots \otimes{\rm F}_{\Phi} \psi_{\pi(k_j)}(\bold{x}_j)\otimes \ldots \otimes \psi_{\pi(k_N)}(\bold{x}_N),
		\end{align*}
		and whose action on arbitrary elements of $\widehat{\mathcal{H}}^1 $ is defined through linearity using the fact that any $\widehat{\mathcal{L}}^2$-orthonormal Slater basis $\mathcal{B}_{\wedge}$ of $\widehat{\mathcal{H}}^1$ is a subset of ${\mS}_{\rm HF}$.
	\end{definition}
	
	As expected, the $N$-particle Fock operator inherits many properties from the single-particle Fock operator. In particular, for any determinant $\Phi \in \mS_{\rm HF}$, the operator $\mathcal{F}_{\Phi}$ is self-adjoint on $\widehat{\mathcal{L}}^2$ with domain $\widehat{\mathcal{H}}^2$ and form domain $\widehat{\mathcal{H}}^1$. Moreover, the $N$-particle Fock operator constructed from the Hartree-Fock determinant $\Psi_{\rm HF} \in \mS_{\rm HF}$ and modified with a suitable shift is also a coercive operator.

	Consider once again the non-linear eigenvalue problem \eqref{eq:HF_eigenvalue}. The practical resolution of this problem requires the introduction of a  finite-dimensional subspace of $\mH^1(\R^3)$. Indeed, denoting by $\mX_K $ any $K$-dimensional subspace of $\mH^1(\R^3)$, we can define:
	
	\begin{definition}[Hartree-Fock Finite-Dimensional Minimisation Set]\label{def:HF_space_finite}~
		
		For any natural number $K \geq N$, we define the finite-dimensional set ${\mS}_{\rm HF}^K$ as the subset of $\widehat{\mathcal{H}}^1$ given by
		\begin{align*}
			{\mS}^K_{\rm HF}:= \Big\{\Phi \in  \widehat{\mathcal{H}}^1 \colon \text{ \rm there exist $\mL^2$-orthonormal } \{\phi_i\}_{i=1}^N \subset \mX_K\\  \text{ \rm with }\Phi(\bold{x}_1, \ldots, \bold{x}_N)=\frac{1}{\sqrt{N!}} \text{\rm det}\; \big(\phi_{i}(\bold{x}_j)\big)_{i, j=1}^N \Big\}.
		\end{align*}
	\end{definition}
	
	A finite-dimensional approximation of the non-linear eigenvalue problem \eqref{eq:HF_eigenvalue} is then given as follows.

	\textbf{Finite-Dimensional Hartree-Fock Eigenvalue Problem}~	
	
	For any determinant $\Phi \in \mS_{\rm HF}$, let the single-particle Fock operator ${\rm F}_{\Phi} \colon \mH^1(\R^3) \rightarrow \mH^{-1}(\R^3)$ be defined through Definition \ref{def:Fock_single}. We seek $\Psi=\frac{1}{\sqrt{N!}} \text{\rm det}\; \big(\psi_{i}(\bold{x}_j)\big)_{i, j=1}^N  \in~{\mS}^K_{\rm HF}\subset \widehat{\mathcal{H}}^1$ and $\{\lambda_i\}_{i=1}^N \subset \mathbb{R}$ with the property that for all $\phi \in \mX_K \subset  \mH^1(\R^3)$ and all $i=1, \ldots, N$ it holds that
	\begin{equation}\label{eq:HF_eigenvalue_approx}
		\left\langle {\rm F}_{\Psi} \psi_{i}, \phi\right\rangle_{\mH^{-1}(\R^3) \times \mH^{1}(\R^3)} = \lambda_i \left(\psi_i, \phi\right)_{\mL^2(\R^3)}.
	\end{equation}
	
	Let us remark here that since the single-particle Fock operator is self-adjoint on 	${\rm X}_K \subset \mH^1(\R^3) $, it will in fact have $K$ orthonormal eigenfunctions $\{\Psi_{i}\}_{i=1}^K \subset {\rm X}_K $. By convention, the $N$ solution eigenfunctions of Equation~\eqref{eq:HF_eigenvalue_approx} are chosen to be the ones with the lowest associated eigenvalues (the so-called Aufbau principle). Thus, the diagonalisation of the single-particle Fock-operator provides an $\mL^2$-orthonormal basis for ${\rm X}_K$, which can be denoted as $\{\psi_i\}_{i=1}^N \cup \{\psi_i\}_{i=N+1}^K$. A standard approach in electronic structure calculation is to take $\mathscr{R}=\text{span}\{\psi_i\}_{i=1}^N$ as an occupied space (recall Definition \ref{def:occ_vir}). We will adopt precisely this view in the forthcoming Appendix \ref{sec:spaces}.

Next, let us point out a fact is often exploited in computational quantum chemistry (see, e.g., Møller Plesset perturbation theory \cite[Chapter 14]{helgaker2014molecular}), namely that the $N$-particle Fock operator defined through Definition \ref{def:Fock_N} induces a decomposition of the electronic Hamiltonian ${H} \colon \widehat{\mathcal{H}}^1 \rightarrow \widehat{\mathcal{H}}^{-1}$. Indeed, we have the following lemma.

\begin{lemma}[Electronic Hamiltonian Decomposition and the Fluctuation Potential]\label{lem:Fluctuation}~

    For any $\Phi=\frac{1}{\sqrt{N!}} \text{\rm det}\; \big(\phi_{i}(\bold{x}_j)\big)_{i, j=1}^N  \in~ {\mS}_{\rm HF}\subset \widehat{\mathcal{H}}^1$, let the $N$-particle Fock operator ${\mathcal{F}}_{\Phi} \colon \widehat{\mathcal{H}}^1\rightarrow \widehat{\mathcal{H}}^{-1}$ be defined as in Definition \ref{def:Fock_N} and let the electronic Hamiltonian ${H} \colon \widehat{\mathcal{H}}^1 \rightarrow \widehat{\mathcal{H}}^{-1}$ be defined through Equation \eqref{eq:Hamiltonian}. Then we have the decomposition
    \begin{align*}
    {H} = {\mathcal{F}}_{\Phi}  + \mathcal{U}_{\Phi},
    \end{align*}
where $\mathcal{U}_{\Phi}:= \mathcal{H}- {\mathcal{F}}_{\Phi}$ is known as the {Fluctuation potential}, and is a bounded linear mapping from $\widehat{\mathcal{H}}^1$ to $\mathcal{\widehat{L}}^2$ with a continuity constant that depends only on $\Phi$ and $N$.
\end{lemma}
\begin{proof}
    The fact that $\mathcal{U}_{\Phi}:= \mathcal{H}- {\mathcal{F}}_{\Phi}$ is a bounded linear mapping from $\widehat{\mathcal{H}}^1$ to $\mathcal{\widehat{L}}^2$ is essentially a consequence of the following Hardy-type inequalities whose proof can be found in \cite[Lemma 1]{HY}: For all $u \in \mathscr{C}^{\infty}_{\rm comp}(\mathbb{R}^3)$ and $v \in \mathscr{C}^{\infty}_{\rm comp}(\mathbb{R}^6)$ it holds that
    \begin{align}\label{eq:Hardy_1}
        \int_{\mathbb{R}^3} \frac{u^2(\bold{x})}{\vert \bold{x}\vert^2} \; d\bold{x} &\leq 4 \int_{\mathbb{R}^3}\vert \nabla u(\bold{x})\vert^2\; d\bold{x},\\[1em] \label{eq:Hardy_2}
        \int_{\mathbb{R}^6} \frac{v^2(\bold{x}, \bold{y})}{\vert \bold{x}-\bold{y}\vert^2} \; d\bold{x}d\bold{y} &\leq 2 \int_{\mathbb{R}^6}\vert \nabla_{\bold{x}} v(\bold{x}, \bold{y})\vert^2 + \vert \nabla_{\bold{y}} v(\bold{x}, \bold{y})\vert^2\; d\bold{x}d\bold{y}.
    \end{align}
    
    Equipped with the Hardy-type inequalities \eqref{eq:Hardy_1} and \eqref{eq:Hardy_2}, we will now show that there exists a constant $C_{\Phi, N}>0$ such that for all $\Psi_H \in \widehat{\mathcal{H}}^1$ and $\Psi_L \in \widehat{\mathcal{L}}^2$ it holds that
\begin{equation}\label{eq:fluc_L2}
    \left(\mathcal{U}_{\Phi}\Psi_H, \Psi_L\right)_{\widehat{\mathcal{L}}^2} \leq C_{\Phi,N} \;\Vert\Psi_H \Vert_{\widehat{\mathcal{H}}^1}\;\Vert\Psi_L \Vert_{\widehat{\mathcal{L}}^2}
\end{equation}

To this end, let us recall Equation \eqref{eq:Hamiltonian} and Definitions~\ref{def:Fock_single}~and~\ref{def:Fock_N} to deduce that for any $\Psi_H\in\widehat{\mathcal{H}}^1$, we have that
    \begin{align}\nonumber
        \mathcal{U}_{\Phi}\Psi_H =& \mathcal{H}\Psi_H - {\mathcal{F}}_{\Phi}\Psi_H\\ \label{eq:new_fluc}
        =& \underbrace{\sum_{i =1}^{N} \sum_{j =1}^{i-1}  \frac{1}{\vert \bold{x}_i - \bold{x}_j\vert}\Psi_H(\bold{x}_1, \ldots, \bold{x}_N)}_{:=\mathcal{U}_1\Psi_H}- \underbrace{  \sum_{i=1}^N\sum_{j=1}^N\int_{\mathbb{R}^{3}}\frac{\vert\phi_j(\bold{y})\vert^2}{\vert \bold{x}_i- \bold{y}\vert }\; d\bold{y} \Psi_H(\bold{x}_1, \ldots, \bold{x}_N)}_{:=\mathcal{U}_2\Psi_H}\\
		+&\underbrace{\sum_{i=1}^N\sum_{j =1}^{N} \int_{\mathbb{R}^{3}}\frac{\phi_j(\bold{x}_i)\phi_j(\bold{y})}{\vert \bold{x}_i- \bold{y}\vert} \Psi_H(\bold{x}_1,\cdots,\bold{x}_{i-1},\bold{y},\bold{x}_{i+1},\cdots,\bold{x}_N)\; d\bold{y}}_{:=\mathcal{U}_3\Psi_H} \quad \text{for a.e. } (\bold{x}_1, \ldots, \bold{x}_N) \in \mathbb{R}^{3N}.  \nonumber
    \end{align}

Thus, the fluctuation potential $\mathcal{U}_{\Phi}$ can be written as the sum of three terms, which we denote $\mathcal{U}_{\ell}, ~ \ell\in \{1, 2, 3\}$. We will now prove that for each $\ell \in \{1, 2, 3\}$, there exists a constant $K^{\ell}_{N, \Phi}>0$ such that for all $\Psi_H \in \widehat{\mathcal{H}}^1$ and~$\Psi_L \in \widehat{\mathcal{L}}^2$ it holds that
\begin{equation}\label{eq:fluc_parts_L2}
    \left(\mathcal{U_{\ell}}\Psi_H, \Psi_L\right)_{\widehat{\mathcal{L}}^2} \leq K^{\ell}_{N, \Phi}\;\Vert\Psi_H\Vert_{\widehat{\mathcal{H}}^1}\;\Vert\Psi_L\Vert_{\widehat{\mathcal{L}}^2}.
\end{equation}

    Let us begin by considering the operator $\mathcal{U}_1$. Observe that for any $1\leq j< i\leq N$, all $\Psi_H \in \widehat{\mathcal{H}}^1$ and $\Psi_L \in \widehat{\mathcal{L}}^2$, and a.e. $(\bold{x}_1, \ldots, \bold{x}_N) \in \mathbb{R}^{3N}$ it holds that
    \begin{align*}
      &\int_{\mathbb{R}^6}\frac{\Psi_H(\bold{x}_1,\cdots,\bold{x}_N)}{\vert \bold{x}_i - \bold{x}_j\vert}\Psi_L(\bold{x}_1,\cdots,\bold{x}_N)\;d\bold{x}_j d\bold{x}_i\\[0.5em]
        \leq& \Big(\int_{\mathbb{R}^6}\frac{\vert\Psi_H(\bold{x}_1,\cdots,\bold{x}_N)\vert^2}{\vert \bold{x}_i - \bold{x}_j\vert^2}\;d\bold{x}_j d\bold{x}_i\Big)^{\frac{1}{2}}\Big(\int_{\mathbb{R}^6}\vert\Psi_L(\bold{x}_1,\cdots,\bold{x}_N)\vert^2\;d\bold{x}_j d\bold{x}_i\Big)^{\frac{1}{2}}\\[0.5em]
\leq&\sqrt{2}\Big(\int_{\mathbb{R}^6}\vert\nabla_{\bold{x}_j}\Psi_H(\bold{x}_1,\cdots,\bold{x}_N)\vert^2+\vert\nabla_{\bold{x}_i}\Psi_H(\bold{x}_1,\cdots,\bold{x}_N)\vert^2\;d\bold{x}_j\bold{x}_i\Big)^{\frac{1}{2}}\Big(\int_{\mathbb{R}^6}\vert\Psi_L(\bold{x}_1,\cdots,\bold{x}_N)\vert^2\;d\bold{x}_j\bold{x}_i\Big)^{\frac{1}{2}},
    \end{align*}
    where the last step follows from the Hardy Inequality \eqref{eq:Hardy_2}. Using now Fubini's theorem and the Cauchy-Schwarz inequality together with the shorthand notation \[d\bold{x}^{ij}_r:=d\bold{x}_1\cdots d\bold{x}_{j-1} d\bold{x}_{j+1}\cdots d\bold{x}_{i-1}d\bold{x}_{i+1}\cdots d\bold{x}_N,\] we can further deduce that for any $1\leq j< i\leq N$, all $\Psi_H \in \widehat{\mathcal{H}}^1$ and $\Psi_L \in \widehat{\mathcal{L}}^2$ it holds that
    \begin{align*}
      &\int_{\mathbb{R}^{3N-6}}\int_{\mathbb{R}^6}\frac{\Psi_H(\bold{x}_1,\cdots,\bold{x}_N)}{\vert \bold{x}_i - \bold{x}_j\vert}\Psi_L(\bold{x}_1,\cdots,\bold{x}_N)\;d\bold{x}_j d\bold{x}_i d\bold{x}^{ij}_r\\
        \leq&\sqrt{2}\int_{\mathbb{R}^{3N-6}}\left(\int_{\mathbb{R}^6}\vert\nabla_{\bold{x}_j}\Psi_H(\bold{x}_1,\cdots,\bold{x}_N)\vert^2+\vert\nabla_{\bold{x}_i}\Psi_H(\bold{x}_1,\cdots,\bold{x}_N)\vert^2\;d\bold{x}_j d\bold{x}_i\right)^{\frac{1}{2}} \\
        \hphantom{a} &\hspace{0.5mm}\hphantom{\sqrt{2}sahsl}\cdot \left(\int_{\mathbb{R}^6}\vert\Psi_L(\bold{x}_1,\cdots,\bold{x}_N)\vert^2\;d\bold{x}_j d\bold{x}_i\right)^{\frac{1}{2}}d\bold{x}^{ij}_r\\
 \leq&\sqrt{2}\left(\Vert\nabla_{\bold{x}_j}\Psi_H\Vert^2_{\widehat{\mathcal{L}}^2}+\Vert\nabla_{\bold{x}_i}\Psi_H\Vert^2_{\widehat{\mathcal{L}}^2}\right)^{\frac{1}{2}}\Vert\Psi_L\Vert_{\widehat{\mathcal{L}}^2},
    \end{align*}
    with the last step following from a Cauchy-Schwarz inequality applied to the outer integral.
    
Recalling now the definition of operator $\mathcal{U}_1$, we can conclude from a basic combinatorial argument that for all $\Psi_H \in \widehat{\mathcal{H}}^1$ and $\Psi_L \in \widehat{\mathcal{L}}^2$ it holds that
    \begin{equation}\label{eq:U_1}
        \begin{aligned}
\left(\mathcal{U_{1}}\Psi_H, \Psi_L\right)_{\widehat{\mathcal{L}}^2}\leq&\sqrt{2}\Vert\Psi_L\Vert_{\widehat{\mathcal{L}}^2}\sum_{i=1}^N\sum_{j=1}^{i-1}\left(\Vert\nabla_{\bold{x}_j}\Psi_H\Vert^2_{\widehat{\mathcal{L}}^2}+\Vert\nabla_{\bold{x}_i}\Psi_H\Vert^2_{\widehat{\mathcal{L}}^2}\right)^{\frac{1}{2}}\\[0.25em]
            \leq&\sqrt{N}(N-1)\Vert\nabla\Psi_H\Vert_{\widehat{\mathcal{L}}^2}\;\Vert\Psi_L\Vert_{\widehat{\mathcal{L}}^2},
        \end{aligned}
        \end{equation}
which shows that the bound \eqref{eq:fluc_parts_L2} indeed holds for the operator $\mathcal{U}_1$.

Next, we focus on the term $\mathcal{U}_2$. Once again, we observe that for any $i\leq j \leq N$ and a.e. $\bold{x}_i\in \mathbb{R}^3$ it holds that
    \begin{align*}
        \int_{\mathbb{R}^{3}}\frac{\vert\phi_j(\bold{y})\vert^2}{\vert \bold{x}_i- \bold{y}\vert }\; d\bold{y}\leq \left( \int_{\mathbb{R}^{3}}\vert\phi_j(\bold{y})\vert^2\; d\bold{y}\right)^{\frac{1}{2}}\left( \int_{\mathbb{R}^{3}}\frac{\vert\phi_j(\bold{y})\vert^2}{\vert \bold{x}_i- \bold{y}\vert^2 }\; d\bold{y}\right)^{\frac{1}{2}} &\leq\left( \int_{\mathbb{R}^{3}}\frac{\vert\phi_j(\bold{y})\vert^2}{\vert \bold{x}_i- \bold{y}\vert^2 }\; d\bold{y}\right)^{\frac{1}{2}} \\
        &\leq 2\left( \int_{\mathbb{R}^{3}}\vert\nabla\phi_j(\bold{y})\vert^2\; d\bold{y}\right)^{\frac{1}{2}},
    \end{align*}
    where the second inequality results from the fact that $\Vert \phi_j \Vert_{\mL^2}=1$ for all $j \in \{1, \ldots, N\}$ by assumption, and the third inequality follows from the Hardy inequality \eqref{eq:Hardy_1}. We therefore deduce that
    \begin{align*}
        \sum_{i=1}^N\sum_{j=1}^N\int_{\mathbb{R}^{3}}\frac{\vert\phi_j(\bold{y})\vert^2}{\vert \bold{x}_i- \bold{y}\vert }\; d\bold{y}\leq 2N\;\sum_{j=1}^N\Vert\nabla \phi_j\Vert_{\mL^2},
    \end{align*}
    and consequently, for all $\Psi_H \in \widehat{\mathcal{H}}^1$ and~$\Psi_L \in \widehat{\mathcal{L}}^2$ it holds that
    \begin{equation}\label{eq:U_2}
        \begin{aligned}
         \left(\mathcal{U_{2}}\Psi_H, \Psi_L\right)_{\widehat{\mathcal{L}}^2}\leq &2N\; \sum_{j=1}^N\Vert\nabla\phi_j\Vert_{\mL^2}\int_{\mathbb{R}^{3N}}\big\vert\Psi_H(\bold{x}_1,\cdots,\bold{x}_N)\Psi_L(\bold{x}_1,\cdots,\bold{x}_N)\big\vert\;d\bold{x}_1,\cdots,d\bold{x}_N \\
\leq&2N\;\sum_{j=1}^N\Vert\nabla\phi_j\Vert_{\mL^2}\Vert\Psi_H\Vert_{\widehat{\mathcal{L}}^2}\Vert\Psi_L\Vert_{\widehat{\mathcal{L}}^2}.
        \end{aligned}
    \end{equation}
    We have thus shown that the bound \eqref{eq:fluc_parts_L2} also holds for the operator $\mathcal{U}_2$.
   
   It remains to consider the final term $\mathcal{U}_3$ that appears in the decomposition of the fluctuation potential. To this end, we once again observe that for any $1\leq i,j\leq N$, all $\Psi_H \in \widehat{\mathcal{H}}^1$, and a.e. $(\bold{x}_1, \ldots, \bold{x}_N) \in \mathbb{R}^{3N}$ it holds that
    \begin{align*}
        &\int_{\mathbb{R}^{3}}\frac{\phi_j(\bold{x}_i)\phi_j(\bold{y})}{\vert \bold{x}_i- \bold{y}\vert} \Psi_H(\bold{x}_1,\cdots,\bold{x}_{i-1},\bold{y},\bold{x}_{i+1},\cdots,\bold{x}_N)\; d\bold{y}\\[0.5em]
        \leq&\vert\phi_j(\bold{x}_i)\vert\;\left(\int_{\mathbb{R}^{3}}\frac{\vert\phi_j(\bold{y})\vert^2}{\vert \bold{x}_i- \bold{y}\vert^2}\;d\bold{y}\right)^{\frac{1}{2}}\;\left(\int_{\mathbb{R}^{3}}\vert\Psi_H(\bold{x}_1,\cdots,\bold{x}_{i-1},\bold{y},\bold{x}_{i+1},\cdots,\bold{x}_N)\vert^2\;d\bold{y}\right)^{\frac{1}{2}}\\[0.5em]
\leq&2\vert\phi_j(\bold{x}_i)\vert\;\Vert\nabla\phi_j\Vert_{\mL^2}\left(\int_{\mathbb{R}^{3}}\vert\Psi_H(\bold{x}_1,\cdots,\bold{x}_{i-1},\bold{y},\bold{x}_{i+1},\cdots,\bold{x}_N)\vert^2\;d\bold{y}\right)^{\frac{1}{2}},
    \end{align*}
    where the last step follows again from the Hardy inequality \eqref{eq:Hardy_1}. 
    
    Consequently, for any $1\leq i,j\leq N$, all $\Psi_H \in \widehat{\mathcal{H}}^1$ and $\Psi_L \in \widehat{\mathcal{L}}^2$, and a.e. $(\bold{x}_1, \ldots, \bold{x}_{i-1}, \bold{x}_{i+1}, \ldots, \bold{x}_N) \in \mathbb{R}^{3N-3}$, we can use Fubini's theorem and the Cauchy-Schwarz inequality to deduce that
    \begin{align*}
        &\int_{\mathbb{R}^{3}}\int_{\mathbb{R}^{3}}\frac{\phi_j(\bold{x}_i)\phi_j(\bold{y})}{\vert \bold{x}_i- \bold{y}\vert} \Psi_H(\bold{x}_1,\cdots,\bold{x}_{i-1},\bold{y},\bold{x}_{i+1},\cdots,\bold{x}_N)\; d\bold{y}\Psi_L(\bold{x}_1,\cdots,\bold{x}_N)\;d\bold{x}_i\\[0.5em]    \leq&\int_{\mathbb{R}^{3}}2\vert\phi_j(\bold{x}_i)\vert\;\Vert\nabla\phi_j\Vert_{\mL^2}\left(\int_{\mathbb{R}^{3}}\vert\Psi_H(\bold{x}_1,\cdots,\bold{x}_{i-1},\bold{y},\bold{x}_{i+1},\cdots,\bold{x}_N)\vert^2\;d\bold{y}\right)^{\frac{1}{2}}\vert\Psi_L(\bold{x}_1,\cdots,\bold{x}_N)\vert\;d\bold{x}_i\\[0.5em]
        \leq&2\Vert\nabla\phi_j\Vert_{\mL^2}\left(\int_{\mathbb{R}^{3}}\vert\Psi_H(\bold{x}_1,\cdots,\bold{x}_{i-1},\bold{y},\bold{x}_{i+1},\cdots,\bold{x}_N)\vert^2\;d\bold{y}\right)^{\frac{1}{2}}\\[0.5em]
        &\hspace{1.3cm}\cdot\left(\int_{\mathbb{R}^{3}}\vert\phi_j(\bold{x}_i)\vert^2\;d\bold{x}_i\right)^{\frac{1}{2}}\left(\int_{\mathbb{R}^{3}}\vert\Psi_L(\bold{x}_1,\cdots,\bold{x}_N)\vert^2\;d\bold{x}_i\right)^{\frac{1}{2}}\\[0.5em]
        =&2\Vert\nabla\phi_j\Vert_{\mL^2}\left(\int_{\mathbb{R}^{3}}\vert\Psi_H(\bold{x}_1,\cdots,\bold{x}_{i-1},\bold{y},\bold{x}_{i+1},\cdots,\bold{x}_N)\vert^2\;d\bold{y}\right)^{\frac{1}{2}}\left(\int_{\mathbb{R}^{3}}\vert\Psi_L(\bold{x}_1,\cdots,\bold{x}_N)\vert^2\;d\bold{x}_i\right)^{\frac{1}{2}},
    \end{align*}
    where the last simplification results from the fact that $\Vert \phi_j \Vert_{\mL^2}=1$ for all $j \in \{1, \ldots, N\}$ by assumption. Using again the shorthand notation \[d\bold{x}^{i}_r:=d\bold{x}_1\cdots d\bold{x}_{i-1}d\bold{x}_{i+1}\cdots d\bold{x}_N,\] together with Fubini's theorem and the Cauchy-Schwarz inequality, we can conclude that for any $1\leq j< i\leq N$, all $\Psi_H \in \widehat{\mathcal{H}}^1$ and $\Psi_L \in \widehat{\mathcal{L}}^2$ it holds that
    \begin{align*}
        &\int_{\mathbb{R}^{3(N-1)}}\int_{\mathbb{R}^{3}}\int_{\mathbb{R}^{3}}\frac{\phi_j(\bold{x}_i)\phi_j(\bold{y})}{\vert \bold{x}_i- \bold{y}\vert} \Psi_H(\bold{x}_1,\cdots,\bold{x}_{i-1},\bold{y},\bold{x}_{i+1},\cdots,\bold{x}_N)\; d\bold{y}\Psi_L(\bold{x}_1,\cdots,\bold{x}_N)\;d\bold{x}_i d\bold{x}^i_r\\[0.5em]
        \leq&2\Vert\nabla\phi_j\Vert_{\mL^2}\int_{\mathbb{R}^{3(N-1)}}\left(\int_{\mathbb{R}^{3}}\vert\Psi_H(\bold{x}_1,\cdots,\bold{x}_{i-1},\bold{y},\bold{x}_{i+1},\cdots,\bold{x}_N)\vert^2\;d\bold{y}\right)^{\frac{1}{2}}\\[0.5em]
        &\cdot \hspace{2.6cm}\left(\int_{\mathbb{R}^{3}}\vert\Psi_L(\bold{x}_1,\cdots,\bold{x}_N)\vert^2\;d\bold{x}_i\right)^{\frac{1}{2}}d\bold{x}^i_r\\
        \leq&2\Vert\nabla\phi_j\Vert_{\mL^2}\left(\int_{\mathbb{R}^{3(N-1)}}\int_{\mathbb{R}^{3}}\vert\Psi_H(\bold{x}_1,\cdots,\bold{x}_N)\vert^2\;d\bold{x}_id\bold{x}^i_r\right)^{\frac{1}{2}}\left(\int_{\mathbb{R}^{3(N-1)}}\int_{\mathbb{R}^{3}}\vert\Psi_L(\bold{x}_1,\cdots,\bold{x}_N)\vert^2\;d\bold{x}_id\bold{x}^i_r\right)^{\frac{1}{2}}\\[0.5em]
 =&2\Vert\nabla\phi_j\Vert_{\mL^2}\;\Vert\Psi_H\Vert_{\widehat{\mathcal{L}}^2}\;\Vert\Psi_L\Vert_{\widehat{\mathcal{L}}^2}.
    \end{align*}
   
   Recalling now the definition of operator $\mathcal{U}_3$, we see that for all $\Psi_H \in \widehat{\mathcal{H}}^1$ and $\Psi_L \in \widehat{\mathcal{L}}^2$ it holds that
    \begin{equation}\label{eq:U_3}
 \left(\mathcal{U_{3}}\Psi_H, \Psi_L\right)_{\widehat{\mathcal{L}}^2}\leq 2N\;\sum_{j=1}^N\Vert\nabla\phi_j\Vert_{\mL^2}\;\Vert\Psi_H\Vert_{\widehat{\mathcal{L}}^2}\;\Vert\Psi_L\Vert_{\widehat{\mathcal{L}}^2},
    \end{equation}
which demonstrates that the bound \eqref{eq:fluc_parts_L2} also holds for the operator $\mathcal{U}_3$. 

Combining now Inequalities \eqref{eq:U_1}, \eqref{eq:U_2} and \eqref{eq:U_3} yields the continuity of the fluctuation potential $\mathcal{U}_{\Phi}$ as a mapping from $\widehat{\mathcal{H}}^1$ to $\widehat{\mathcal{L}}^2$ with a continuity constant that depends only on $N$ and $\sum_{j=1}^N\Vert\nabla\phi_j\Vert_{\mL^2}$. Let us remark here that since $\Phi = \frac{1}{\sqrt{N!}} \text{\rm det}\; \big(\phi_{i}(\bold{x}_j)\big)_{i, j=1}^N$, it is easy to show that $\sum_{j=1}^N\Vert\nabla\phi_j\Vert_{\mL^2} \lesssim ~\Vert \Phi \Vert_{\widehat{\mathcal{H}}^1}$.    
\end{proof}

\section{Approximation Spaces Satisfying Structure Assumptions B.I and B.II}\label{sec:spaces}~

The goal of this section is to state some prototypical examples of coupled cluster approximation spaces that satisfy \textbf{Structure Assumptions B.I and B.II}. Throughout this section, we assume the settings of Sections~\ref{sec:3}~and~\ref{sec:4} and we recall the Hartree-Fock method introduced in Appendix \ref{sec:HF}.

We begin by taking an $\mL^2(\R^3)$-orthonormal basis of $\mH^1(\R^3)$ which we denote as $\{\phi_{j}\}_{j \in \mathbb{N}}$, and we define, for every $K\geq N$ the set $\mathcal{B}_K:= \left\{\phi_j \right\}_{j=1}^K$ and the space $\mX_K := \text{span } \mathcal{B}_K$.  

Equipped with the sequence of spaces $\{\mX_K\}_{K\geq N}$ we define, for each $K\geq N$, the Hartree-Fock finite-dimensional minimisation set $\mS^K_{\rm HF}$ according to Definition \ref{def:HF_space_finite}, and we consider the following finite-dimensional Hartree-Fock eigenvalue problem (c.f., Equation \eqref{eq:HF_eigenvalue_approx}):

Find ${\Psi_K}=\frac{1}{\sqrt{N!}} \text{\rm det}\; \big({\psi}^K_{i}(\bold{x}_j)\big)_{i, j=1}^N  \in~{\mS}^K_{\rm HF}\subset \widehat{\mathcal{H}}^1$ and $\{\lambda^K_i\}_{i=1}^N \subset \mathbb{R}$ with the property that for all $\phi \in \mX_K$ and all $j=1, \ldots, N$ it holds that
	\begin{equation}\label{eq:HF_eigenvalue_approx*}
		\left\langle {\rm F}_{{\Psi}_K} {\psi}^K_{j}, \phi\right\rangle_{\mH^{-1}(\R^3) \times \mH^{1}(\R^3)} = \lambda^K_j \left({\psi}^K_j, \phi\right)_{\mL^2(\R^3)},
	\end{equation}
where $\mF_{{\Psi}_K}\colon \mH^1(\mathbb{R}^3)\rightarrow \mH^{-1}(\mathbb{R})$ denotes the single-particle Fock operator defined through Definition~\ref{def:Fock_single}.

We emphasise that, as described in Appendix \ref{sec:HF}, for each $K\geq N$, the eigenfunctions $\{{\psi}^K_j\}_{j=1}^K$ of the restricted single-particle Fock operator $\mF_{\Psi_K}\colon \mX_K \rightarrow \mX_K^*$ form an $\mL^2$-orthonormal basis of $\mX_K$, and the~$N$ solution eigenfunctions of Equation \eqref{eq:HF_eigenvalue_approx*} are chosen according to the Aufbau principle, i.e., as the ones with the lowest eigenvalues.

Next, we introduce, for every $K\geq N$, the index set $\mathcal{J}_{K}^N \subset \{1, \ldots, K\}^N$ given by
\begin{align*}
	\mathcal{J}_{K}^N := \Big\{\alpha = (\alpha_1, \alpha_2, \ldots, \alpha_N)\in \{1, \ldots, K\}^N \colon \alpha_1 < \alpha_2 < \ldots < \alpha_N\Big\}.
\end{align*}

We now have the following definition.
\begin{definition}[Finite Dimensional $N$-Particle Basis]\label{def:n_basis}~
	
	We define the $\widehat{\mathcal{L}}^2$-orthonormal, ${K}\choose{N}$-dimensional $N$-particle basis $ \mathcal{B}_K^{N} \subset \widehat{\mathcal{H}}^1$ as
	\begin{align}\label{eq:n_particle}
		\mathcal{B}^{N}_K := \left\{ \Psi_{\alpha}(\bold{x}_1, \bold{x}_2, \ldots, \bold{x}_N)=\frac{1}{\sqrt{N!}}\text{\rm det} \big({\psi}^K_{\alpha_i}(\bold{x}_j)\big)_{i, j=1}^N \colon \hspace{1mm} \alpha=(\alpha_1, \alpha_2, \ldots, \alpha_N) \in \mathcal{J}_{K}^N \right\}.
	\end{align} 
 
	Additionally, we define the subspace spanned by this basis set as $\mathcal{Z}_K := \text{\rm span}\; \mathcal{B}^N_K$. 
\end{definition}

The so-called Full-CC discretisations based on canonical Hartree-Fock orbitals now consist of defining, for each $K\geq N$, 
\begin{align}\nonumber
    \Psi_{0, K}:=& \frac{1}{\sqrt{N!}} \text{\rm det}\; \big({\psi}^K_{j}(\bold{x}_i)\big)_{i, j=1}^N  \quad \text{ with } \{\psi^K_j\}_{j=1}^N ~\text{denoting the Aufbau solutions to Equation \eqref{eq:HF_eigenvalue_approx*}},\\[0.5em] \label{eq:full-cc}
    \widetilde{\mathcal{V}}_K:=& \{\Psi_{0, K}\}^{\perp}\cap \mathcal{Z}_K, \qquad \text{and}\qquad \mathcal{V}_K := \widetilde{\mathcal{V}}_K \oplus \text{span}\; \{\Psi_{0, K}\}.
\end{align}
A direct calculation shows (c.f., \cite[Remark 36]{Hassan_CC}) that the Full-CC approximation spaces $\{\mathcal{V}_K\}_{K\geq N}$ defined in the above manner indeed satisfy \textbf{Structure Assumption B.I}. Let us remark here that the above construction of Full-CC approximation spaces satisfies \textbf{Structure Assumption B.I}, even if the canonical Hartree-Fock orbitals $\{\psi^K_{j}\}_{j=1}^K$ are rotated through a unitary transformation-- as long as the unitary transformation preserves the $\mL^2(\mathbb{R}^3)$ orthogonality of the set of so-called occupied orbitals $\{\psi^K_{j}\}_{j=1}^N$ and the set of so-called virtual orbitals $\{\psi^K_{j}\}_{j={N+1}}^K$.

We now turn our attention to coupled cluster approximation spaces that satisfy \textbf{Structure Assumption B.II}. To this end, we recall again the sequence of $N$-particle basis sets $\{\mathcal{B}_K^N\}_{K\geq N}$ and the sequence of finite-dimensional spaces $\{\mathcal{Z}_K\}_{K\geq N}$ defined through Definition \ref{def:n_basis}. Next, we define a so-called excitation index set and appropriate excitation operators to characterise the Slater determinants in basis set $\mathcal{B}_K^N$

\begin{definition}[Excitation Index Sets For Finite Bases]\label{def:Excitation_Index_finite}~
	
	For each $K\geq N$ and each $Q \in \{1, \ldots, N\}$ we define the index set $\mathcal{I}_Q^K$ as
	\begin{align*}
		\mathcal{I}_Q^K := \left\{ {{a_1, \ldots, a_Q}\choose{\ell_1, \ldots, \ell_Q}} \colon \ell_1 < \ldots< \ell_Q \in \{1, \ldots, N\} \text{ and } a_1< \ldots < a_Q  \in \{N+1, \ldots, K\} \right\}.
	\end{align*}
\end{definition}

\begin{definition}[Excitation Operators for Finite Bases]\label{def:Excitation_Operator_finite}~
	
 Let $K\geq N$, let $Q \in \{1, \ldots, N\}$, let $\mu \in \mathcal{I}_Q^K$ be of the form
	\begin{align*}
		\mu={{a_1, \ldots, a_Q}\choose{\ell_1, \ldots, \ell_Q}} \colon \ell_1 < \ldots< \ell_Q \in \{1, \ldots, N\} \text{ and } a_1< \ldots < a_Q  \in \{N+1, \ldots, K\},
	\end{align*}
and let the $N$-particle basis set $\mathcal{B}_K^N$ and finite-dimensional space $\mathcal{Z}_{K}$ be defined as in Definition \ref{def:n_basis}.

We define the excitation operator $\mathcal{X}_{K, \mu} \colon \mathcal{Z}_{K} \rightarrow \mathcal{Z}_{K}$ through its action on the $N$-particle basis set $\mathcal{B}_K^N$: For $\Psi_{\nu}(\bold{x}_1, \ldots, \bold{x}_N) =  \frac{1}{\sqrt{N!}} \text{\rm det}\; \big(\psi^K_{\nu_j}(\bold{x}_i)\big)_{ i, j=1}^N \in \mathcal{B}_K^N$, we set 
	\begin{align*}
		\mathcal{X}_{K, \mu} \Psi_{\nu} = \begin{cases}
			0  &\quad \text{ if } \{\ell_1, \ldots, \ell_Q\} \not \subset \{\nu_1, \ldots, \nu_N\},\\
			0   & \quad \text{ if } \exists  a_{m} \in  \{a_1, \ldots, a_Q\}  \text{ such that } a_m \in \{\nu_1, \ldots, \nu_N\},\\
			\Psi_{\nu, a} \in \mathcal{B}_{\wedge} & \quad \text{ otherwise},
		\end{cases}
	\end{align*}
	where the determinant $\Psi_{\nu, a} $ is constructed from $\Psi_{\nu}$ by replacing all functions $\psi^K_{\ell_1}, \ldots, \psi^K_{\ell_Q}$ used to construct $\Psi_{\nu}$ with functions $\psi^K_{a_1},\ldots \psi^K_{a_Q} $ respectively.
\end{definition}

Equipped with Definitions \ref{def:Excitation_Index_finite}~ and \ref{def:Excitation_Operator_finite}~, we are now ready to define the class of coupled cluster equations that can reasonably be expected to satisfy \textbf{Structure Assumption B.II}: The so-called truncated-CC discretisations of excitation rank $Q$, based on canonical Hartree-Fock orbitals consist of defining, for each $K\geq N$ and some given $Q \in \{1, \ldots, N\}$ 
\begin{align}\nonumber
    \Psi_{0, K}:=& \frac{1}{\sqrt{N!}} \text{\rm det}\; \big({\psi}^K_{j}(\bold{x}_i)\big)_{i, j=1}^N  \quad \text{ with } \{\psi_j^K\}_{j=1}^N ~\text{denoting the Aufbau solutions to Equation \eqref{eq:HF_eigenvalue_approx*}},\\[0.5em] \label{eq:truncated-cc}
    \widetilde{\mathcal{V}}_K:=& \text{span}\;\big\{\Psi_{\nu} \in \mathcal{B}_K^N\colon ~ \exists j\in \{1, \ldots, Q\}, ~\mu \in \mathcal{I}_j^K \text{ such that } \mathcal{X}_{K, \mu} \Psi_{0, K} =\Psi_{\nu}\big\},\\[0.5em] \nonumber
    \mathcal{V}_K := &\widetilde{\mathcal{V}}_K \oplus \text{span}\{\Psi_{0, K}\}, \qquad \mathcal{W}_K:= \mathcal{Z}_K, \quad \mathcal{F}_K := \begin{cases} \mathcal{F}_{\Psi_{0, K}} \quad &\text{on}~ \mathcal{H}^{1, \perp}_{\Psi_{0, K}}\\[0.5em]
    (\sum_{j=1}^N \lambda^K_j+1){\rm I} \quad &\text{on} ~\text{span}\{\Psi_{0, K}\}~\end{cases}, \quad \text{and} \\
    \Lambda_0^{(K)} :=& \sum_{j=1}^N \lambda^K_j,\nonumber
\end{align}
where $\mathcal{F}_{\Psi_{0, K}}\colon \widehat{\mathcal{H}}^1 \rightarrow \widehat{\mathcal{H}}^{-1}$ is the $N$-particle Hartree-Fock operator constructed from the reference determinant $\Psi_{0, K}$ as defined in Definition \ref{def:Fock_N} and $\{\lambda^K_j\}_{j=1}^N$ denote the $N$ lowest eigenvalues of the finite-dimensional Hartree-Fock eigenvalue problem \eqref{eq:HF_eigenvalue_approx*} in the space $\mX_K$. Let us remark, by way of examples, that the classical CCSD method corresponds to taking $Q=2$, the CCSDT method corresponds to taking $Q=3$ and so on.

As before, a direct calculation confirms that the truncated-CC approximation spaces $\{\mathcal{V}_K\}_{K\geq N}$ defined in the above manner satisfy the first two conditions of \textbf{Structure Assumption B.II}, and this is the case even if the canonical Hartree-Fock orbitals $\{\psi^K_{j}\}_{j=1}^K$ are rotated through a unitary transformation-- as long as the unitary transformation preserves the $\mL^2(\mathbb{R}^3)$ orthogonality of the set of so-called occupied orbitals $\{\psi^K_{j}\}_{j=1}^N$ and the set of so-called virtual orbitals $\{\psi^K_{j}\}_{j={N=1}}^K$.

Coming now to the third condition of \textbf{Structure Assumption B.II}, we see that for every $K\geq N$, the mean-field operator $\mathcal{F}_K$ is defined as $\mathcal{F}_K:= \mathcal{F}_{\Psi_{0, K}}$ on $\mathcal{H}^{1, \perp}_{\Psi_{0, K}}$. Recall now that we have the decomposition of the electronic Hamiltonian $H:= \mathcal{F}_{\Psi_{0, K}} + \mathcal{U}_{\Psi_{0, K}}$, and we have shown in Lemma~\ref{lem:Fluctuation} that $\mathcal{U}_{\Psi_{0, K}} \colon \widehat{\mathcal{H}}^{1} \rightarrow \widehat{\mathcal{L}}^2$ is a bounded operator with continuity constant depending only on $N$ and $\Psi_{0, K}$. Consequently, under the assumption that the $N$ discrete Aufbau eigenfunctions and corresponding Aufbau eigenvalues of the sequence of finite-dimensional Hartree-Fock eigenvalue problems \eqref{eq:HF_eigenvalue_approx*} converge in the complete basis set limit $K \to \infty$, we have uniform boundedness of the continuity and coercivity constants of the sequence of mean-field operators $\{\mathcal{F}_K\}_{K\geq N}$. Thus the third condition of \textbf{Structure Assumption B.II} will be satisfied. 

Let us incidentally remark that while the convergence of the discrete Aufbau Hartree-Fock eigenpairs in the complete basis set limit is a reasonable assumption, to the best of our knowledge, no formal proof of this fact or precise conditions under which it holds, have been stated in the literature. Note that this convergence assumption will also guarantee that the prior \textbf{Assumption A.II} concerning the uniform boundedness of the reference determinants $\{\Psi_{0, K}\}_{K\geq N}$ in $\widehat{\mathcal{H}}^1$.

Finally, let us consider the fourth and final condition of \textbf{Structure Assumption B.II}. We have already demonstrated in Lemma \ref{lem:Fluctuation} that $\mathcal{U}_{\Psi_{0, K}} \colon \widehat{\mathcal{H}}^{1} \rightarrow \widehat{\mathcal{L}}^2$ is a bounded operator with continuity constant depending only on $N$ and $\Psi_{0, K}$, and thus, under the same assumption that the discrete Aufbau Hartree-Fock eigenpairs converge in the complete basis set limit $K\to \infty$, we have uniform boundedness of the continuity constants of the sequence of operators $\{\mathcal{U}_K\}_{K\geq N}$ with each $\mathcal{U}_K:= \mathcal{U}_{\Psi_{0, K}}$. It remains to study  the additional smallness assumption that we impose on the continuity constant of each $\mathcal{U}_K$, namely (see also, Inequality \eqref{eq:lemma_inf-sup_13}),
\begin{align}\label{eq:smallness}
     \left \Vert \mathbb{P}^{\perp}_{\mathcal{F}_K}\widetilde{\mathcal{U}}_K \mathbb{P}_{\mathcal{F}_K}\right\Vert_{\mathcal{F}_K \to \widehat{\mathcal{L}}^2}< \frac{1}{2}\frac{\sqrt{\Lambda_{\min}^{\mathcal{F}_K}}\;\Gamma^{*}_{\rm GS}}{\beta_K(\Theta_{K, \rm GS}^*)},
\end{align}
where
\begin{align*}
  \Lambda_{\min}^{\mathcal{F}_K}:=& \min_{\substack{\Phi_K \in \mathcal{W}_K\\ \Phi_K \in \mathcal{V}_K^{\perp}}} \frac{\langle \Phi_K, (\mathcal{F}_K -\Lambda_0^{(K)}) \Phi_K\rangle_{\widehat{\mathcal{H}}^1 \times \widehat{\mathcal{H}}^{-1}} }{\Vert \Phi_K \Vert^2_{\widehat{\mathcal{L}}^2}},\\[0.5em]
  \Gamma^{*}_{\rm GS}:= &\min_{\Phi_K \in \{\Psi_{{\rm GS}, K}^*\}^{\perp}\cap\mathcal{W}_K}  \frac{\langle \Phi_K, (H-\mathcal{E}_{\rm GS^*}) \Phi_K\rangle_{\widehat{\mathcal{H}}^1 \times \widehat{\mathcal{H}}^{-1}}}{\Vert \Phi_K \Vert_{\mathcal{F}_K}^2}, \text{and}\\[1em]
  \beta_K(\Theta_{K, \rm GS}^*):=& \left \Vert \mathbb{P}^{\perp}_{\mathcal{F}_K}{e^{\mathcal{T}(\Theta^{\Pi}_{K, \rm GS})}}\mathbb{P}_{\mathcal{F}_K}{e^{-\mathcal{T}(\Theta^{\Pi}_{K, \rm GS})}} \mathbb{P}_{\mathcal{F}_K}\right\Vert_{\mathcal{F}_K \to \mathcal{F}_K}.
\end{align*}

As stated earlier in Remark \ref{rem:B.II(3)}, the validity of Inequality \eqref{eq:smallness} is, in general, problem dependent and not a universal property that holds for arbitrary molecules. Indeed, it can readily be seen that the constant $\Gamma^{*}_{\rm GS}$ depends on the spectral gap of the electronic Hamiltonian $H$ that describes the molecule under study. Similarly, the constant $\beta_K(\Theta_{K, \rm GS}^*)$ can be shown to be $\mathcal{O}\big(\Vert\mathcal{T}(\Theta^*_{K, \rm GS})\Vert_{\widehat{\mathcal{H}}^1 \to \widehat{\mathcal{H}}^1}\big)$ for $\Vert\mathcal{T}(\Theta^*_{K, \rm GS})\Vert_{\widehat{\mathcal{H}}^1 \to \widehat{\mathcal{H}}^1} < 1$, and therefore reflects how well the reference determinant $\Psi_{0, K}$ approximates the sought-after ground state eigenfunction $\Psi_{{\rm GS}, K}^*$ of the electronic Hamiltonian $H$. Finally, as is well known, the single particle Hartree-Fock operator defined through Definition \ref{def:Fock_single} possesses an essential spectrum consisting of $[0, \infty)$ and therefore the constant $\Lambda_{\min}^{\mathcal{F}_K}$, will \underline{not} diverge to $+\infty$ in the complete basis set limit $K\to \infty$.

Having said this, in the specific case of the truncated-CC approximation spaces constructed as in Equation \eqref{eq:truncated-cc}, the constant $\Lambda_{\min}^{\mathcal{F}_K}$ can reasonably be expected to be fairly large in comparison to the operator norm $\big \Vert \mathbb{P}^{\perp}_{\mathcal{F}_K}\widetilde{\mathcal{U}}_K \mathbb{P}_{\mathcal{F}_K}\big\Vert_{\mathcal{F}_K \to \widehat{\mathcal{L}}^2}$, at least for large excitation truncation ranks $Q$. Indeed, it follows directly from the definition of $\Lambda_{\min}^{\mathcal{F}_K}$, the mean-field operator $\mathcal{F}_K$ and the excitation rank-truncated approximation space $\widetilde{\mathcal{V}}_K$ that for any $K\geq N+Q+1$, it holds that
\begin{align}\label{eq:Final_2}
    \Lambda_{\min}^{\mathcal{F}_K}:=& \min_{\substack{\Phi_K \in \mathcal{W}_K\\ \Phi_K \in \mathcal{V}_K^{\perp}}} \frac{\langle \Phi_K, (\mathcal{F}_K-\Lambda_0^{(K)}) \Phi_K\rangle_{\widehat{\mathcal{H}}^1 \times \widehat{\mathcal{H}}^{-1}} }{\Vert \Phi_K \Vert^2_{\widehat{\mathcal{L}}^2}} = \sum_{j=1}^{N-Q-1}\lambda_j^K+  \sum_{j=N+1}^{N+Q+1}\lambda_j^K  -\sum_{j=1}^{N} \lambda_j^K,
\end{align}
 where $\{\lambda^K_j\}_{j=1}^{N+Q+1}$ denote the $N+Q+1$ lowest eigenvalues of the finite-dimensional Hartree-Fock eigenvalue problem \eqref{eq:HF_eigenvalue_approx*} in the space $\mX_K$.

Equation \eqref{eq:Final_2} is simply a consequence of the fact the space $\mathcal{W}_K\cap \mathcal{V}_K^{\perp}$ is spanned by Slater determinants from $\mathcal{B}_K^N$ that can be written as excitations of the reference determinant $\Psi_{0, K}$ of rank at least $Q+1$. Since the Slater determinants in $\mathcal{B}_K^N$ are, additionally, discrete eigenfunctions of the $N$-particle Hartree-Fock operator $\mathcal{F}_{\Psi_{0, K}}$, and we have adopted the convention of ordering all eigenvalues in ascending order, it is relatively simple to see that the Slater determinant that achieves the minimum in the above equation must be given by
\begin{align*}
     \Psi_{K, \rm min}= \psi^K_1 \wedge \psi^K_2 \wedge \ldots \psi^K_{N-Q-1} \wedge \ldots \psi^K_{N+1} \wedge \psi^K_{N+2} \wedge \ldots\wedge \psi^K_{N+Q+1},
 \end{align*}
which immediately leads to Equation \eqref{eq:Final_2}. Notice that we can equivalently write Equation \eqref{eq:Final_2} as
\begin{align*}
\Lambda_{\min}^{\mathcal{F}_K}= \sum_{j=N+1}^{N+Q+1} \lambda_j^K -\sum_{j=N-Q}^N \lambda_j^K =\sum_{j=1}^{Q+1} \varepsilon^K_j,
 \end{align*}
where each $\varepsilon^K_j:= \lambda_{N+j}^K- \lambda_{N-j+1}^K $ is the $j^{\rm th}$ discrete excitation energy associated with the $N$-particle Hartree-Fock operator $\mathcal{F}_{\Psi_{0, K}}$. Note that if the discrete Hartree-Fock eigenvalues $\{\lambda_j^K\}_{j=1}^{2N}$ converge to some limiting eigenvalues $\{\lambda_j^0\}_{j=1}^{2N}$, then we can, of course, replace Estimate \eqref{eq:Final_2} with a $K$-independent bound.


We end this section with some numerical tests that explore the validity of Inequality \eqref{eq:smallness} in different regimes. We begin by considering the most common excitation rank-truncated coupled cluster discretisations, namely, the CCSD (truncation at rank two) and CCSDT (truncation at rank three) discretisations with canonical Hartree-Fock orbitals. Our results for a collection of small molecules at equilibrium geometries using the minimal STO-6G basis sets are displayed in Tables \ref{table:CCSD} and \ref{table:CCSDT}. For reference, we include the theoretical continuous inf-sup constants corresponding to Theorem \ref{thm:CC_der_inv}. We also include the numerically computed estimate of the discrete inf-sup constant given by Equation~\eqref{eq:inf-sup}, as well as the corresponding theoretical estimate resulting from our analysis (see the proof of Lemma~\ref{lem:inf-sup}), which is given by 
\begin{align*}
    \gamma_{\rm inf-sup}:= \dfrac{\Gamma_{\rm GS}^{\mathcal{F}_K} - \left \Vert \mathbb{P}^{\perp}_{\mathcal{F}_K}\widetilde{\mathcal{U}}_K \mathbb{P}_{\mathcal{F}_K}\right\Vert_{\mathcal{F}_K \to \widehat{\mathcal{L}}^2}\dfrac{\beta_K(\Theta_{K, \rm GS}^*)}{\sqrt{\Lambda_{\min}^{\mathcal{F}_K}}}- \Vert \left(H-\mathcal{E}^*_{\rm GS}\right)e^{\mathcal{T}(\Theta^{\Pi}_{K, \rm GS})}\Psi_{0, K}\Vert_{\mathcal{F}_K^{-1}}}{\left \Vert \mathbb{P}^{\perp}_{\mathcal{F}_K}{e^{-\mathcal{T}(\Theta^{\Pi}_{K, \rm GS})}}\right\Vert_{\mathcal{F}_K \to \mathcal{F}_K}\left \Vert e^{\mathcal{T}(\Theta^{\Pi}_{K, \rm GS})^{\dagger}}\right\Vert_{\mathcal{F}_K \to \mathcal{F}_K}}.
\end{align*}

\begin{table}[ht!]
	\centering
	\begin{tabular}{||c| c| c| c| c| c||} 
		\hline \hline
		\shortstack{Molecule\\ \hphantom{Molecule}\\ \hphantom{Molecule} \\ \hphantom{Molecule}} & \shortstack{$\left \Vert \mathbb{P}^{\perp}_{\mathcal{F}_K}\widetilde{\mathcal{U}}_K \mathbb{P}_{\mathcal{F}_K}\right\Vert_{\mathcal{F}_K \to \widehat{\mathcal{L}}^2}$\\ \hphantom{$\left \Vert \mathbb{P}^{\perp}_{\mathcal{F}_K}\widetilde{\mathcal{U}}_K \mathbb{P}_{\mathcal{F}_K}\right\Vert_{\mathcal{F}_K \to \widehat{\mathcal{L}}^2}$}} &  \shortstack{$\frac{\sqrt{\Lambda_{\min}^{\mathcal{F}_K}}\;\Gamma^{*}_{\rm GS}}{\beta_K(\Theta_{K, \rm GS}^*)}$\\ \hphantom{$\frac{\sqrt{\Lambda_{\min}^{\mathcal{F}_K}}\;\Gamma^{*}_{\rm GS}}{\beta_K(\Theta_{K, \rm GS}^*)}$}} &   \shortstack{Continuous\\ inf-sup\\ constant\\ $\Lambda^*/\beta$} & \shortstack{Discrete\\ inf-sup\\ constant\\ $\gamma_{\rm inf-sup}$} & \shortstack{$\Vert \mD\mathcal{f}_K^{-1}(\Theta^{\Pi}_{K, \rm GS})\Vert^{-1}_{\mathcal{F}_K^{-1} \to \mathcal{F}_K}$\\ with $\mD\mathcal{f}_K(\Theta^{\Pi}_{K, \rm GS})$\\ viewed as a mapping\\ from $\widetilde{\mathcal{V}}_K$ to $\widetilde{\mathcal{V}}_K^*$}\\ [0.5ex]  
		\hline\hline
		{\rm BeH$_2$} & 0.2508 &  1.9807 &0.2568 & 0.2532&0.3592\\ 
		{\rm BH$_3$} & 0.3056  & 1.5447 &  0.2081&0.2064&0.3254\\
		\rm{HF} & 0.5010   &2.4073 &0.2529&0.2016&0.2993\\
		\rm{H$_2$O} & 0.3067  &  2.2724& 0.2789&0.2652&0.3646\\
		\rm{LiH} & 0.1878   & 2.4044 &0.2164&0.1953&0.2630\\
		\rm{NH$_3$} & 0.3721   & 2.0420  &0.2784&0.2732&0.4302\\ 
		\hline\hline
	\end{tabular}
	\vspace{2mm}
	\caption{Numerical results for the \textbf{CCSD} scheme (excitation rank-truncation of order \textbf{two}). The calculations were performed in STO-6G basis sets with the exception of the HF and LiH molecules for which 6-31G basis sets were used. Since the exact coupled cluster ground state zero $\Theta^*_{K, \rm GS} \in \widehat{\mathcal{H}}^{1, \perp}_{\Psi_{0, K}}$ is unavailable, the Full-CC ground state zero in the corresponding basis set (STO-6G or 6-31G) was taken as a reference.}\label{table:CCSD}
\end{table}

\begin{table}[ht!]
	\centering
	\begin{tabular}{||c| c| c| c| c| c||} 
		\hline \hline
		\shortstack{Molecule\\ \hphantom{Molecule}\\ \hphantom{Molecule}\\ \hphantom{Molecule}} &  \shortstack{$\left \Vert \mathbb{P}^{\perp}_{\mathcal{F}_K}\widetilde{\mathcal{U}}_K \mathbb{P}_{\mathcal{F}_K}\right\Vert_{\mathcal{F}_K \to \widehat{\mathcal{L}}^2}$\\ \hphantom{$\left \Vert \mathbb{P}^{\perp}_{\mathcal{F}_K}\widetilde{\mathcal{U}}_K \mathbb{P}_{\mathcal{F}_K}\right\Vert_{\mathcal{F}_K \to \widehat{\mathcal{L}}^2}$}}&  \shortstack{$\frac{\sqrt{\Lambda_{\min}^{\mathcal{F}_K}}\;\Gamma^{*}_{\rm GS}}{\beta_K(\Theta_{K, \rm GS}^*)}$\\ \hphantom{$\frac{\sqrt{\Lambda_{\min}^{\mathcal{F}_K}}\;\Gamma^{*}_{\rm GS}}{\beta_K(\Theta_{K, \rm GS}^*)}$}} &  \shortstack{Continuous\\ inf-sup\\ constant\\ $\Lambda^*/\beta$}  & \shortstack{Discrete\\ inf-sup\\ constant\\ $\gamma_{\rm inf-sup}$} &\shortstack{$\Vert \mD\mathcal{f}_K^{-1}(\Theta^{\Pi}_{K, \rm GS})\Vert^{-1}_{\mathcal{F}_K^{-1} \to \mathcal{F}_K}$\\ with $\mD\mathcal{f}_K(\Theta^{\Pi}_{K, \rm GS})$\\ viewed as a mapping\\ from $\widetilde{\mathcal{V}}_K$ to $\widetilde{\mathcal{V}}_K^*$}\\ [0.5ex] 
		\hline\hline
		{\rm BeH$_2$} & 0.1835 &  2.2865 &0.2568 &0.2321&0.3403\\ 
		{\rm BH$_3$} & 0.2581   & 2.1659&  0.2081&  0.1752&0.3081\\
		\rm{HF} & 0.4903   &3.5897&0.2529&0.2187&0.2995\\
		\rm{H$_2$O} & 0.2431   & 3.0995&0.2789&0.2504&0.3592\\
		\rm{LiH} & 0.1629  &  3.3790 &0.2164&0.2038&0.2628\\
		\rm{NH$_3$} &  0.3038   & 2.6085  &0.2784&0.2338&0.4147\\ 
		\hline\hline
	\end{tabular}
	\vspace{2mm}
	\caption{Numerical results for the \textbf{CCSDT} scheme (excitation rank-truncation of order \textbf{three}). As before, the calculations were performed in STO-6G basis sets with the exception of the HF and LiH molecules for which 6-31G basis sets were used, and the Full-CC ground state zero in the corresponding basis set (STO-6G or 6-31G) was taken as a reference.}\label{table:CCSDT}
\end{table}

This first set of molecules constitutes an example of `well-behaved' molecules, i.e., molecules for which the basic CCSD scheme work well, attaining in all cases the so-called chemical accuracy threshold of an error less than $10^{-3}$ Hartree. As a more stringent test case, we consider the nitrogen dimer and the carbon monoxide molecule at equilibrium geometry, these being examples of molecules for which the CCSD scheme is unable to achieve a chemically accurate solution\footnote{In fact, the CCSD scheme applied to these two molecules is an order of magnitude less accurate compared to the first set of molecules.}. The corresponding results, using once again minimal STO-6G basis sets, are displayed in Table \ref{table:3}. In the case of the carbon monoxide molecule, we see that Inequality \eqref{eq:smallness} is not satisfied for the CCSD and CCSDT discretisations. On the other hand, the higher excitation rank-truncation discretisations indeed satisfy Inequality \eqref{eq:smallness}.

	\begin{table}[ht!]
			\centering
			\begin{tabular}{|l||c||c||} 
				\hline 
				Molecule &N$_2$ & CO \\ [0.5ex] 
				\hline\hline
				CCSD Energy Error (milli-Hartree)& $4.0$ & $8.2$\\  \hline \hline
				\shortstack{Discrete inf-sup constant $\gamma_{\rm inf-sup}$\\at excitation truncation-rank two}& 		0.0004  &		{\color{red}-0.1074} \\ \hline \hline 
				\shortstack{Discrete inf-sup constant $\gamma_{\rm inf-sup}$\\at excitation truncation-rank three}& 		0.0402   & 		{{\color{red}-0.0426}}\\ \hline \hline 
				\shortstack{Discrete inf-sup constant $\gamma_{\rm inf-sup}$\\at excitation truncation-rank four}& 		 0.0908  & 		0.0225\\  \hline \hline 
				\shortstack{Discrete inf-sup constant $\gamma_{\rm inf-sup}$\\at excitation truncation-rank five}& 		{0.1364}   & 		0.0666\\ \hline \hline 
				Continuous inf-sup constant  $\Lambda^*/\beta$& 		{0.1614}   & 		{0.1255} \\[1ex] 
				\hline
			\end{tabular}
	\caption{Numerical results for the nitrogen dimer and the carbon monoxide molecule at equilibrium geometries. The calculations were performed in STO-6G basis sets, and as before, the Full-CC ground state zero in the STO-6G basis set was taken as a reference.}\label{table:3}
\end{table}

\end{document}